\documentclass{report}
\usepackage{amsmath}
\usepackage{amsthm, amssymb,setspace, stmaryrd}
\usepackage{baththesis}
\usepackage{hyperref}
\usepackage{enumerate}

\hypersetup{
    bookmarks=true,         % show bookmarks bar?
    unicode=false,          % non-Latin characters in Acrobat's bookmarks
    pdftoolbar=true,        % show Acrobat's toolbar?
    pdfmenubar=true,        % show Acrobat's menu?
    pdffitwindow=false,     % window fit to page when opened
    pdfstartview={FitH},    % fits the width of the page to the window
    pdfnewwindow=true,      % links in new window
    colorlinks=true,       % false: boxed links; true: coloured links
    linkcolor=black,          % colour of internal links
    citecolor=black,        % colour of links to bibliography
    urlcolor=black           % colour of external links
}

\newcommand{\R}{ \ensuremath{ \mathbb{R}}}
\newcommand{\cp}{\text{cp}}
\newcommand{\pd}[2]{\dfrac{\partial #1}{\partial #2}}

\theoremstyle{plain}
\newtheorem{thm}{Theorem}[section]
\newtheorem{cor}[thm]{Corollary}
\newtheorem{lem}[thm]{Lemma}

\newtheorem{eg}[thm]{Example}
\theoremstyle{definition}
\newtheorem{defn}[thm]{Definition}

\theoremstyle{remark}
\newtheorem{rem}[thm]{Remark}

\numberwithin{equation}{section}

\title{Bounding Betti Numbers of Sets Definable in O-Minimal Structures Over the Reals}
\author{Mahana Clutha}
\degree{Doctor of Philosophy}
\department{Department of Computer Science}
\degreemonthyear{January 2011}
\norestrictions

\begin{document}

\maketitle

\clearpage
\hspace{1cm}
\clearpage

\begin{abstract}
A bound for Betti numbers of sets definable in o-minimal structures is presented.

An axiomatic complexity measure is defined, allowing various concrete complexity measures for definable functions to be covered.  This includes common concrete measures such as the degree of polynomials, and complexity of Pfaffian functions.

A generalisation of the Thom-Milnor Bound \cite{Milnor, Thom} for sets defined by the conjunction of equations and non-strict inequalities is presented, in the new context of sets definable in o-minimal structures using the axiomatic complexity measure.  Next bounds are produced for sets defined by Boolean combinations of equations and inequalities, through firstly considering sets defined by sign conditions, then using this to produce results for closed sets, and then making use of a construction to approximate any set defined by a Boolean combination of equations and inequalities by a closed set.

 Lastly, existing results \cite{BNSASPSets} for sets defined using quantifiers on an open or closed set  are generalised, using a construction from Gabrielov and Vorobjov \cite{ApproxDefSetCompFam} to approximate any set by a compact set.  This results in a method to find a general bound  for any set definable in an o-minimal structure in terms of the axiomatic complexity measure.  As a consequence for the first time an upper bound for sub-Pfaffian sets defined by arbitrary formulae with quantifiers is given.  This bound is singly exponential if the number of quantifier alternations is fixed.
\end{abstract}

\clearpage
\hspace{1cm}
\clearpage

\chapter*{Acknowledgements}
\vspace{2cm}

\begin{center}
\noindent Thanks to Nicolai and Guy for their support and patience with an unusual student.

\vspace{1cm}
\noindent Thanks to the worldwide judo family for providing the perfect complement to a research degree.

\vspace{1cm}
\noindent Thanks to my real family -- the Cluthas, the Weirs and the Rigarlsfords.

\vspace{1cm}
\noindent Special thanks to Lance, Sally, Cassie and Tom for being my adoptive kiwi family so far from home.

\vspace{1cm}
\noindent Also special thanks to the Cousins for being my British adoptive family.

\vspace{1cm}
\noindent Faith -- words are inadequate to portray the value of your friendship.

\vspace{1cm}
\noindent For Isobel, Jane and Florence.  Three generations of exceptional women.

\end{center}

\clearpage
\hspace{1cm}
\clearpage

\tableofcontents

\clearpage
\hspace{1cm}
\clearpage

%%%%%%%%%%%%%%%%%%%%%%%%%%%%%%%%%%%%%%%%%%%%%%%%%%%%%%%%%%%%%%%%%%%%%%%%%%%%%%%%%%%%%%%%
%%%%%%%%%%%%%%%%%%%%%%%%%%%%%%%%%%%%%%%%%%%%%%%%%%%%%%%%%%%%%%%%%%%%%%%%%%%%%%%%%%%%%%%%
\chapter{Preliminaries}
%%%%%%%%%%%%%%%%%%%%%%%%%%%%%%%%%%%%%%%%%%%%%%%%%%%%%%%%%%%%%%%%%%%%%%%%%%%%%%%%%%%%%%%%
%%%%%%%%%%%%%%%%%%%%%%%%%%%%%%%%%%%%%%%%%%%%%%%%%%%%%%%%%%%%%%%%%%%%%%%%%%%%%%%%%%%%%%%%

%%%%%%%%%%%%%%%%%%%%%%%%%%%%%%%%%%%%%%%%%%%%%%%%%%%%%%%%%%%%%%%%%%%%%%%%%%%%%%%%%%%%%%%%%
\section{Introduction}
%%%%%%%%%%%%%%%%%%%%%%%%%%%%%%%%%%%%%%%%%%%%%%%%%%%%%%%%%%%%%%%%%%%%%%%%%%%%%%%%%%%%%%%%%

This thesis considers the problem of relating the complexity of a formula to the complexity of the ``shape'' it describes.  The natural questions arising from this initial statement of the problem are ``what type of formula?'', ``how do we measure the complexity of the formula?'' and ``how do we measure the complexity of the shape?''.

There are a variety of answers to the first two questions.  Existing work in this area mostly deals with polynomial functions, with degree as complexity measure.  Some work has been done with Pfaffian functions (for example \cite{Zell}), see Section \ref{pfaffian} for definitions and the relevant complexity measure.  This thesis, however, considers a more general type of formula, that is sets definable in  o-minimal structures over the reals (see Section \ref{o-min}).  There is not standard method of measuring the complexity of such formula, hence one of the main contributions of this work is to define an axiomatic complexity measure to solve this problem.

The final question has a standard answer -- we measure the complexity of ``shapes'', i.e. some sort of geometric complexity, using Betti numbers. These are defined in Section \ref{basicAT}.  In the following we frequently bound the sum of the Betti numbers of a given set, which obviously bounds each individual Betti number (as they are a finite sequence of integers).

%%%%%%%%%%%%%%%%%%%%%%%%%%%%%%%%%%%%%%%%%%%%%%%%%%%%%%%%%%%%%%%%%%%%%%%%%%%%%%%%%%%%%%%%
\section{Previous Work}
%%%%%%%%%%%%%%%%%%%%%%%%%%%%%%%%%%%%%%%%%%%%%%%%%%%%%%%%%%%%%%%%%%%%%%%%%%%%%%%%%%%%%%%%

In 1964 Milnor \cite{Milnor} considered a real algebraic set $S \subset \mathbb{R}^{n}$, defined by polynomial equations
\begin{equation*}
f_{1}(x_{1},\ldots,x_{m})=\ldots=f_{p}(x_{1},\ldots,x_{m})=0
\end{equation*}
of maximum degree $d$.  He showed the sum of the Betti numbers of $S$ is bounded by $d(2d-1)^{n-1}$. Milnor also showed that if the
$p$ equations above were instead non-strict inequalities, and they have total degree $d' =
\text{deg}(f_{1})+\hdots+\text{deg}(f_{p})$, then the Betti sum of the set defined by these inequalities is bounded by
$\frac{1}{2}(2+d')(1+d')^{n-1}$.
Various alternate versions of the above results have been published (see Petrovskii and Oleinik \cite{OlePetr}, Thom \cite{Thom}),
and modified, simpler proofs offered (for example Basu, Pollack, Roy \cite{AlgInRAG}).

In 1999 Basu \cite{BettiSA} showed that if a closed semi-algebraic set $S \subset \mathbb{R}^{n}$ is defined as the intersection of a
real algebraic set $Q=0$, where $Q$ is a polynomial with $\text{deg}(Q)\leq d$, whose real dimension is $n'$, with a set defined by a quantifier-free Boolean
formula with no negations with atoms of the form $P_{i}=0$, $P_{i} \geq 0$, $P_{i} \leq 0$, for $1 \leq i \leq s$, and the degree of
polynomials $P_{i}$ being at most $d$, then the sum of the Betti numbers of $S$ is bounded by $s^{n'}(O(d))^{n}$.  In 2004 Basu,
Pollack and Roy \cite{SignCon} refined this bound to
\begin{equation*}
\sum_{i=0}^{n'}\sum_{j=0}^{n'-i}\binom{s}{j}6^{j}d(2d-1)^{n-1}.
\end{equation*}
In the same paper a bound is also constructed for the sum of the $i$-th Betti numbers over the realisation of all sign conditions of
a set of polynomials on an algebraic set.

The preceding results all concern sets defined in terms of polynomials.  In 1999, Zell \cite{Zell} produced a similar result for
semi-Pfaffian sets, directly following the above methods.  Further to this, a survey of upper bounds on sets defined by Pfaffian and
Noetherian functions is presented in \cite{PffNoe}.

In 2005, Gabrielov and Vorobjov \cite{BNSASetsQFForm} described a way to replace any arbitrary semi-algebraic set, defined by a
arbitrary Boolean quantifier-free formula, by a compact set with coinciding Betti numbers, and presented a bound of $O(k^{2}d)^{n}$.
In the second edition of \cite{AlgInRAG} an alternative, simpler proof of this result is given, showing homotopy equivalence rather
than simply coinciding Betti numbers.

As the Tarski-Seidenberg Theorem (see for example \cite{SAG}) states that in the semi-algebraic case, every first-order formula is equivalent to a quantifier-free formula, the polynomial case is essentially complete, save refinements of bounds.  Sets defined in o-minimal structures however, do not generally have this property, for example see \cite{quantElim}.

Gabrielov, Vorobjov and Zell \cite{BNSASPSets} showed how to associate a spectral sequence with a surjective map, and used this to
find an upper bound for open and closed sets defined by first-order formulae.  In particular, results are presented for sub-Pfaffian sets
defined by first-order formulae.  This is one of the main ingredients of the last chapter of this thesis, and indeed one of the primary aims
of this work is to remove the restriction that requires  sets to be open or closed.

In 2009, Gabrielov and Vorobjov \cite{ApproxDefSetCompFam} presented a method for approximating any o-minimal set $S$ by a compact
o-minimal set $T$, with each Betti number of $S$ being bounded by that of $T$, and, in a particular case, these sets are shown to be
homotopy equivalent.  This result leads to a refinement of some of the above bounds in the semi-algebraic and semi- and sub- Pfaffian
cases, and, more importantly, presents a method that can be used for any o-minimal set.

%%%%%%%%%%%%%%%%%%%%%%%%%%%%%%%%%%%%%%%%%%%%%%%%%%%%%%%%%%%%%%%%%%%%%%%%%%%%%%%%%%%%%%%%
\section{Outline}
%%%%%%%%%%%%%%%%%%%%%%%%%%%%%%%%%%%%%%%%%%%%%%%%%%%%%%%%%%%%%%%%%%%%%%%%%%%%%%%%%%%%%%%%
Existing work suggests two areas for expansion.

Firstly, we seek to reproduce the above results, this time for functions defined in o-minimal structures instead of simply
polynomials or Pfaffian functions.  We start with sets defined by equalities, then closed sets, then sets defined by an arbitrary
Boolean combination, before finally considering sets defined using quantifiers.

The second new area derives naturally from this problem -- established complexity measures for polynomials (for example degree,
number of monomials, additive complexity) and Pfaffian functions exist, but there is no conventional way to measure the complexity of
a function defined in an o-minimal structure.  So, before proceeding with the above calculations, we create an axiomatic complexity
definition.  The definition presented in this work was inspired by \cite{RASASets}.

The remainder of this chapter consists of definitions and results from existing work that are used in our proofs, and concludes with
a presentation of our new axiomatic complexity definition.
In Chapter \ref{chComp} we start with the case of a set defined by multiple equalities, following the methods of Milnor, then move on
to sets defined by the conjunction of equations and non-strict inequalities.  Chapter \ref{chBool} initially shows how to bound the
Betti numbers of sign conditions, then uses this to bound the Betti numbers of closed sets.  Lastly a construction is presented to
approximate a set defined by any Boolean combination of equations and inequalities by a closed set, and this, in conjunction with the
previous result, gives  a bound of the Betti numbers of the original set.  The final chapter considers sets defined using
quantifiers, and is the most significant contribution of this thesis.  In \cite{BNSASPSets} a method is presented to bound the Betti
numbers of a surjective map, using a particular spectral sequence, and then this is used in the case of the projective map to bound
the Betti numbers of a open or closed set defined with quantifiers.  In \cite{ApproxDefSetCompFam}, a method is shown to approximate
any set with a compact set.  The main result of this thesis combines these two results, along with the preceding work, to give a
bound for any set defined in an o-minimal structure using quantifiers.
As a consequence for the first time an upper bound for sub-Pfaffian sets defined by arbitrary formulae with quantifiers is given.  This bound is singly exponential if the number of quantifier alternations is fixed.

Our methods are necessary for sub-Pfaffian sets, as there is no equivalent to Tarski-Seidenberg for polynomials.  As is made clear in Section \ref{pfaffian}, a restricted sub-Pfaffian set need not be semi-Pfaffian, and there is no quantifier elimination.  This is therefore also true for sets definable in o-minimal structures over the reals (as this includes sets defined by Pfaffian functions).

%%%%%%%%%%%%%%%%%%%%%%%%%%%%%%%%%%%%%%%%%%%%%%%%%%%%%%%%%%%%%%%%%%%%%%%%%%%%%%%%%%%%%%%%
\section{Main Results}
%%%%%%%%%%%%%%%%%%%%%%%%%%%%%%%%%%%%%%%%%%%%%%%%%%%%%%%%%%%%%%%%%%%%%%%%%%%%%%%%%%%%%%%%
This thesis has three main contributions to knowledge:
\begin{itemize}
\item A new axiomatic complexity metric is defined in Section \ref{complexSec}, which improves on the presentation in \cite{RASASets}.

\item Know quantitative bounds in polynomial and Pfaffian settings are generalized to our new setting of sets definable in o-minimal structures.  These new results reduce to existing results in the polynomial and Pfaffian cases.

\item New results are produced in the o-minimal setting for quantified formula over sets which need not be open or closed.  The application of this to the Pfaffian case is also new.  Main Result A (Theorem \ref{arbQuant}) gives a bound in terms of the summation of sets, and Main Result B (Theorem \ref{clsdUpBd}, really a Corollary of Theorem \ref{arbQuant}) gives a specific bound in terms of the complexity metric.
\end{itemize}

%%%%%%%%%%%%%%%%%%%%%%%%%%%%%%%%%%%%%%%%%%%%%%%%%%%%%%%%%%%%%%%%%%%%%%%%%%%%%%%%%%%%%%%%
\section{Background}
%%%%%%%%%%%%%%%%%%%%%%%%%%%%%%%%%%%%%%%%%%%%%%%%%%%%%%%%%%%%%%%%%%%%%%%%%%%%%%%%%%%%%%%%

%%%%%%%%%%%%%%%%%%%%%%%%%%%%%%%%%%%%%%%%%%%%%%%%%%%%%%%%%%%%%%%%%%%%%%%%%%%%%%%%%%%%%%%%
\subsection{O-Minimal Structures}\label{o-min}
%%%%%%%%%%%%%%%%%%%%%%%%%%%%%%%%%%%%%%%%%%%%%%%%%%%%%%%%%%%%%%%%%%%%%%%%%%%%%%%%%%%%%%%%
A very good text on this subject is \cite{tameTopOMin}.  The following two definitions are taken directly from \cite{OMin}:

\begin{defn}
A \em structure expanding the real closed field $R$ \em is a collection $S = (S_{n})_{n \in \mathbb{N}}$, where each $S_{n}$ is a set
of subsets of the affine space $R^{n}$, satisfying the following axioms:

\begin{itemize}

\item [1.] All algebraic subsets of $R^{n}$ are in $S_{n}$

\item [2.]For every $n$, $S_{n}$ is a Boolean subalgebra of the powerset of $R^{n}$

\item [3.] If $A \in S_{m}$, and $B \in S_{n}$, then $A \times B \in S_{m+n}$

\item [4.] If $p:R^{n+1} \to R^{n}$ is the projection on the first $n$ coordinates, and $A \in S_{n+1}$, then $p(A) \in S_{n}$.

\end{itemize}

The elements of $S_{n}$ are called the \em definable subsets of $R^{n}$\em.  The structure $S$ is said to be \em{o-minimal} \em if,
moreover, it satisfies:

\begin{itemize}

\item [5.] The elements of $S_{1}$ are precisely the finite unions of points and intervals.

\end{itemize}

\end{defn}

\begin{defn}
A map $f:A \to R^p$ (where $A \subset R^{n}$) is called \em definable \em if its graph is a definable subset of $R^{n} \times
R^{p}$.
\end{defn}

%%%%%%%%%%%%%%%%%%%%%%%%%%%%%%%%%%%%%%%%%%%%%%%%%%%%%%%%%%%%%%%%%%%%%%%%%%%%%%%%%%%%%%%%
\subsubsection{Conic Structure}
%%%%%%%%%%%%%%%%%%%%%%%%%%%%%%%%%%%%%%%%%%%%%%%%%%%%%%%%%%%%%%%%%%%%%%%%%%%%%%%%%%%%%%%%
Let $\overline{B_{n}}(a,r)$ denote the closed ball in $\R^{n}$ with radius $r$, centered at $a$, and let $S_{n}(a,r)$ denote the sphere in $\R^{n}$ with radius $r$, centered at $a$.
The following Theorems are taken from \cite{OMin}
\begin{thm}[Local Conic Structure]
Let $A \subset \R^{n}$ be a closed definable set, and $a$ a point in $A$.  There is a $r > 0$ such that there exists a definable
homeomorphism $h$ from the cone with vertex $a$ and base $S_{n}(a,r) \cap A$ onto $\overline{B_{n}}(a,r) \cap A$ satisfying
$h|_{S_{n}(a,r)\cap A} = \text{Id}$ and $|h(x) - a| = |x - a|$ for all $x$ in the cone.
\end{thm}

The following theorem tells us that any closed definable set is homotopy equivalent to the intersection of that set with a ball of
sufficiently large radius.  This is useful later when we have a closed set defined by a combination of equations and/or inequalities,
and we require a compact set -- we only need to add one additional inequality.

\begin{thm}[Conic Structure at Infinity]\label{conInf}
Let $A \subset \mathbb{R}^{n}$ be a closed definable set.  Then there exists $r \in \mathbb{R}$, $r > 0$, such that for every $r'$,
$r' \geq r$, there is a definable deformation retraction from $A$ to $A_{r'} = A \cap \overline{B_{n}}(0,r')$ and a definable
deformation retraction from $A_{r'}$ to $A_{r}$.
\end{thm}

\begin{proof}
Let us suppose $A$ is not bounded.  Through an inversion map $\phi:\mathbb{R}^{n} \setminus {0} \to \mathbb{R}^{n} \setminus
{0}$, $\phi(x) = x/|x|^{2}$, we can reduce to the property of local conic structure for $\phi(A) \cup {0}$ at $0$.
\end{proof}

This theorem can be visualised as follows: consider a simplicial complex $K$ equivalent to $A$.  Take $r$ large enough so that each
simplex in $K$ intersects with the ball of radius $r$, and each bounded simplex is a subset of the ball.  Then each unbounded simplex
can be ``shrunk'' in, and be replaced with an equivalent bounded simplex, without affecting the homotopy type.

%%%%%%%%%%%%%%%%%%%%%%%%%%%%%%%%%%%%%%%%%%%%%%%%%%%%%%%%%%%%%%%%%%%%%%%%%%%%%%%%%%%%%%%%
\subsection{Pfaffian Functions}\label{pfaffian}
%%%%%%%%%%%%%%%%%%%%%%%%%%%%%%%%%%%%%%%%%%%%%%%%%%%%%%%%%%%%%%%%%%%%%%%%%%%%%%%%%%%%%%%%
We present a definition for Pfaffian functions from \cite{PffNoe}, modified to only include the real case:

\begin{defn}
A \em Pfaffian chain \em of the order $r \geq 0$ and degree $\alpha \geq 1$ in an open domain $G \subset \mathbb{R}^{n}$ is a
sequence of analytic functions $f_{1}, \ldots f_{r}$ in $G$ satisfying differential equations
\begin{equation*}
df_{j}({\bf x}) = \sum_{1\leq i \leq n} g_{ij}({\bf x}, f_{1}({\bf x}),\ldots,f_{j}({\bf x}))dx_{i}
\end{equation*}
for $1 \leq j \leq r$.  Here $g_{ij}({\bf x}, y_{1}, \ldots, y_{j})$ are polynomials in ${\bf x} = (x_{1}, \ldots, x_{n})$,
$y_{1},\ldots,y_{j}$ of degrees not exceeding $\alpha$.  A function $f({\bf x}) = P({\bf x}, f_{1}({\bf x}),\ldots,f_{j}({\bf x}))$,
where $P({\bf x}, f_{1}({\bf x}),\ldots,f_{j}({\bf x}))$ is a polynomial of a degree not exceeding $\beta \geq 1$, is called a \em
Pfaffian function \em of order $r$ and degree $(\alpha, \beta)$.  Note that the Pfaffian function $f$ is defined only in the domain
$G$ where all functions $f_{1}, \ldots, f_{r}$ are analytic, even if $f$ itself can be extended as an analytic function to a larger
domain.
\end{defn}

We present a few examples to illustrate, taken from \cite{PffNoe}.
\begin{eg}
\begin{enumerate}[(a)]
\item Pfaffian functions of order $0$ and degree $(1,\beta)$ are polynomials of degree not exceeding $\beta$.

\item The exponential function $f(x) = e^{ax}$ is a Pfaffian function of order $1$ and degree $(1,1)$ in $\mathbb{R}$, due to the
    equation $df(x) = af(x)dx$.

\item The function $f(x) =1/x$ is a Pfaffian function of order $1$ and degree $(2,1)$ in the domain $\{x \in \mathbb{R}:x  \neq
    0\}$, due to the equation $df(x) = -f^{2}(x)dx$.

\item The polynomial $f(x) = x^{m}$ can be viewed as a Pfaffian function of order $2$ and degree $(2,1)$ in the domain $\{x \in
    \mathbb{R}:x  \neq 0\}$ (but not in $\mathbb{R}$), due to the equations $df(x) = mf(x)g(x)dx$ and $dg(x) = -g^{2}(x)dx$,
    where $g(x) = 1/x$.
\end{enumerate}
\end{eg}
The expansion of $\mathbb{R}$ by sets defined by Pfaffian functions was proven to be o-minimal in \cite{pfaffOMin}.
The following  lemmas and definitions are also taken from \cite{PffNoe}.

\begin{lem}\label{pfafSumProd}
The sum (resp. product) of two Pfaffian functions $f_{1}$ and $f_{2}$ of orders $r_{1}$ and $r_{2}$ and degrees $(\alpha_{1},
\beta_{1})$ and $(\alpha_{2}, \beta_{2})$ respectively, is a Pfaffian function of order $r_{1}+r_{2}$ and degree $(\alpha,
\text{max}\{\beta_{1}, \beta_{2}\})$ (resp. $(\alpha, \beta_{1} + \beta_{2})$, where $\alpha = \text{max}\{\alpha_{1}, \alpha_{2}\}$.
If the two functions are defined by the same Pfaffian chain of order $r$, then the orders of the sum and the product are both equal
to $r$.
\end{lem}

\begin{lem}\label{pfafDeriv}
A partial derivative of a Pfaffian function of order $r$ and degree $(\alpha, \beta)$, is a Pfaffian function having the same
Pfaffian chain of order $r$, and degree $(\alpha, \alpha + \beta -1)$.
\end{lem}

\begin{defn}
A set $X \subset \mathbb{R}^{n}$ is called \em semi-Pfaffian \em in an open domain $G \subset \mathbb{R}^{n}$ if it consists of
points in $G$ satisfying a Boolean combination $\mathcal{F}$ of some atomic equations and inequalities $f=0$, $g>0$, where $f,g$ are
Pfaffian functions having a common Pfaffian chain defined in $G$.  We will write $X = \{\mathcal{F}\}$.  A semi-Pfaffian set $X$ is
\em restricted in $G$ \em if its topological closure lies in $G$.  A semi-Pfaffian set is called \em basic \em if the Boolean
combination is just a conjunction of equations and strict inequalities.
\end{defn}

\begin{defn}
A set $X \subset \mathbb{R}^{n}$ is called \em sub-Pfaffian \em in an open domain $G \subset \mathbb{R}^{n}$ if it is an image of a
semi-Pfaffian set under a projection into a subspace.
\end{defn}

\begin{defn}
Consider the closed cube $I^{m+n} = [-1,1]^{m+n}$ in an open domain $G \subset \mathbb{R}^{m+n}$, and the projection map
$\pi:\mathbb{R}^{m+n} \to \mathbb{R}^{n}$.  A subset $Y \subset I^{n}$ is called \em restricted sub-Pfaffian \em if $Y = \pi(X)$ for
a restricted semi-Pfaffian set $X \subset I^{m+n}$.
\end{defn}

Note that a restricted sub-Pfaffian set need not be semi-Pfaffian, see \cite{PffNoe} for an example.  This is the most significant
difference between the theories of semi- and sub-Pfaffian sets on the one hand, and semialgebraic sets on the other, and is a key
reason this thesis is needed.

We now present an analogue of Bezout's theorem for Pfaffian functions:

\begin{thm}[Khovanskii, \cite{Khov1}, \cite{Khov2}]\label{khovanskii}
Consider a system of equations $f_{1} = \hdots = f_{n} = 0$, where $f_{i}$, $1 \leq i \leq n$, are Pfaffian functions in a domain $G
\subset \mathbb{R}^{n}$, having a common Pfaffian chain of order $r$, and degrees $(\alpha, \beta_{i})$ respectively.  Then, the
number of non-degenerate solutions of this system does not exceed
\begin{equation*}
\mathcal{M}(n,r,\alpha,\beta_{1}, \ldots, \beta_{n}) = 2^{r(r-1)/2}\beta_{1}\hdots\beta_{n}(\text{min}\{n,r\}\alpha + \beta_{1} +
\hdots + \beta_{n} - n + 1)^{r}.
\end{equation*}
\end{thm}

%%%%%%%%%%%%%%%%%%%%%%%%%%%%%%%%%%%%%%%%%%%%%%%%%%%%%%%%%%%%%%%%%%%%%%%%%%%%%%%%%%%%%%%%
\subsection{Hardt's Triviality}
%%%%%%%%%%%%%%%%%%%%%%%%%%%%%%%%%%%%%%%%%%%%%%%%%%%%%%%%%%%%%%%%%%%%%%%%%%%%%%%%%%%%%%%%
The following is taken from \cite{OMin}.
\begin{defn}
Let $X \subset \R^{m} \times \R^{n}$ be a definable family, and denote by $\pi_{m}$ the projection onto $\R^{m}$.  Let $A$ be a
definable subset of $\R^{m}$, and let $X_{A} = X \cap (A \times \R^{n})$.  We say that the family $X$ is \em definably trivial over
\em $A$ if there exists a definable set $F$ and a definable homeomorphism $h:A \times F \to X_{A}$ such that, for $x \in A \times F$
\begin{equation*}
\pi_{m}(h(x)) = \pi_{m}(x).
\end{equation*}

We say that $h$ is a \em definable trivialisation of $X$ over $A$\em.  Now let $Y$ be a definable subset of $X$.  We say that the
trivialisation $h$ is \em compatible with $Y$ \em if there is a definable subset $G$ of $F$ such that $h(A \times G) = Y_{A}$.  Note
that if $h$ is compatible with $Y$, its restriction to $Y_{A}$ is a trivialisation of $Y$ over $A$.
\end{defn}

%The following theorem can be considered as somewhat of an analogue to the elementary geometric notion that the area of a rectangle is
%equal to the length times the height.

\begin{thm}[Hardt's Theorem for Definable Families]\label{Hardt}
Let $X \subset \R^{m} \times \R^{n}$ be a definable family. Let $Y_{1},\ldots,Y_{l}$ be definable subsets of $X$. There exists a
finite partition of $\R^{m}$ into definable sets $C_{1},\ldots,C_{k}$ such that $X$ is definably trivial over each $C_{i}$, and
moreover, the trivilisations over each $C_{i}$ are compatible with $Y_{1},\ldots,Y_{l}$.
\end{thm}

This will frequently be used in a situation like the following (see Definition \ref{restrNotat} for notation).

\begin{eg}
For a definable set $A\subset \R^{n}$ there exists a homeomorphism $\phi:A_{a} \times (0, a] \to A_{(0,a]}$.  Moreover the
homeomorphism can be chosen so that for $x \in A_{a}$, $y \in (0,a)$, the projection onto the second factor of $\phi(x,y)$ is $y$,
and $\phi(A_{a}, a) = A_{a}$.
\end{eg}

%%%%%%%%%%%%%%%%%%%%%%%%%%%%%%%%%%%%%%%%%%%%%%%%%%%%%%%%%%%%%%%%%%%%%%%%%%%%%%%%%%%%%%%%
\subsection{Complexity Notation}
%%%%%%%%%%%%%%%%%%%%%%%%%%%%%%%%%%%%%%%%%%%%%%%%%%%%%%%%%%%%%%%%%%%%%%%%%%%%%%%%%%%%%%%%
In this thesis we use some notation that is not standard to Complexity theory.  The ``$O$'' notation that follows is as normal, but the ``$o$'' notation is defined differently (taken from \cite{AlgInRAG}), and $\Omega$ is a function defined later in the text, and is nothing to do with complexity.

We use the  notation, whenever a more explicit bound is not more useful.
\begin{defn}
We say $f(x) \leq O(g(x))$ if there exists a positive real number $M$ and an $x_{0} \in \R$ such that
\begin{equation*}
|f(x)| \leq M|g(x)|
\end{equation*}
for all $x > x_{0}$.
\end{defn}

The following notation allows us to express how ``small'' certain expressions are, and is taken from \cite{AlgInRAG}
\begin{defn}
For an expression $f = \sum_{i} (a_{i} \delta^{r_{i}})$, where $a_{i} \neq 0$ and  $\delta$ is taken to be very small, we denote by
$o(f)$ the smallest value of $r_{i}$.
\end{defn}

%%%%%%%%%%%%%%%%%%%%%%%%%%%%%%%%%%%%%%%%%%%%%%%%%%%%%%%%%%%%%%%%%%%%%%%%%%%%%%%%%%%%%%%%
\subsection{Sard's Theorem}
%%%%%%%%%%%%%%%%%%%%%%%%%%%%%%%%%%%%%%%%%%%%%%%%%%%%%%%%%%%%%%%%%%%%%%%%%%%%%%%%%%%%%%%%
The following is taken from \cite{DiffTop}:

\begin{defn}
An $n$-cube $C \subset \mathbb{R}^{n}$ of edge $\lambda > 0$ is a product $C = I_{1} \times \hdots \times I_{n}$ of closed intervals
of length $\lambda$.  The measure of $C$ is $\lambda^{n}$.  A subset $X \subset \mathbb{R}^{n}$ has measure zero if for every
$\varepsilon > 0$ it can be covered by a family of $n$-cubes, the sum of whose measures is less than $\varepsilon$.
\end{defn}

\begin{lem}
A countable union of sets of measure zero has measure zero.
\end{lem}

\begin{proof}
Trivial.
\end{proof}

\begin{thm}[Sard's Theorem]
Let $M$, $N$ be manifolds of dimensions $m$, $n$ and $f: M \to N$ a $C^{r}$ map.  If $r > max\{0, m-n\}$ then the set of critical
values of $f$ has measure zero in $N$.
\end{thm}

%%%%%%%%%%%%%%%%%%%%%%%%%%%%%%%%%%%%%%%%%%%%%%%%%%%%%%%%%%%%%%%%%%%%%%%%%%%%%%%%%%%%%%%%
\subsection{Basic concepts in Algebraic Topology}\label{basicAT}
%%%%%%%%%%%%%%%%%%%%%%%%%%%%%%%%%%%%%%%%%%%%%%%%%%%%%%%%%%%%%%%%%%%%%%%%%%%%%%%%%%%%%%%%
For a topological space $X$, we denote by $H_{i}(X)$ its $i$-th homology group, and by $b_{i}(X) = \text{rank }H_{i}(X)$ its $i$-th
Betti number.  We denote by $b(X) = \sum_{i} b_{i}(X)$ the sum of the Betti numbers of $S$.

\begin{lem}
The number of connected components of a non-empty closed and bounded definable set $S$ is equal to $b_{0}(S)$
\end{lem}
\begin{proof}
See for example Proposition 6.26 of \cite{AlgInRAG}.
\end{proof}

Since in this thesis we spend very little time dealing with homotopy groups, we usually use the symbol $\pi$ for projections.
However, in the few cases where we must mention homotopy groups, we follow the standard convention of using $\pi_{i}(X)$ to denote
the $i$-th homotopy group of $X$, but clearly explain the situation to avoid any confusion.

The symbol $\simeq$ denotes homotopy equivalence.  If $Y \subset X$, then $closure(Y)$ denotes its closure in $X$, and $\overline{Y}$
its complement.  It is important to note that this departs from usual conventions, but is necessary due to the complex expressions
that follow.

%%%%%%%%%%%%%%%%%%%%%%%%%%%%%%%%%%%%%%%%%%%%%%%%%%%%%%%%%%%%%%%%%%%%%%%%%%%%%%%%%%%%%%%%
\subsection{Significantly Small Real Numbers}
%%%%%%%%%%%%%%%%%%%%%%%%%%%%%%%%%%%%%%%%%%%%%%%%%%%%%%%%%%%%%%%%%%%%%%%%%%%%%%%%%%%%%%%%
We now give a definition for the $\ll$ symbol.

\begin{defn}\label{ll}
Let $\mathcal{P}=\mathcal{P}(\varepsilon_{0}, \ldots, \varepsilon_{l})$ be a predicate (property) over $(0,1)^{l+1}$. We say that
property $\mathcal{P}$ holds for
\begin{equation*}
0 < \varepsilon_{0} \ll \varepsilon_{1} \ll \ldots \ll \varepsilon_{l} \ll 1
\end{equation*}
if there exist definable functions $f_{k}:(0,1)^{l-k} \to (0,1)$, where $k=0,\ldots,l$ (with $f_{l}$ being a positive constant), such
that $\mathcal{P}$ holds for any sequence $\varepsilon_{0},\ldots,\varepsilon_{l}$ satisfying
\begin{equation*}
0<\varepsilon_{k}<f_{k}(\varepsilon_{k+1},\ldots,\varepsilon_{l}) \text{  for } k = 0,\ldots, l.
\end{equation*}
\end{defn}

%%%%%%%%%%%%%%%%%%%%%%%%%%%%%%%%%%%%%%%%%%%%%%%%%%%%%%%%%%%%%%%%%%%%%%%%%%%%%%%%%%%%%%%%
\subsection{Mayer-Vietoris Inequalities}\label{MV}
%%%%%%%%%%%%%%%%%%%%%%%%%%%%%%%%%%%%%%%%%%%%%%%%%%%%%%%%%%%%%%%%%%%%%%%%%%%%%%%%%%%%%%%%
Let $S_{1}$ and $S_{2}$ be closed and bounded definable sets.  Then there exists the following long exact sequence of homology
groups:
\begin{equation*}
\cdots \to H_{i}(S_{1} \cap S_{2}) \to H_{i}(S_{1}) \oplus H_{i}(S_{2}) \to H_{i}(S_{1} \cup S_{2}) \to H_{i-1}(S_{1} \cap S_{2}) \to
\cdots
\end{equation*}
see for example \cite{AT}, \cite{SignCon}.

From this, the following inequalities are easily derived:
\begin{align}
b_{i}(S_{1}) + b_{i}(S_{2}) &\leq b_{i}(S_{1} \cup S_{2}) + b_{i}(S_{1} \cap S_{2}) \label{MV1}\\
b_{i}(S_{1} \cap S_{2}) &\leq b_{i}(S_{1}) + b_{i}(S_{2}) + b_{i+1}(S_{1} \cup S_{2})\label{MV3},
\end{align}
and if $i>0$
\begin{align}
b_{i}(S_{1} \cup S_{2}) &\leq b_{i}(S_{1}) + b_{i}(S_{2}) + b_{i-1}(S_{1} \cap S_{2})\label{MV2}.
\end{align}
If $i=0$ we have $b_{0}(S)$ is equal to the number of connected components of $S$, so we have
\begin{align*}
b_{0}(S_{1} \cup S_{2}) &\leq b_{0}(S_{1}) + b_{0}(S_{2}) .
\end{align*}

More generally we have:

\begin{lem}[Mayer-Vietoris inequality]\label{MVGen}
Let $X_{1}, \ldots, X_{m} \subset [-1,1]^{n}$ be all open or all closed in $[-1,1]^{n}$.  Then
\begin{equation*}
b_{i}\left( \bigcup_{1\leq j \leq n} X_{j} \right) \leq \sum_{J \subset \{1, \ldots, n\}} b_{i - |J| + 1} \left( \bigcap_{j \in J}
X_{j} \right)
\end{equation*}
and
\begin{equation*}
b_{i}\left( \bigcap_{1\leq j \leq n} X_{j} \right) \leq \sum_{J \subset \{1, \ldots, n\}} b_{i + |J| - 1} \left( \bigcup_{j \in J}
X_{j} \right),
\end{equation*}
where $b_{i}$ is the $i$-th Betti number.
\end{lem}

\begin{proof}
See \cite{BNSASPSets}.
\end{proof}

%%%%%%%%%%%%%%%%%%%%%%%%%%%%%%%%%%%%%%%%%%%%%%%%%%%%%%%%%%%%%%%%%%%%%%%%%%%%%%%%%%%%%%%%
\section{Complexity}\label{complexSec}
%%%%%%%%%%%%%%%%%%%%%%%%%%%%%%%%%%%%%%%%%%%%%%%%%%%%%%%%%%%%%%%%%%%%%%%%%%%%%%%%%%%%%%%%
In \cite{RASASets} the following axiomatic complexity metric was presented:
\begin{defn}
A \em complexity \em on $\mathbb{R}[X_{1},\ldots X_{n}]$ is a map $c:\mathbb{R}[X_{1},\ldots X_{n}] \to \mathbb{N}$ satisfying the following axioms:
\begin{itemize}
\item[(c1)] if $P$ is constant, then $c(P)=0$; $c(X_{i})=1$ $(1 \leq i \leq n)$;

\item[(c2)] $c(P+Q) \vartriangleleft (c(P), c(Q))$; $c(PQ) \vartriangleleft (c(P),c(Q))$  $\forall P,Q \in \mathbb{R}[X_{1},\ldots X_{n}]$;

\item [(c3)] $c(\partial P/\partial X_{i})\vartriangleleft c(P)$ $\forall P \in \mathbb{R}[X_{1},\ldots X_{n}]$;

\item [(c4)] for $P_{1}, \ldots P_{n}  \in \mathbb{R}[X_{1},\ldots X_{n}]$, let $n(S)$ be the number of non-degenerate solutions of the system $S: P_{1} = \cdots = P_{n} = 0$, then $n(S) \vartriangleleft (n, c(P_{1}), \ldots, c(P_{n}))$.
\end{itemize}
\end{defn}
The $\vartriangleleft$ notation was defined as follows:
\begin{defn}
If $g:E \to \mathbb{N}$ and $f: E \to F$ are set maps, we write $g \vartriangleleft f$ if there is a set map $h:F \to \mathbb{N}$, ``effectively'' computable, such that $g(x) \leq h(f(x)) \forall x \in E$.
\end{defn}

One of the main contributions of this thesis is that I present a new definition of an axiomatic complexity metric, that improves on the preceding in the following ways:
\begin{itemize}
\item We allow different classes of functions, rather than simply considering polynomials

\item We allow multivariate complexity -- the complexity of polynomials can simply be measured by degree, but other functions (for example Pfaffian) may have more that a single complexity measure

\item Functions must be specified in place of the rather vague $\vartriangleleft$ notation

\item We specify that adding a constant to a function and multiplying a function by a non-zero constant cannot change complexity
\end{itemize}

The new axiomatic complexity metric is as follows:

\begin{defn}\label{complexityDefn}

Let $H  \subset \mathbb{R}^{n}$ be a definable open set, and consider a set $D$ of differentiable functions from $H$ to $\mathbb{R}$ in an o-minimal
structure over the reals, such that $D$ is closed under addition, multiplication and partial derivatives.  Fix a positive integer $m$.
A \textit{complexity}  on $D$ is a pair $(c,T)$, where $c = (c_{1}, \ldots c_{m})$ with $c_1, \ldots, c_m: D \to \mathbb{N}$,
and $T = (T_1, \ldots ,T_m)$ with $T_i=(t_{i,+},t_{i, \times}, t_{i, \partial},\{ t_{n} \}_{1 \le n < \infty})$, such that:

\begin{itemize}

\item[(c0)] $t_{i,+}, t_{i,\times}: \mathbb{N}^{2m} \to \mathbb{N}$, $t_{i,\partial}:\mathbb{N}^{m}\to\mathbb{N}$,
    $t_{n}:\mathbb{N}^{nm}\to\mathbb{N}$

\item[(c1)] If $F \in D$ is constant, then $c_{i}(F) = 0$; $c_{i}(X_{i}) = 1$ $(1 \leq i \leq m)$;

\item[(c2)] For all $F, G \in D$,
\begin{equation*}
c_{i}(F+G) \leq t_{i,+}(c_{1}(F), \ldots ,c_{m}(F), c_{1}(G), \ldots , c_{m}(G))\text{ for all }i=1, \ldots, m
\end{equation*}
and
\begin{equation*}
c_{i}(FG) \leq t_{i,\times}(c_{1}(F), \ldots ,c_{m}(F), c_{1}(G), \ldots , c_{m}(G))\text{ for all }i=1, \ldots, m.
\end{equation*}
Furthermore if $F \in D$ is constant, then $c_{i}(F+G) = c_{i}(G)$, and if $F$ is also non-zero then $c_{i}(FG) = c_{i}(G)$.

\item[(c3)] For all $F \in D$, $c_{i}\left(\pd{F}{x_{j}}\right) \leq t_{i, \partial}(c_{1}(F), \ldots , c_{m}(F))$;

\item[(c4)] For $F_{1},\ldots,F_{n} \in D$, let $\chi$ be the number of
solutions in $\mathbb{R}^{n}$ of the system $F_{1}=\cdots=F_{n}=0$.  Then  $\chi \leq t_{n}(c_{1}(F_{1}), \ldots ,c_{m}(F_{1}), \ldots , c_{1}(F_{n}), \ldots,c_{m}(F_{n}))$.

\end{itemize}

\end{defn}

We will let $t_{+}$ denote $(t_{1,+},\ldots, t_{m,+})$,  $t_{\times}$ denote $(t_{1,\times},\ldots, t_{m,\times})$, and $t_{\partial}$ denote $(t_{1,\partial},\ldots$ $\ldots, t_{m,\partial})$.  We will also sometimes use $c(F)$ to denote $(c_{1}(F), \ldots , c_{m}(F))$.

\begin{eg}[The Degree of a Polynomial] \label{deg eg}
Let $H = \R^{n}$, $m=1$, $D = \mathbb{R}[X_{1},  \ldots, X_{n}]$, with $P, P_{i} \in D$, we have $c(P) = (c_{1}(P)) = (deg(P))$, and $t_{1,+} =
max(c_{1}(P_{1}), c_{1}(P_{2}))$, $t_{1,\times} = c_{1}(P_{1}) + c_{1}(P_{2})$, $t_{1,\partial} = c_{1}(P) - 1$, and by Bezout's
theorem (\cite{Milnor} Lemma 1),
$t_{n} = c_{1}(P_{1})c_{1}(P_{2}) \cdots$
$\cdots  c_{1}(P_{n})$.
\end{eg}

\begin{eg}[Pfaffian Functions]\label{pfaffEg}
Let $H \subset \R^{n}$ be an open domain (for example $(0,1)^{n}$), $D$ be the set of Pfaffian functions, and take $m=3$.
 Let $f_{1}, \ldots, f_{n} \in D$ have degrees $c(f_{i}) = (c_{1}, c_{2}, c_{3}) = (\alpha_{i}, \beta_{i}, r_{i})$ for $1
\leq i \leq n$.   Then from Lemma \ref{pfafSumProd} we can take
\begin{align*}
c(f_{1}+f_{2}) = (max\{\alpha_{1}, \alpha_{2}\}, max\{\beta_{1}, \beta_{2}\}, r_{1}+r_{2}), \text{ so}\\
t_{1,+} = max\{\alpha_{1}, \alpha_{2}\}, t_{2,+} = max\{\beta_{1}, \beta_{2}\}), t_{3,+} = r_{1}+r_{2};
\end{align*}
\begin{align*}
c(f_{1}f_{2}) = (max\{\alpha_{1}, \alpha_{2}\}, \beta_{1} + \beta_{2}, r_{1}+r_{2}), \text{ so}\\
t_{1, \times} = max\{\alpha_{1}, \alpha_{2}\}, t_{2, \times} = \beta_{1} + \beta_{2}, t_{3,\times} = r_{1}+r_{2}.
\end{align*}
If $f_{1}$ and $f_{2}$ are defined by the same Pfaffian chain of order $r$ we can further refine the above to
\begin{equation*}
t_{3,+} = t_{3,\times} = r.
\end{equation*}
We also have from Lemma \ref{pfafDeriv}
\begin{align*}
c\left(\pd{f_{1}}{x_{j}}\right) = (\alpha_{1}, \alpha_{1}+\beta_{1}-1, r_{1}), \text{ so}\\
t_{1, \partial} = \alpha_{1}, t_{2, \partial} = \alpha_{1}+\beta_{1}-1, t_{3, \partial} = r_{1}.
\end{align*}
If we take $\text{max}_{i}\{\alpha_{i}\} = \alpha$ and $\text{max}_{i}\{r_{i}\} = r$,  Khovanskii's bound states the number of non-degenerate solutions of
$f_{1} = \hdots = f_{n}=0$ is bounded by
\begin{equation*}
t_{n}=2^{r(r-1)/2}\beta_{1}\hdots\beta_{n}(\text{min}\{n,r\}\alpha + \beta_{1} + \hdots + \beta_{n} - n + 1)^{r}
 \end{equation*}
 (see Section \ref{pfaffian}).
\end{eg}

To avoid overcomplicating future expressions, we also define the following:
\begin{defn}
For a complexity $(c,T)$, we define $t_{+,s}:\mathbb{N}^{s} \to \mathbb{N}$ and $t_{\times,s}:\mathbb{N}^{s} \to \mathbb{N}$
inductively:
\begin{align*}
t_{+,2} &= t_{+}\\
t_{+,s}(c_{1}, \ldots, c_{s}) &= t_{+}(t_{+,s-1}(c_{1}, \ldots, c_{s-1}), c_{s}).
\end{align*}
For the definition of $t_{\times,s}$, replace $+$ by $\times$.

We define for repeated arguments to these functions:
\begin{equation*}
t_{+,s}^{*}(m) = t_{+,s}(m,\ldots,m)
\end{equation*}
and define $t_{\times,s}^{*}(m)$ similarly.
\end{defn}

%%%%%%%%%%%%%%%%%%%%%%%%%%%%%%%%%%%%%%%%%%%%%%%%%%%%%%%%%%%%%%%%%%%%%%%%%%%%%%%%%%%%%%%%%
\subsection{The Functions $\kappa$ and $\gamma$}
%%%%%%%%%%%%%%%%%%%%%%%%%%%%%%%%%%%%%%%%%%%%%%%%%%%%%%%%%%%%%%%%%%%%%%%%%%%%%%%%%%%%%%%%%
We will define some functions, whose backgrounds will be explained in full at a later stage.

The first arises when we compute the complexity of a function, after rotating the coordinates so that certain specifications are
satisfied.  Let $\kappa = (\kappa_{1}, \ldots, \kappa{_m})$, where for $1 \leq i \leq m$,
\begin{center}
$\kappa_{i}(c(F)) = t_{i+}(t_{i\partial}(c_{1}(F), \ldots , c_{m}(F)),t_{i\times}(0,t_{i\partial}(c_{1}(F), \ldots , c_{m}(F))))$
\end{center}

The second function comes from when we are trying to find critical points on a hypersurface: we want the function $F$ to be zero,
along with certain combinations of partial derivatives, taking into account the rotation above.  In the following, remember $t_{n}$ has $nm$
arguments:
\begin{center}
$\gamma(n,c(F)) = t_{n}(c_{1}(F), \ldots , c_{m}(F), \kappa(c(F)),\kappa(c(F)), \ldots, \kappa(c(F))$
\end{center}

\begin{eg}[The Degree of a Polynomial]\label{gamPolEg}
In the case of Example \ref{deg eg}, if we have a polynomial $P$ of degree $d$, then $\kappa(c(P)) = d-1$, and $\gamma(n, c(P)) = d(d-1)^{n-1}$.
\end{eg}

\begin{eg}[Pfaffian Functions]\label{gamPfaffEg}
Consider the setting of Example \ref{pfaffEg}, where we deal with Pfaffian functions.  Take such a function $f$ with complexity $(\alpha, \beta, r)$, then $\kappa(c(f)) = (\alpha, \alpha+\beta-1, r)$, and
\begin{align*}
\gamma(n, c(f)) &= t_{n}(c(f), \kappa(c(f)), \ldots, \kappa(c(f)))\\
&= 2^{r(r-1)/2}\beta(\alpha+\beta-1)^{n-1}(\text{min}\{n,r\}\alpha + \beta + (n-1)(\alpha+\beta-1)-n+1)^{r}\\
&= 2^{r(r-1)/2}\beta(\alpha+\beta-1)^{n-1}(\text{min}\{n,r\}\alpha + n\beta + (n-1)\alpha-2n+2)^{r}\\
\end{align*}
\end{eg}

%%%%%%%%%%%%%%%%%%%%%%%%%%%%%%%%%%%%%%%%%%%%%%%%%%%%%%%%%%%%%%%%%%%%%%%%%%%%%%%%%%%%%%%%
%%%%%%%%%%%%%%%%%%%%%%%%%%%%%%%%%%%%%%%%%%%%%%%%%%%%%%%%%%%%%%%%%%%%%%%%%%%%%%%%%%%%%%%%
\chapter{Sets defined by equalities and non-strict inequalities}\label{chComp}
%%%%%%%%%%%%%%%%%%%%%%%%%%%%%%%%%%%%%%%%%%%%%%%%%%%%%%%%%%%%%%%%%%%%%%%%%%%%%%%%%%%%%%%%
%%%%%%%%%%%%%%%%%%%%%%%%%%%%%%%%%%%%%%%%%%%%%%%%%%%%%%%%%%%%%%%%%%%%%%%%%%%%%%%%%%%%%%%%

\section{Sets defined by equalities}
%%%%%%%%%%%%%%%%%%%%%%%%%%%%%%%%%%%%%%%%%%%%%%%%%%%%%%%%%%%%%%%%%%%%%%%%%%%%%%%%%%%%%%%%
\subsection{Introduction}
%%%%%%%%%%%%%%%%%%%%%%%%%%%%%%%%%%%%%%%%%%%%%%%%%%%%%%%%%%%%%%%%%%%%%%%%%%%%%%%%%%%%%%%%
We follow the methods of Milnor to produce a bound for the sum of the Betti numbers of a set defined by equalities.  We construct a
hypersurface that has a ``nice'' projection map, and show that when moving a plane orthogonally to the axis of this projection map, Betti numbers can
only increase or decrease at critical points of this projection map.  The number of critical points of this projection map can be found using
normal analytic methods, and thus a relationship between algebra and topology has been found.  We now move to find the link between
this hypersurface and the area it bounds.

In the following, let $(c,T)$ be a complexity on a set $D$.

%%%%%%%%%%%%%%%%%%%%%%%%%%%%%%%%%%%%%%%%%%%%%%%%%%%%%%%%%%%%%%%%%%%%%%%%%%%%%%%%%%%%%%%%
\subsection{Main Result}
%%%%%%%%%%%%%%%%%%%%%%%%%%%%%%%%%%%%%%%%%%%%%%%%%%%%%%%%%%%%%%%%%%%%%%%%%%%%%%%%%%%%%%%%
Our main result in this section is:
\begin{thm} \label{main}
Let $f_{1}, \ldots, f_{m}$ be functions in $D$, and let $S = \{f_{1} = \ldots = f_{m} =0\}$.  Then
the sum of the Betti numbers  $b(S)$ of $S$ satisfies
\begin{center}
$b(S) \leq \dfrac{\gamma(n, c(f_{1}^{2} + \ldots + f_{m}^{2} + |x|^{2}))}{2}$.
\end{center}
\end{thm}

\begin{eg}[The Degree of a Polynomial]
Using the complexity described in Example \textup{\ref{deg eg}}, if the functions $f_{i}$ have maximum degree $d$, then $f_{1}^{2} + \ldots + f_{n}^{2}+ |x|^{2}$ has maximum degree $2d$, and $\gamma = 2d(2d-1)^{n-1}$, and $b(S) \leq d(2d-1)^{n-1}$, which is the well-known Thom-Milnor bound.
\end{eg}

\begin{eg}[Pfaffian Functions]
Consider the setting of Example \ref{pfaffEg}, and allow  functions $f_{i}$ to have a common Pfaffian chain of order $r$, and have complexity $(\alpha_{i}, \beta_{i}, r)$.  Then $f_{i}^{2}$ has complexity $(\alpha_{i}, 2\beta_{i}, r)$, and $f_{1}^{2} + \ldots + f_{n}^{2}+ |x|^{2}$ has complexity \begin{align*}
(\text{max}_{i}\{\alpha_{i}\}, 2\text{max}_{i}\{\beta_{i}\}, r)=(\alpha, 2\beta, r).
\end{align*}
Therefore
\begin{align*}
b(S) \leq& \gamma(k, c(F))/2 \\
=& 2^{r(r-1)/2}\beta(\alpha+2\beta-1)^{n-1}(\text{min}\{n,r\}\alpha + 2n\beta + (n-1)(\alpha)-2n+2)^{r}.
\end{align*}
This is equivalent to a bound given in \cite{Zell}.
\end{eg}

\begin{defn}\label{piDef}
Define $\pi:\R^{n} \to \R$ to be the projection onto  the $X_{1}$-axis, sending $x = (x_{1}, \ldots,
x_{n})\in \mathbb{R}^{n}$ to $x_{1}\in\mathbb{R}$.
\end{defn}
We will also use the following notation, taken from \cite{AlgInRAG}:
\begin{defn}\label{restrNotat}
For $S \subset \mathbb{R}^{n}$, and $X \subset \mathbb{R}$, let $S_{X}$ denote $S \cap
\pi^{-1}(X)$.  We also use the abbreviations $S_{x}$, $S_{< x}$, and $S_{\leq x}$ for $S_{\{x\}}$, $S_{(-\infty, x)}$, and
$S_{(-\infty, x]}$ respectively.
\end{defn}

%%%%%%%%%%%%%%%%%%%%%%%%%%%%%%%%%%%%%%%%%%%%%%%%%%%%%%%%%%%%%%%%%%%%%%%%%%%%%%%%%%%%%%%%
\subsection{Projections and Morse Functions}
%%%%%%%%%%%%%%%%%%%%%%%%%%%%%%%%%%%%%%%%%%%%%%%%%%%%%%%%%%%%%%%%%%%%%%%%%%%%%%%%%%%%%%%%
In this section we show that we can rotate the coordinates to ensure that $\pi$ is a \em Morse function\em.  We closely follow the
working of \cite{AlgInRAG}.

\begin{defn}\label{diffeo}
Let $W=\{f=0\}$ be a compact, non-singular (i.e. $Grad(f(x))\vert_{W} \neq 0$ at every point $x \in W$)  hypersurface. Let $p \in W$ be a critical point of $\pi$.  We can choose $(X_{2},\ldots,
X_{n})$ to be a local system of coordinates in a sufficiently small neighbourhood of $p$.  More precisely, we have an open
neighbourhood $U \subset \mathbb{R}^{n-1}$ of $p'=(p_{2}, \ldots, p_{n})$ and a map $\phi:U \to \mathbb{R}$ such that, with $x' =
(x_{2},\ldots, x_{n})$ and
\begin{equation*}
\Phi(x') = (\phi(x'), x') \in W,
\end{equation*}
$\Phi$ is a diffeomorphism from $U$ to $\Phi(U)$.
\end{defn}
\begin{defn}\label{non-deg}
The critical point $p$ is non-degenerate if the $(n-1) \times (n-1)$ Hessian matrix
\begin{equation*}
 \left[\frac{\partial^{2}\phi}{\partial X_{i} \partial X_{j}}(p')\right], 2 \leq i, j \leq n
\end{equation*}
is invertible.
\end{defn}
\begin{defn}
The function $\pi$ is called a Morse function if all its critical points are non-degenerate, and all critical values are pair-wise distinct.
\end{defn}

Before moving onto the next Lemma, we define the following basic concepts:

\begin{defn}[Orthogonal matrix]
An matrix is \em orthogonal \em if its transpose is equal to its inverse.  When considered as a linear map, an orthogonal matrix corresponds to an isometry (distance preserving map) between spaces, for example a rotation or reflection.
\end{defn}

\begin{defn}[Changes of coordinates]
An \em orthogonal change of coordinates \em consists of multiplying the existing coordinate system by an orthogonal matrix. This allows manipulation of the ``viewpoint'' of the space without affecting topological invariants.  A particular example of this is rotation about a point -- in this case the orthogonal matrix must also have determinant 1.
\end{defn}

\begin{lem} \label{morse}
Let $f$ be a function in $D$, and $W$ be as defined in Definition \ref{diffeo}.  Up to an orthogonal change of coordinates, the projection $\pi$ onto the $X_{1}$-axis is a Morse function on $W$.
\end{lem}

To prove this, we first need the following.

\begin{defn}[Gauss map]
Define the \em Gauss map \em from the hypersurface $W \subset \R^{n}$ to the unit sphere $S^{n-1} \subset \R^{n}$ by:
\begin{equation*}
g(x)=\dfrac{Grad(f(x))}{|Grad(f(x))|}.
\end{equation*}
\end{defn}

\begin{lem}\label{GaussDif}
Let $p \in W$ be a critical point of $\pi$.
Then $p$ is not a critical point of the Gauss map if and only if $p$ is a non-degenerate critical point of $\pi$.
\end{lem}

\begin{proof}
Since $p$ is a critical point of $\pi$, $g(p) = (\pm 1, 0, \ldots, 0)$.  Using Definition \ref{diffeo}, for $x' \in U$, $x = \Phi(x') =
(\phi(x'),x')$, and applying the chain rule,
\begin{equation*}
\frac{\partial f}{\partial X_{i}}(x) + \frac{\partial f}{\partial X_{1}}(x) \frac{\partial \phi}{\partial X_{i}}(x') = 0, 2 \leq i
\leq n.
\end{equation*}
Thus
\begin{equation*}
g(x) = \pm \frac{1}{\sqrt{1 + \sum_{i=2}^{n}\left(\frac{\partial \phi}{\partial X_{i}}(x')\right)^{2}}}\left(-1, \frac{\partial
\phi}{\partial X_{2}}(x'), \ldots, \frac{\partial \phi}{\partial X_{n}}(x')\right).
\end{equation*}
Taking the partial derivative with respect to $X_{i}$ of the $j$-th coordinate $g_{j}$ of $g$, $2 \leq i, j, \leq n$, and evaluating
at $p$ we obtain
\begin{equation*}
\frac{\partial g_{j}}{\partial X_{i}}(p) = \pm \frac{\partial^{2}\phi}{\partial X_{j} \partial X_{i}}(p'), 2 \leq i, j, \leq n.
\end{equation*}
The matrix $M=\left[\frac{\partial g_{j}}{\partial X_{i}}(p)\right]$, $2 \leq i, j \leq n$ is invertible %, and therefore the Gauss map is a diffeomorphism in a neighbourhood of $p$,
if and only if $p$ is a non-degenerate critical point of $\phi$ by Definition \ref{non-deg}.  If $M$ is invertible then it is non-zero, and $g$ is non-critical at $p$.
\end{proof}

\begin{lem}\label{orthChan}
There exists an orthogonal change of coordinates such that $\pi$ has only non-degenerate critical points.
\end{lem}

\begin{proof}
Consider the Gauss map $g$ from $W$ to the unit sphere;
we want $(\pm 1, 0, \ldots, 0)$ to be regular values of $g$.  If this is not the case, we find a rotation of coordinates as follows.
Sard's theorem tells us that the set of critical values of $g$ on the unit sphere has measure zero.  Take  $U$ to be the upper
hemisphere of the unit sphere, (note that $U$ does not have measure zero).  Consider two  subsets, one  $R$ containing regular
values of $g$ and one $C$ containing critical values, such that $U = R \cup C$.  Since a countable union of sets of measure zero also
has measure zero, we know that $R$ does not have measure zero.  Consider the set directly opposite $R$ on the lower hemisphere.  This
set also does not have measure zero, so contains at least one regular value of $g$.  Take this regular value, and the opposite point
on the upper hemisphere, which is also a regular value, and rotate the coordinates so that the points on $W$ that were taken by $g$
to these values are now taken to $(\pm 1, 0, \ldots, 0)$.

We can now find a rotation of coordinates to ensure $(\pm 1, 0, \ldots, 0)$ are not critical values of $g$.  The claim now
follows from Lemma \ref{GaussDif}.
\end{proof}

\begin{defn}
For a compact, non-singular hypersurface $W$ in $\mathbb{R}^{n}$, we
define $\chi$ to be the number of critical points of $\pi$ on $W$.
\end{defn}

We now only need to prove that, changing the coordinates if necessary, critical values are pair-wise distinct.
The first step in doing this uses our definition of complexity to find the number of critical points.

\begin{lem} \label{critComp}
Under the conditions of Lemma \textup{\ref{morse}}, suppose that $\pi$ has only non-degenerate critical points.
Then $\chi$ is finite, and bounded by $\gamma(n, c(f))$.
\end{lem}

\begin{proof}
Lemma \ref{orthChan} tells us that we may need to rotate the coordinate system to ensure critical points are non-degenerate.  If all critical points are non-degenerate without any rotation, the definition of a critical point means than at each point
 $\pd{f}{X_{2}} = \ldots = \pd{f}{X_{n}} = 0$.

However, as stated, the above workings may have required an orthogonal change of coordinates (i.e. multiplication of the coordinate system by an orthogonal matrix).  This means that each partial derivative of the new coordinate system can be expressed as a linear combination of partial derivatives of the old system.  As the matrix is orthogonal, it is invertible, and we can solve this system of $n-1$ equations to express each of $\pd{f}{X_{i}}$ for $i=2\ldots n$ as $\pd{f}{X_{1}}$ multiplied by a real number. In other words we have
\begin{align*}
\pd{f}{X_{2}} &= \lambda_{2} \pd{f}{X_{1}}\\
&\vdots\\
\pd{f}{X_{n}} &= \lambda_{n} \pd{f}{X_{1}},
\end{align*}
for $\lambda_{2} \ldots \lambda_{n} \in \R$.

In addition to these $n-1$ equations, we must also require $f=0$ (as the critical point is on the hypersurface).
We now have a system of $n$ equations, and each solution corresponds to a critical point of $\pi$.  As the critical points are
non-degenerate, every solution is non-singular, and we can use the definition of complexity to bound the number of solutions.

We have
\begin{center}
\begin{math}
\chi \leq t_{n}\left(c(f), c\left(\pd{f}{X_{2}} - \lambda_{2} \pd{f}{X_{1}}\right), \ldots, c\left(\pd{f}{X_{n}} - \lambda_{n}
\pd{f}{X_{1}}\right)\right).
\end{math}
\end{center}
\noindent Also, remembering that all constants have complexity $0$, we have
\begin{center}
\begin{math}
c_{j}\left(\pd{f}{X_{i}} - \lambda_{i} \pd{f}{X_{1}}\right) \leq t_{j+}(t_{j\partial}(c(F)), t_{j\times}(0, t_{j\partial}(c(F)))),
\end{math}
\end{center}
and hence we have
\begin{center}
\begin{math}
\chi \leq \gamma(n, c(f))
\end{math}
\end{center}

\end{proof}

Suppose that $\pi$ has only non-degenerate critical points.  Lemma \ref{critComp} tells us that these are finite in number.  We can
suppose that without loss of generality all the critical points have distinct $X_{2}$ coordinates, making if necessary an orthogonal
change of coordinates in the variables $X_{2}, \ldots, X_{n}$ only.

\begin{lem}
Let $\delta>0$ be sufficiently small.  The set $S$ of points $\bar{p} = (\bar{p_{1}}, \ldots, \bar{p_{n}}) \in W= \{f=0\}$ with
gradient vector $\text{Grad}(f)(\bar{p})$ proportional to $(1, \delta, 0, \ldots, 0)$ is finite.  Its number of elements is equal to
$\chi$, the number of critical points of $\pi$.  Moreover, there is a point $\bar{p}$ of $S$ infinitesimally close to every
critical point $p$ of $\pi$, and  $p$ and $\bar{p}$ are either both non-degenerate, or both degenerate.
\end{lem}

\begin{proof}
Note that, modulo the orthogonal change of variable
\begin{equation*}
X_{1}' = X_{1} + \delta X_{2}, X_{2}' = X_{2} - \delta X_{1}, X_{i}'=X_{i} \text{ for } i \geq 3,
\end{equation*}
 a point $\bar{p}$ such that $\text{Grad}(f)(\bar{p})$ is proportional to $(1, \delta, 0, \ldots, 0)$ is a critical point of the
 projection $\pi'$ onto the $X_{1}'$ axis, and the corresponding critical value of $\pi'$ is $\bar{p_{1}} + \delta \bar{p_{2}}$.

Since $W= \{f=0\}$ is bounded, a point $\bar{p} \in W$ always has an image by $\lim_{\delta \to 0}$.  If $\bar{p}$ is such that
$\text{Grad}(f)(\bar{p})$ is proportional to $(1, \delta, 0, \ldots, 0)$, then $\text{Grad}(f)(\lim_{\delta \to 0}(\bar{p}))$ is
proportional to $(1, 0, 0, \ldots, 0)$, and thus $p = \lim_{\delta \to 0}(\bar{p})$ is a critical point of $\pi$.  Suppose without loss of
generality that $\text{Grad}(f)(p) =(1, 0, 0, \ldots, 0)$.  Since $p$ is a non-degenerate critical point of $\pi$, Lemma
\ref{GaussDif} implies that there is a neighbourhood $U$ of $p' = (p_{2}, \ldots p_{n})$, such that $g \circ \Phi$ is a
diffeomorphism from $U$ to a neighbourhood of $(1, 0, 0, \ldots, 0) \in S^{n-1}(0,1)$.  Denoting by $g'$ the inverse of the
restriction of $g$ to $\Phi(U)$, and considering $g':g(\Phi(U)) \to \Phi(U)$, there is a unique $\bar{p} \in \Phi(U)$ such that
$\text{Grad}(f)(\bar{p})$ is proportional to $(1, \delta, 0, \ldots, 0)$.  Moreover, denoting by $J$ the Jacobian of $g'$, $J(1, 0,
0, \ldots, 0) = t$ is a non-zero real number.  Thus the signature of the Hessian (see Definition \ref{non-deg}) at $p$ and $\bar{p}$ coincide, and they are either both non-degenerate, or both degenerate.
\end{proof}

\begin{proof}[Proof of Lemma \ref{morse}]
Since $J$ is the Jacobian of $g'$, $J(1, 0, 0, \ldots, 0) = t$ is a non-zero real number, $\lim_{\delta \to 0}(J(y)) = t$ for every $y \in
S^{n-1}(0,1)$ infinitesimally close to $(1, 0, \ldots, 0)$.  Using the mean value theorem,
\begin{equation*}
o(|\bar{p} - p|) = o\left(\left|\frac{1}{\sqrt{1+\delta^{2}}}(1, \delta, 0, \ldots, 0) - (1, 0, 0, \ldots, 0)\right|\right) = 1.
\end{equation*}
Thus $o(\bar{p_{i}} - p_{i}) \geq 1$, $i \geq 1$.

Let $b_{i,j} = \frac{\partial^{2} \phi}{\partial X_{i} \partial X_{j}}(p)$, $2 \leq i \leq n$, $2 \leq j \leq n$.  Taylor's formula
for $p$ at $\phi$ gives
\begin{equation*}
\bar{p_{1}} = p_{1} + \sum_{2 \leq i \leq n, 2 \leq j \leq n} b_{i,j}(\bar{p_{i}} - p_{i})(\bar{p_{j}} - p_{j}) + c,
\end{equation*}
with $o(c) \geq 2$.  Thus $o(\bar{p_{1}} - p_{1}) \geq 2$.

It follows that the critical value of $\pi'$ at $\bar{p}$ is $\bar{p_{1}} + \delta \bar{p_{2}} = p_{1} + \delta p_{2} + w$, with
$o(w) > 1$.

Let us take two critical values $v = p_{1} + \delta p_{2} + w$, $v' = p_{1}' + \delta p_{2}' + w'$.  If $p_{1} \neq p_{1}'$, then we
can assume that $\delta$ is chosen to be small enough for $v \neq v'$.  If $p_{1} = p_{1}'$, we know that all values of $p_{2}$ are
distinct, and $w$ is negligible, so we also have $v \neq v'$.

Thus all critical values of $\pi'$ on $W$ are distinct.

\end{proof}

%%%%%%%%%%%%%%%%%%%%%%%%%%%%%%%%%%%%%%%%%%%%%%%%%%%%%%%%%%%%%%%%%%%%%%%%%%%%%%%%%%%%%%%%
\subsection{A Bounding Hypersurface}
%%%%%%%%%%%%%%%%%%%%%%%%%%%%%%%%%%%%%%%%%%%%%%%%%%%%%%%%%%%%%%%%%%%%%%%%%%%%%%%%%%%%%%%%
Before proving Theorem \ref{main}, we need a few preliminary results.

Firstly we follow the method used in \cite{AlgInRAG} to bound the sum of the Betti numbers of a compact non-singular hypersurface in
terms of the number of critical points of the projection.

%Intuitively: imagine moving along the $X_{1}$-axis, passing through the critical points of $\pi$.  We can visualise these points as
%when the bounding hypersurface ``splits'' or ``joins''.  Between the critical points the surface can ``stretch'' or ``shrink'', but
%not change homotopy type.  This means that each critical point corresponds directly to where the ``complexity'' of the surface
%``changes''.  This is made more precise below.

\begin{lem} \label{eqn}
If $W=\{f=0\}$ is a compact, non-singular hypersurface, and $\pi$ is a Morse function, then the sum of the Betti numbers of $W$ is bounded by $\chi$.
\end{lem}

\begin{proof}
Let $v_{1}, \ldots v_{l}$ be the critical values of $\pi$, and $p_{1} < p_{2} < \ldots < p_{l}$ the corresponding critical points
such that $\pi(p_{i}) = v_{i}$.  We will show that $b(W_{\leq v_{i}}) \leq i$.

We move along the $X_{1}$-axis.  First note that $W_{\leq v_{1}} = \{p_{1}\}$, and hence $b(W_{\leq v_{1}}) = 1$.  Morse Theory (for
details see \cite{DiffTop}, \cite{AlgInRAG}) tells us that between critical points of $\pi$, homotopy type doesn't change, i.e.
$W_{\leq v_{i}+ \varepsilon}$ is homotopy equivalent to $W_{\leq v_{i+1}-\varepsilon}$ for any small enough $\varepsilon > 0$. We
thus only need to consider what happens when we pass through critical points.  Morse Theory also tells us that the homotopy  type of
$W_{\leq v_{i} + \varepsilon}$ is that of the union of $W_{\leq v_{i} - \varepsilon}$ with a topological ball.

From the Mayer-Vietoris Inequalities \eqref{MV2} we have that if $S_{1}$ and $S_{2}$ are two closed sets with non-empty intersection, and if $i >0$
then
\begin{equation*}
b_{i}(S_{1}\cup S_{2}) \leq b_{i}(S_{1}) + b_{i}(S_{2}) + b_{i-1}(S_{1}\cap S_{2}).
\end{equation*}
For a closed and bounded set $S$, $b_{0}(S)$ is equal to the number of connected components of $S$.  Since $S_{1}\cap S_{2} \neq
\emptyset$, for $i=0$ we have
\begin{equation*}
b_{0}(S_{1}\cup S_{2}) \leq b_{0}(S_{1}) + b_{0}(S_{2}) -1.
\end{equation*}
We also have, for $\lambda>1$
\begin{align*}
b_{0}(\overline{B_{\lambda}}) = b_{0}(S^{\lambda -1}) = b_{\lambda-1}(S^{\lambda-1})=1,
\end{align*}
and for $0 < i < \lambda-1$ we have
\begin{align*}
b_{i}(S^{\lambda-1}) = 0.
\end{align*}
Therefore, for $\lambda>1$, attaching a $\lambda$-ball can increase $b_{\lambda}$ by at most one, and none of the other Betti numbers
can increase.

For $\lambda=1$, $b_{\lambda-1}(S^{\lambda-1})=b_{0}(S^{0})=2$.  In this case, $b_{1}$ can increase by at most one --  adding an edge
to a graph can increase the number of cycles by at most one, and therefore adding $\overline{B_{1}}$ to a set can increase the number
of ``holes'', i.e. $b_{1}$ by $1$.  For $i>1$, $b_{i-1}(S \cap \overline{B_{1}}) = 0$, and $b_{i}(\overline{B_{1}})=0$, so $b_{i}(S
\cup \overline{B_{1}}) \leq b_{i}(S)$, and none of the other Betti numbers can increase.

Thus $b(W_{\leq v_{i+1}}) \leq b(W_{\leq v_{i}}) + 1$, and the result follows.
\end{proof}

We again closely follow the methods used in \cite{AlgInRAG}, this time to bound the sum of the Betti numbers of the set bounded by
the above hypersurface.

%Again, intuitively: as we move along the $X_{1}$-axis, a critical point of $\pi$ corresponds to the bounded set either ``splitting''
%or ``joining''.  Now, intuitively speaking, ``splitting'' cannot increase Betti numbers/change homotopy type, as each of the split
%branches can be ``shrunk'' back into the shape, whereas ``joining'' can create a ``hole'', and thus change complexity.  We can
%reverse the direction of the axis to make at most half the critical points of the ``joining'' type.  Precisely:

\begin{lem} \label{boundary}
If $W=\{f=0\}$ is again as above, $K = \{f\geq0\}$ is bounded by $W$, and $\pi$ is a Morse function, then the sum of the Betti
numbers of $K$ is bounded by $\chi / 2$.
\end{lem}

\begin{proof}
As above, let $v_{1}, \ldots v_{l}$ be the critical values of $\pi$, and $p_{1} < p_{2} < \ldots < p_{l}$ the corresponding critical
points such that $\pi(p_{i}) = v_{i}$.  Let $J$ be the subset of $\{1, \ldots, l\}$ such that the direction of $Grad(f)(p_{i})$
belongs to $K$.  We will show that $b(K_{\leq v_{i}})$ is less than or equal to the number of $j \in J$ with $j \leq i$.

First note that $K_{\leq v_{1}} = \{p_{1}\}$, and hence $b(K_{\leq v_{1}}) = 1$.  Morse Theory (as above) tells us that $K_{\leq
v_{i+1}-\varepsilon}$ is homotopic to $K_{v_{i} + \varepsilon}$ for any small enough $\varepsilon > 0$, and thus $b(K_{\leq
v_{i+1}-\varepsilon}) = b(K_{v_{i} + \varepsilon})$.

Morse Theory also tells us that $K_{\leq v_{i} + \varepsilon}$ has the same homotopy type as $K_{\leq v_{i} - \varepsilon}$ if $i
\notin J$, and has the homology type of the union of $K_{\leq v_{i} - \varepsilon}$ with a topological ball if $i \in J$ (see
\cite{AlgInRAG}, Proposition 7.20).  It follows that
\begin{align*}
b(K_{\leq v_{i} + \varepsilon}) =
\begin{cases}
b(K_{\leq v_{i} - \varepsilon}) &\text{if }i \notin J\\
b(K_{\leq v_{i} - \varepsilon}) + 1 &\text{if }i \in J.
\end{cases}
\end{align*}

By switching the direction of the $X_{1}$-axis if needed, we can always ensure that the number of points in $J$ is at most $\chi/2$.
\end{proof}

%%%%%%%%%%%%%%%%%%%%%%%%%%%%%%%%%%%%%%%%%%%%%%%%%%%%%%%%%%%%%%%%%%%%%%%%%%%%%%%%%%%%%%%%
\subsection{Proof of main result}
%%%%%%%%%%%%%%%%%%%%%%%%%%%%%%%%%%%%%%%%%%%%%%%%%%%%%%%%%%%%%%%%%%%%%%%%%%%%%%%%%%%%%%%%
We can summarise three preceding results as follows:
\begin{cor}
From Lemmas  \textup{\ref{critComp}}, \textup{\ref{eqn}}, and \textup{\ref{boundary}} we now have the following:
\begin{align*}
b(W) &\leq \chi \\
b(K) &\leq \chi/2 \\
\chi &\leq \gamma(n,c(f))
\end{align*}
and can deduce
\begin{align*}
b(W) &\leq \gamma(n,c(f)) \\
b(K) &\leq \gamma(n,c(f)) / 2.
\end{align*}
\end{cor}

We can now prove the main result of this section, again largely following the methods used in \cite{AlgInRAG}.

\begin{proof}[Proof of Theorem \ref{main}]
Since $S$ is a definable set, by Theorem \ref{conInf} we can choose a large $r$ such that $r^{2} - |x|^{2}$ is always strictly positive for $x \in S$, and the
sum of the Betti numbers of $S \cap B(0,r)$ coincides with that of $S$.

Let $G(x) = \frac{f_{1}^{2}+ \ldots +f_{m}^{2}}{r^{2} - |x|^{2}}$.  By Sard's theorem, the set of critical values of $G$ is finite.
Hence there is a positive $a \in \mathbb{R}$ so that no $b \in (0,a)$ is a critical value of $G$, and thus the set $W_{b} = \{x \in
\mathbb{R}^{n}:f_{1}^{2}+ \ldots +f_{m}^{2}+b(|x|^{2}-r^{2})=0\}$ is a non-singular hypersurface in $\mathbb{R}^{n}$.  To see this,
setting $F(x,b) = f_{1}^{2}+ \ldots +f_{m}^{2}+b(|x|^{2}-r^{2})$, we have that if  for some $x$,
\begin{center}
$F(x,b) = \pd{F}{X_{1}}(x,b) = \ldots = \pd{F}{X_{n}}(x,b) = 0$
\end{center}
then
\begin{center}
$G(x) = b$ and $\pd{G}{X_{1}}(x) = \ldots = \pd{G}{X_{n}}(x) = 0$,
\end{center}
 which means that $b$ is a critical value of $G$, a contradiction.

Moreover, $W_{b}$ is the boundary of the closed and bounded set $K_{b} = \{x \in \mathbb{R}^{n}:F(x,b)\leq0\}$.  By Lemma \ref{eqn}
the sum of the Betti numbers of $W_{b}$ is less than or equal to $\gamma(n,c(F))$, and using Lemma \ref{boundary} the sum of the
Betti numbers of $K_{b}$ is at most half that of $W_{b}$.

We now claim that $S \cap B(0,r)$ is homotopy equivalent to $K_{b}$ for all small enough $b>0$.  We replace $b$ in the definition by
a new variable $t$, and consider the set $K \subset \mathbb{R}^{n+1}$ defined by $\{(x,t)\in\mathbb{R}^{n+1}:F(x,t)\leq0\}$.  Let
$\pi_{X}$ (respectively $\pi_{T}$) denote the projection map onto the $x$ (respectively $t$) coordinates.

Clearly, $S \cap B(0,r) \subset K_{b}$.  By Hardt's triviality (Theorem \ref{Hardt}), for all small enough $b>0$, there exists a
homeomorphism $\phi:K_{b}\times(0,b] \to K\cap \pi^{-1}_{T}((0,b])$, such that $\pi_{T}(\phi(x,s))=s$ and $\phi(S \cap B(0,r),s) = S
\cap B(0,r)$ for all $s \in (0,b]$.

Let $H:K_{b} \times [0,b] \to K_{b}$ be the map defined by $H(x,s) = \pi_{X}(\phi(x,s))$ for $s>0$, and $H(x,0) = \lim_{s\to 0
+}\pi_{X}(\phi(x,s))$.  Let $h:K_{b} \to S \cap B(0,r)$ be the map $H(x,0)$, and $i:S \cap B(0,r) \to K_{b}$ be the inclusion map.
Using the homotopy $H$, we see that
$i \circ g \sim Id_{K_{b}}$, and
$g \circ i \sim Id_{S \cap B(0,r)}$, which shows that $S \cap B(0,r)$ is homotopy equivalent to $K_{b}$ as claimed.

Hence,
\begin{center}
$b(S \cap B(0,r)) = b(K_{b}) \leq \frac{1}{2}b(W_{b}) \leq \frac{\gamma(n,c(F))}{2}$.
\end{center}

Since $b$ and $r$ are constants, we have $c(f_{1}^{2}+ \ldots +f_{m}^{2}+b(|x|^{2}-r^{2})) = c(f_{1}^{2}+ \ldots +f_{m}^{2}+|x|^{2})$, and produce the desired result.

\end{proof}

%%%%%%%%%%%%%%%%%%%%%%%%%%%%%%%%%%%%%%%%%%%%%%%%%%%%%%%%%%%%%%%%%%%%%%%%%%%%%%%%%%%%%%%%
%%%%%%%%%%%%%%%%%%%%%%%%%%%%%%%%%%%%%%%%%%%%%%%%%%%%%%%%%%%%%%%%%%%%%%%%%%%%%%%%%%%%%%%%
\section{Sets defined by multiple non-strict inequalities}
%%%%%%%%%%%%%%%%%%%%%%%%%%%%%%%%%%%%%%%%%%%%%%%%%%%%%%%%%%%%%%%%%%%%%%%%%%%%%%%%%%%%%%%%
%%%%%%%%%%%%%%%%%%%%%%%%%%%%%%%%%%%%%%%%%%%%%%%%%%%%%%%%%%%%%%%%%%%%%%%%%%%%%%%%%%%%%%%%

We next move to the case where our set is defined by multiple equations and non-strict inequalities.  The next Lemma follows the
method shown in \cite{Milnor} to construct a set defined in terms of a collection of non-strict inequalities, and uses the method of
proof of homotopy equivalence shown in Proposition 7.28 of \cite{AlgInRAG}.

\begin{lem}
Let  $f_{1}, f_{2}, \ldots, f_{p}$ be functions  in $D$, and let $S = \{f_{1} \geq 0, f_{2} \geq 0, \ldots, f_{p} \geq 0\}$.  Then the sum
of the Betti numbers of $S$ $b(S)$ satisfies
\begin{center}
 $b(S) \leq \dfrac{\gamma(n, c(|x|^{2}f_{1}f_{2}\cdots f_{p}))}{2}$.
\end{center}
\end{lem}

\begin{proof}
The set $S$ is homotopy equivalent to $S \cap B(0,r)$ for sufficiently large $r$, so we can add to the system the inequality $f_{0} =
r^{2} - |x|^{2} \geq 0$ to allow us to assume that the set is compact.  For $\varepsilon \geq \delta > 0$, consider the set
$L_{\varepsilon, \delta}$  defined by
\begin{align}
f_{0} + \varepsilon \geq 0, \hdots , f_{p} + \varepsilon \geq 0
\end{align}
and
\begin{align}
(f_{0} + \varepsilon)(f_{1} + \varepsilon)\hdots(f_{p} + \varepsilon) - \delta \geq 0. \label{boundaryEqn}
\end{align}

This set is compact, and we obtain the boundary by setting the left hand side of (\ref{boundaryEqn}) equal to zero. Given $\varepsilon$ we can choose
$\delta$ so that the bounding surface is smooth, and therefore we can use Lemma \ref{eqn} to bound the sum of the Betti numbers of
this bounding surface, and use Lemma \ref{boundary} to bound the sum of the Betti numbers of $L_{\varepsilon, \delta}$.

Setting $M_{\varepsilon, \delta }= (f_{0} + \varepsilon)(f_{1} + \varepsilon)\hdots(f_{p} + \varepsilon) - \delta$,
we have

\begin{center}
$b(L_{\varepsilon, \delta}) \leq \dfrac{\gamma(n, c(M_{\varepsilon, \delta}))}{2}$.
\end{center}

We can use the definition of the complexity $c$ to find the complexity $c(M_{\varepsilon, \delta})$.  From $(c1)$ we know the
complexity of an expression does not depend on the value of constants, so $c(M_{\varepsilon, \delta})$ does not depend on the values
of $\varepsilon$ and $\delta$ and we can rename all instances of this $c(M)$.

It is clear that $S \cap B(0,r) \subset L_{\varepsilon, \delta}$.  Set
\begin{align*}
L = \{(x, \varepsilon, \delta) \in \mathbb{R}^{n+2}:&\\
f_{0}(x) + \varepsilon \geq 0,  \ldots , f_{p}(x)& \geq 0, (f_{0}(x) + \varepsilon)(f_{1}(x) + \varepsilon) \hdots (f_{p}(x) +
\varepsilon) - \delta \geq 0\}.
 \end{align*}

Let $\pi_{X}$, $\pi_{E}$, $\pi_{\Delta}$ be the projection onto the $x$, $\varepsilon$ and $\delta$ coordinates respectively.  By
Hardt's triviality (Theorem \ref{Hardt}) there exists a homeomorphism
\begin{align*}\phi: L_{\varepsilon, \delta} \times (0, \varepsilon] \times (0, \delta] \to L \cap \pi^{-1}_{E}((0,\varepsilon]) \cap
\pi^{-1}_{\Delta}((0, \delta])\end{align*}
such that
$\pi_{E}(\phi(x, a, b)) = a$,
$\pi_{\Delta}(\phi(x, a, b)) = b$,
and
\begin{align*}\phi(S \cap B(0,r), a, b) = S \cap B(0,r)\end{align*}
for all $a \in (0,\varepsilon]$, $b \in (0, \delta]$.

Let
\begin{align*} G: L_{\varepsilon, \delta} \times [0, \varepsilon] \times [0, \delta] \to L_{\varepsilon, \delta} \end{align*}
be the map defined by

\begin{equation*}G(x, a, b) =
\begin{cases}
\pi_{X}(\phi(x, a, b)) & \text{if $a, b > 0$,} \\
\lim_{a \to 0^{+}} \pi_{X}(\phi(x, a, b)) &\text{if $a = 0, b > 0$,}\\
\lim_{b \to 0^{+}} \pi_{X}(\phi(x, a, b)) &\text{if $a > 0, b = 0$,}\\
\lim_{a, b \to 0^{+}} \pi_{X}(\phi(x, a, b)) &\text{if $a, b =0$.}
\end{cases}
\end{equation*}

Let $g: L_{\varepsilon, \delta} \to S \cap B(0,r)$ be the map $G(x, 0, 0)$, and $i: S \cap B(0,r) \to L_{\varepsilon, \delta}$ be the
inclusion map.  We can see that $i \circ g \sim Id_{L_{\varepsilon, \delta}}$, and $g \circ
 i \sim Id_{S \cap B(0,r)}$, so $S \cap B(0,r)$ is homotopy equivalent to $L_{\varepsilon, \delta}$ as desired.

Hence:
\begin{align*}b(S) = b(S \cap B(0,r)) = b(L_{\varepsilon, \delta})\leq \dfrac{\gamma(n, c(M))}{2}.\end{align*}

Finally we can simplify the resulting expression by removing constants.

\end{proof}

\begin{eg}[The Degree of a Polynomial]
Using the complexity described in Example \textup{\ref{deg eg}}, if the functions $f_{i}$ have maximum degree $d$, then the function $|x|^{2}f_{1}\cdots f_{p}$ has maximum degree $2+pd$, and $\gamma = (2+pd)(1+pd)^{n-1}$, and $b(S) \leq (2+pd)(1+pd)^{n-1}/2$,which is the same result shown in Theorem 3 of \cite{Milnor}.
\end{eg}

\begin{eg}[Pfaffian Functions]
Consider the setting of Example \ref{pfaffEg}, and allow  functions $f_{i}$ to have degree $(\alpha_{i}, \beta_{i}, r)$.
Set $\text{max}_{i}\{\alpha_{i}\} = \alpha$, $\text{max}_{i}\{\beta_{i}\} = \beta$.
 Then
 \begin{align*}
 c(|x|^{2}f_{1}\cdots f_{p}) =(\alpha, 2 + \sum_{1 \leq i \leq p}\beta_{i}, r),
 \end{align*}
 and
 \begin{align*}
 b(S)\leq &\gamma/2\\
 \leq &2^{r(r-1)/2-1}(2 + p\beta)(\alpha+2 + p\beta-1)^{n-1}(\text{min}\{n,r\}\alpha + n(2 + p\beta) + (n-1)\alpha-2n+2)^{r}\\
 = &2^{r(r-1)/2-1}(2 + p\beta)(\alpha+1 + p\beta)^{n-1}(\text{min}\{n,r\}\alpha  + np\beta + (n-1)\alpha+2)^{r}.\\
 \end{align*}
This is equivalent to results given in \cite{Zell}.
\end{eg}

It is easy to see that any collection of equations can be replaced by two non-strict inequalities, by setting $F = f_{1}^{2} + \ldots
+ f_{m}^{2} = 0$ and considering $F\geq 0$ and $-F\geq 0$.  We can now combine this observation with the above result to get a
theorem considering the union of any collection of equations and non-strict inequalities.

\begin{thm}
Let $S = \{{f_{1}=0, \ldots, f_{m} = 0, f_{m+1} \geq 0, \ldots, f_{p} \geq 0}\}$, where $f_{i} \in D$.  Then we have
\begin{align*}
b(S) \leq \frac{1}{2} \gamma\left(n, c(|x|^{2}(f_{1}^{2} + \ldots + f_{m}^{2})^{2}f_{m+1}\ldots f_{p})\right).
\end{align*}
\end{thm}

\begin{proof}
Follows from above discussion.
\end{proof}

%%%%%%%%%%%%%%%%%%%%%%%%%%%%%%%%%%%%%%%%%%%%%%%%%%%%%%%%%%%%%%%%%%%%%%%%%%%%%%%%%%%%%%%%
%%%%%%%%%%%%%%%%%%%%%%%%%%%%%%%%%%%%%%%%%%%%%%%%%%%%%%%%%%%%%%%%%%%%%%%%%%%%%%%%%%%%%%%%
\chapter{Sets defined by a Boolean combination of functions}\label{chBool}
%%%%%%%%%%%%%%%%%%%%%%%%%%%%%%%%%%%%%%%%%%%%%%%%%%%%%%%%%%%%%%%%%%%%%%%%%%%%%%%%%%%%%%%%
%%%%%%%%%%%%%%%%%%%%%%%%%%%%%%%%%%%%%%%%%%%%%%%%%%%%%%%%%%%%%%%%%%%%%%%%%%%%%%%%%%%%%%%%

\section{Preliminaries}
The following notation is taken from \cite{SignCon}.
Let $G$ and $F$ be finite sets of definable functions from $H \to \mathbb{R}$, where $H \subset \mathbb{R}^{n}$ is a definable open set.  A sign condition on $F$ is an element of
$\{0,1,-1\}^{F}$. Let $F$ have $s$ elements.

We define the sets $Z$ and $Z_{r}$ by
\begin{center}
$\displaystyle Z  = \left\{x \in \mathbb{R}^{n}: \bigwedge_{g \in G} g(x) = 0\right\}$,
\end{center}
\begin{center}
$\displaystyle Z_{r} = Z \cap B(0,r)$.
\end{center}
Let the dimension of $Z$ be $n'$.

The realisation of the sign condition $\sigma$ over $Z$, $\textit{Reali}(\sigma, Z)$ is the definable set
\begin{center}
$\displaystyle \left\{x \in \mathbb{R}^{n}: \bigwedge_{g \in G} g(x)=0 \land \bigwedge_{f \in F}sign(f(x))=\sigma(f)\right\}$.
\end{center}

For each $0 \leq i \leq n'$, we denote the sum of the $i$-th Betti numbers of the relisations of all sign conditions of $F$ on $Z$ by
$b_{i}(F,G)$.

\subsection{The Function $\Omega$}

\begin{defn}\label{omega}
Let
\begin{align*}
\displaystyle \omega(F,G) &= \max_{i}\left(b(\textit{Reali}(f_{i}=\delta, Z_{r}))\right)\\
 &= \max_{i}\left(\frac{\gamma(n, c(f_{i}^2 + \sum_{g \in G} g^{2} + |x|^{2}))}{2}\right);
\end{align*}
we then define $\Omega(F,G)$ to be
\begin{center}
$\displaystyle\max(\omega(F,G), b(Z_{r})) = \max\left(\omega(F,G), \frac{\gamma(n, c(\sum_{g \in G} g^{2} + |x|^{2} )}{2}\right)$.
\end{center}
\end{defn}

\begin{eg}[The Degree of a Polynomial]
Again in the situation of Example \ref{deg eg}, with the degree of the functions in $F$ and $G$ being bounded by $d$, we have
\begin{align*}
\Omega(F,G) =& \frac{\gamma(n, 2d)}{2}\\
=& d(2d-1)^{n-1}
\end{align*}
\end{eg}

\begin{eg}[Pfaffian Functions]
Let the functions in $F$ and $G$ have a common Pfaffian chain of order $r$.  For $h \in F, G$, let $\text{max}\{c_{1}(h)\}=\alpha$, $\text{max}\{c_{2}(h)\}=\beta$.  Then
\begin{align*}
\Omega(F,G) =& \frac{\gamma(n, (\alpha, 2\beta, r))}{2}\\
=& 2^{r(r-1)/2}\beta(\alpha+2\beta-1)^{n-1}(\text{min}\{n,r\}\alpha + 2n\beta + (n-1)\alpha-2n+2)^{r}\\
\end{align*}
\end{eg}

\subsection{Main Result}

Our main result in this chapter is this:
\begin{thm}
Let $Y\subset \R^{n}$ be a set defined by a Boolean combination of definable functions $\{f_{1}, \ldots, f_{s}\}$.  Then
\begin{equation*}
b(Y) \leq \sum_{i=0}^{n}\sum_{j=1}^{n-i} \binom{2s^{2}+1}{j}6^{j}\Omega(F',\emptyset)
\end{equation*}
where $F'$ is a set of functions containing all the elements $f_{i}$ and $|x|^{2}$.
\end{thm}
We initially consider the realisation of sign conditions of sets of functions, and use this to show how to bound the Betti numbers of
closed sets.  We then use an inductive construction to find a closed set from any set defined by a Boolean combination of functions,
and prove that in fact this closed set we have found is homotopy equivalent to our initial set.  We can then use our bound for closed
sets to produce the above bound for arbitrary Boolean combinations.
The following sections owe much to \cite{SignCon}, \cite{BettiSA} and \cite{AlgInRAG}.

%%%%%%%%%%%%%%%%%%%%%%%%%%%%%%%%%%%%%%%%%%%%%%%%%%%%%%%%%%%%%%%%%%%%%%%%%%%%%%%%%%%%%%%%
\section{Realisation of Sign Conditions}
%%%%%%%%%%%%%%%%%%%%%%%%%%%%%%%%%%%%%%%%%%%%%%%%%%%%%%%%%%%%%%%%%%%%%%%%%%%%%%%%%%%%%%%%

Our main result in this section is the following:
\begin{thm}\label{realSigCon}
\begin{equation*}
\displaystyle b_{i}(F,G) \leq \sum_{j=0}^{n'-i}\binom{s}{j}4^{j}\Omega(F,G).
\end{equation*}
\end{thm}

To prove this result,  we define a set $W_{0}$ to be the union of some sets we can easily bound, and bound the Betti numbers
of $W_{0}$ using the Mayer-Vietoris Inequalities.  We define a set $W_{1}$ that has $W_{0}$ as its boundary, and use
the Mayer-Vietoris inequalities and the previous result to bound the Betti numbers of $W_{1}$.  Finally we relate these sets back to
$b_{i}(F,G)$, and prove the main result.  These methods are the same as those used in \cite{SignCon}, but with the generalisation to
functions defined in o-minimal structures (instead of simply polynomials), which results in replacing the expression $d(2d-1)^{n-1}$
with $\Omega(F,G)$.

%Let $S_{1},\ldots,S_{s}\subset\mathbb{R}^{n}$ be closed o-minimal sets, contained in a closed o-minimal set $T$ of dimension $n'$.
%For  $J \subset \{1, \ldots, s\}$, $J \neq \emptyset$, let
%\begin{center}
%$\displaystyle S_{J} = \bigcap_{j \in J} S_{j}$, $\displaystyle S^{J} = \bigcup_{j \in J} S_{j}$,
%\end{center}
%and let $S^{\emptyset}=T$.

\subsection{Sign Conditions}

The following lemmas also follow the method of proof shown in \cite{SignCon}, this time adapting for our more general circumstances.

Set $\displaystyle W_{0} = \bigcup_{1 \leq i \leq j}(\textit{Reali}(f_{i}^{2}(f_{i}^{2}-\delta^2)=0, Z_{r}))$ for sufficiently small
$\delta$.

\begin{lem} \label{w0}
$\displaystyle b_{i}(W_{0}) \leq (4^{j}-1)\Omega(F,G)$
\end{lem}

\begin{proof}
Each of the sets $\textit{Reali}(f_{i}^{2}(f_{i}^{2}-\delta^2)=0, Z_{r})$ is the disjoint union of three sets defined by definable
equations, namely $\textit{Reali}(f_{i}=0, Z_{r})$, $\textit{Reali}(f_{i}=\delta, Z_{r})$, and $\textit{Reali}(f_{i}=-\delta, Z_{r})$.  Moreover,
from the Mayer-Vietoris Inequalities (Lemma \ref{MVGen}), each of the Betti numbers of their union is bounded by the sum of the Betti numbers of all possible
non-empty sets that can be obtained by taking, for $1 \leq l \leq j$, $l$-ary intersections of these sets.  The number of possible
$l$-ary intersections is $\binom{j}{l}$.  Each such intersection is a disjoint union of $3^{l}$ sets.  The sum of the Betti numbers
of each of these sets is bounded by $\Omega(F,G)$ by Definition \ref{omega}, thus
\begin{center}
$\displaystyle b_{i}(W_{0}) \leq \sum_{l=1}^{j}\binom{j}{l}3^{l}\Omega(F,G) = (4^{j}-1)\Omega(F,G)$
\end{center}
\end{proof}

Now set $\displaystyle W_{1} = \bigcup_{1 \leq i \leq j}(\textit{Reali}(f_{i}^{2}(f_{i}^{2}-\delta^2)\geq 0, Z_{r}))$ for sufficiently
small $\delta$.

\begin{lem} \label{w1}
$\displaystyle b_{i}(W_{1}) \leq (4^{j}-1)\Omega(F,G)+b_{i}(Z_{r})$
\end{lem}

\begin{proof}
Let $Q_{i} = f_{i}^{2}(f_{i}^{2}-\delta^{2})$, and
\begin{center}
$\displaystyle K = \textit{Reali}\left(\bigwedge_{1 \leq i \leq j}(Q_{i} \leq 0) \lor \bigvee_{1 \leq i \leq j}(Q_{i}=0), Z_{r}\right)$.
\end{center}
Now applying inequality (\ref{MV1}), noting that $W_{1} \cup K = Z_{r}$ and $W_{1} \cap K = W_{0}$.  We get that $b_{i}(W_{1}) \leq
b_{i}(W_{1} \cap K) + b_{i}(W_{1} \cup K) = b_{i}(W_{0}) + b_{i}(Z_{r})$.  We conclude using Lemma \ref{w0}.
\end{proof}

Now let $S_{i} = \textit{Reali}(f_{i}^{2}(f_{i}^{2}-\delta^2)\geq 0, Z_{r})$, and $S$ be the intersection of $S_{i}$.  Then

\begin{lem}
$b_{i}(F,G) = b_{i}(S)$.
\end{lem}

\begin{proof}
Consider a sign condition $\sigma$ on $F$ such that, without loss of generality,
\begin{align*}
\sigma(f_{i}) &= 0 \text{    if } i = 1, \ldots, j,\\
\sigma(f_{i}) &= 1 \text{    if } i = j+1, \ldots, l,\\
\sigma(f_{i}) &= -1 \text{    if } i = l+1, \ldots, s
\end{align*}
and denote by $\overline{\textit{Reali}}(\sigma)$ the subset of $Z_{r}$ defined by
\begin{center}
$\displaystyle \bigwedge_{i=1, \ldots, j}f_{i}(x)=0\land \bigwedge_{i=j+1, \ldots, l}f_{i}(x)\geq \delta \land \bigwedge_{i=l+1,
\ldots, s}f_{i}(x)\leq -\delta$
\end{center}
It follows from Hardt's triviality (Theorem \ref{Hardt}) that $b_{i}(\overline{\textit{Reali}}(\sigma)) = b_{i}(\sigma)$, with $\delta$
sufficiently small.  Note that $S$ is the disjoint union of the $\overline{\textit{Reali}}(\sigma)$ (for the $\sigma$ realisable sign
conditions) so that $\sum_{\sigma}b_{i}(\sigma) = b_{i}(S)$.  On the other hand, $\sum_{\sigma}b_{i}(\sigma) = b_{i}(F, G)$.
\end{proof}

We now can prove our most substantial result

\begin{proof}[proof of Theorem \ref{realSigCon}]
Using the Mayer-Vietoris Inequalities (Lemma \ref{MVGen}), Lemma \ref{w1}, and Lemma \ref{eqn} which implies, for all $i < n'$, $b_{i}(Z_{r}) + b_{n'}(Z_{r})
\leq b(Z_{r}) \leq \Omega(F,G)$, we find that
\begin{align*}
\displaystyle b_{i}(S) &\leq b_{n'}(S^{\emptyset}) + \sum_{j=1}^{n'-i} \sum_{J \subset \{1, \ldots, s\}, \#(J) =
j}(4^{j}-1)\Omega(F,G) + b_{i}(Z_{r}) + b_{n'}(Z_{r})\\
&\leq b_{n'}(S^{\emptyset}) + \sum_{j=1}^{n'-i} \sum_{J \subset \{1, \ldots, s\}, \#(J) = j}4^{j}\Omega(F,G)\\
&\leq \sum_{j=0}^{n'-i}\binom{s}{j}4^{j}\Omega(F,G)
\end{align*}
\end{proof}

%%%%%%%%%%%%%%%%%%%%%%%%%%%%%%%%%%%%%%%%%%%%%%%%%%%%%%%%%%%%%%%%%%%%%%%%%%%%%%%%%%%%%%%%
\section{Closed Sets}
%%%%%%%%%%%%%%%%%%%%%%%%%%%%%%%%%%%%%%%%%%%%%%%%%%%%%%%%%%%%%%%%%%%%%%%%%%%%%%%%%%%%%%%%
In this section we prove a bound for the Betti numbers of closed (and by Alexander's duality open) sets.  We first define closed
formulae (which are used to describe closed sets), and proceed to define sets based on these sets, slightly modified by
small real numbers.  We show how one can bound the Betti numbers of the conjunction of two formulae by the sum of the Betti numbers of
the conjunction of these with a third formula that ranges over a given set, and then use this to show how to bound the Betti numbers
of a formula in terms of the sum of the Betti numbers of some sets that are in some sense a subset of the original set.  We next
define sets $W_{0}$ and $W_{1}$ similar to those in the above section, and proceed to bound the Betti numbers of these sets.  These
are then shown to be exactly related to the bounds we have found for closed formulae above, and a bound for any closed set easily
follows.  Again much of the following is taken from \cite{SignCon}, with the same modification as above, again resulting in the
introduction of $\Omega(F,G)$.  We show (in an example) how this bound reduces to that given in \cite{SignCon}, in the polynomial,
degree case, and show how a parallel to $s^{n'}O(d)^{n}$ can be found.

An $(F,G)$-closed formula is defined as follows:
\begin{itemize}
\item For each $f \in F$, $\bigvee_{g \in G} g = 0 \land f = 0$, $\bigvee_{g \in G} g = 0 \land f \geq 0$, $\bigvee_{g \in G} g =
    0 \land f \leq 0$ are $(F,G)$-closed formulae.
\item If $\phi_{1}$ and $\phi_{2}$ are $(F,G)$-closed formulae, then $\phi_{1} \land \phi_{2}$ and $\phi_{1} \lor \phi_{2}$ are
    $(F,G)$-closed formulae.
\end{itemize}
We denote by $b(\phi)$ the maximum sum of the Betti numbers of any realisation of this formula.

Let $F = \{f_{1}, \ldots f_{s}\}$ be o-minimal functions, and let $\delta_{1}, \ldots, \delta_{s}$ be sufficiently small real numbers, with
$\delta_{i+1} < \delta_{i}$ .  Define the following:
\begin{align*}
F_{>i} &= \{f_{i+1},\ldots,f_{s}\}\\
\Sigma_{i} &= \{f_{i}=0, f_{i} = \delta_{i}, f_{i} = -\delta_{i}, f_{i} \geq 2\delta_{i}, f_{i} \leq -2\delta_{i}\}\\
\Sigma_{\leq i} &= \{\psi:\psi = \bigwedge_{j=1,\ldots,i}\psi_{i}, \psi_{i} \in \Sigma_{i}\}
\end{align*}

\begin{lem} \label{greek1}
For every $(F,G)$-closed formula $\phi$, and every $\psi \in \Sigma_{\leq i}$,
\begin{equation*}
b(\phi \land \psi) \leq \sum_{\pi \in \Sigma_{i+1}}b(\phi \land \psi \land \pi).
\end{equation*}
\end{lem}

\begin{proof}
Consider
\begin{align*}
\phi_{1} &= \phi \land \psi \land (f_{i+1}^{2} - \delta_{i+1}^{2} \geq 0)\\
\phi_{2} &= \phi \land \psi \land (0 \leq f_{i+1}^2 \leq \delta_{i+1}^{2}).
\end{align*}
Clearly a realisation of $\phi \land \psi$ is also a realisation of $\phi_{1} \lor \phi_{2}$.  From inequality (\ref{MV2}), we then
have
\begin{center}
$b(\phi \land \psi) \leq b(\phi_{1}) + b(\phi_{2}) + b(\phi_{1} \land \phi_{2})$.
\end{center}

Now
\begin{align*}
b(\phi_{1} \land \phi_{2}) &= b(\phi \land \psi \land (f_{i+1}^{2} - \delta_{i+1}^{2} \geq 0) \land (0 \leq f_{i+1}^2 \leq
\delta_{i+1}^{2}))\\
&= b(\phi \land \psi \land (f_{i+1} = \delta_{i+1})) + b(\phi \land \psi \land (f_{i+1} = -\delta_{i+1})).
\end{align*}

By Hardt's triviality (Theorem \ref{Hardt}), if we set $M_{t} = \{x \in \textit{Reali}(\phi \land \psi): f_{i+1} = t\}$, then there
exists $t_{0} \in \mathbb{R}$ such that $M_{[-t_{0},0)\cup(0,t_{0}]} = \{x \in \textit{Reali}(\phi \land \psi): f_{i+1} \in
[-t_{0},0)\cup(0,t_{0}]\}$ and $([-t_{0}, 0)\times M_{-t_{0}})\cup((0,t_{0}]\times M_{t_{0}})$ are homeomorphic, and moreover the
homeomorphism can be chosen to be the identity on the fibers $M_{t_{0}}$ and $M_{-t_{0}}$.  This clearly implies that
$M_{[\delta,t_{0}]}=\{x \in \textit{Reali}(\phi \land \psi):t_{0} \geq f_{i+1} \geq \delta\}$ and $M_{[2\delta,t_{0}]}=\{x \in
\textit{Reali}(\phi \land \psi):t_{0} \geq f_{i+1} \geq 2\delta\}$ are homeomorphic.  Hence,
\begin{center}
$b(\phi_{1}) = b(\phi \land \psi \land (f_{i+1} \geq 2\delta_{i+1})) + b(\phi \land \psi \land (f_{i+1} \leq -2\delta_{i+1}))$.
\end{center}

Note that $M_{0} = \textit{Reali}(\phi \land \psi \land (f_{i+1} = 0))$, and $M_{[-\delta, \delta]} = \textit{Reali}(\phi_{2})$.  By Hardt's
Triviality (Theorem \ref{Hardt}), for every $0 \leq u \leq 1$ there is a fiber-preserving homeomorphism $\phi_{u}$ from $M_{[-\delta,
-u\delta]}$ to $[-\delta, -u\delta] \times M_{-u\delta}$ (resp. a homeomorphism $\psi_{u}$ from $M_{[u\delta, \delta]}$ to $[u\delta,
\delta] \times M_{u\delta}$).  We define a continuous homotopy $m$ from $M_{[-\delta, \delta]}$ to $M_{0}$ as follows:
\begin{itemize}
\item $m(0,-)$ is $\lim_{\delta_{i+1}}$,
\item for $0 < u \leq 1$, $m(u,-)$ is the identity on $M_{[-u\delta,u\delta]}$ and sends $M_{[-\delta,-u\delta]}$ (resp.
    $M_{[u\delta, \delta]}$) to $M_{-u\delta}$ (resp. $M_{u\delta}$) by $\phi_{u}$ (resp. $\psi_{u}$) followed by the projection
    on $M_{u\delta}$ (resp. $M_{-u\delta}$).
\end{itemize}
Thus $b(M_{[-\delta,\delta]}) = b(M_{0})$, and $b(\phi \land \psi) \leq \sum_{\pi \in \Sigma_{i+1}}b(\phi \land \psi \land \pi)$.
\end{proof}

\begin{lem} \label{greek}
For every $(F,G)$-closed formula $\phi$,
\begin{center}
$\displaystyle b(\phi) \leq \sum_{\psi \in \Sigma_{\leq s}, \textit{Reali}(\psi) \subset \textit{Reali}(\phi)}b(\psi)$.
\end{center}
\end{lem}

\begin{proof}
Starting from the formula $\phi$, apply Lemma \ref{greek1} with $\psi$ the empty formula.  Now, repeatedly apply Lemma \ref{greek1}
to the terms appearing on the right-hand side of the inequality obtained, noting that for any $\psi \in \Sigma_{\leq s}$,
\begin{itemize}
\item either $\textit{Reali}(\phi \land \psi) = \textit{Reali}(\psi)$ and $\textit{Reali}(\psi) \subset \textit{Reali}(\psi)$,
\item or $\textit{Reali}(\phi \land \psi) = \emptyset$.
\end{itemize}
\end{proof}

We now use these results to prove a bound on the Betti numbers of closed sets.

Let $F = \{f_{1}, \ldots, f_{j}\}$ be o-minimal functions, and let
\begin{center}
$q_{i} = f_{i}^{2}(f_{i}^{2}-\delta_{i}^{2})^{2}(f_{i}^{2}-4\delta_{i}^{2})$.
\end{center}
Let $W_{0}$ be the union of the set $\textit{Reali}(q_{i}=0, Z_{r})$ and $W_{1}$ be the union of $\textit{Reali}(q_{i}\geq0, Z_{r})$, with
$1\leq i \leq j$.

Note that $W_{1} = \bigcup_{\psi \in \Sigma_{\leq s}}\textit{Reali}(\psi)$.

\begin{lem} \label{greekIneq}
$\displaystyle b_{i}(W_{0}) \leq (6^{j}-1)\Omega(F,G)$
\end{lem}
\begin{proof}

The set $\textit{Reali}((f_{i}^{2}(f_{i}^{2}-\delta_{i}^{2})^{2}(f_{i}^{2}-4\delta_{i}^{2})), Z_{r})$ is the disjoint union of
\begin{align*}
\textit{Reali}(f_{i} &= 0, Z_{r}),\\
\textit{Reali}(f_{i} &= \delta_{i}, Z_{r}),\\
\textit{Reali}(f_{i} &= -\delta_{i}, Z_{r}),\\
\textit{Reali}(f_{i} &= 2\delta_{i}, Z_{r}), \text{ and}\\
\textit{Reali}(f_{i} &= -2\delta_{i}, Z_{r}).
\end{align*}
Moreover, the $i$-th Betti numbers of their union $W_{0}$ is bounded by the sum of the Betti numbers of all possible non-empty sets
that can be obtained by taking intersections of these sets using the Mayer-Vietoris Inequalities (Lemma \ref{MVGen}).

The number of possible $l$-ary intersections is $\binom{j}{l}$.  Each such intersection is a disjoint union of $5^{l}$ algebraic
sets.  The $i$-th Betti number of each of these algebraic sets is bounded by $\Omega(F,G)$ by Definition \ref{omega}.  Thus,
\begin{center}
$\displaystyle b_{i}(W_{0}) \leq \sum_{l=1}^{j}5^{l}\Omega(F,G) = (6^{j}-1)\Omega(F,G)$
\end{center}
\end{proof}

\begin{lem}\label{greekIneq2}
$b_{i}(W_{1}) \leq (6^{j}-1)\Omega(F,G) + b_{i}(Z_{r})$.
\end{lem}

\begin{proof}
Let $S = \textit{Reali}(\bigwedge_{1\leq i\leq j} q_{i} \leq 0 \lor \bigvee_{1\leq i \leq j} q_{i} = 0, Z_{r})$.  Now, $W_{1} \cup S =
Z_{r}$, and $W_{1} \cap S = W_{0}$.  Using inequality (\ref{MV1}) we deduce that $b_{i}(W_{1}) \leq b_{i}(W_{1} \cap S) + b_{i}(W_{1}
\cup S) = b_{i}(W_{0}) + b_{i}(Z_{r})$.  We conclude using Lemma \ref{greekIneq}.
\end{proof}

\begin{lem}\label{greek with compelexity}
$\displaystyle \sum_{\psi \in \Sigma_{<s}}b(\psi) \leq \sum_{j=0}^{n'-1}\binom{s}{j}6^{j}\Omega(F,G)$
\end{lem}

\begin{proof}
Since for all $i < n'$, $b_{i}(Z_{r}) + b_{n'}(Z_{r}) \leq \Omega(F,G)$ by Definition \ref{omega}, we have that
\begin{center}
$\displaystyle \sum_{\psi \in \Sigma_{\leq s}} b(\psi) = b(W_{1}) \leq b_{n'}(Z_{r}) + \sum_{j=1}^{n'-1}\binom{s}{j}6^{j}\Omega(F,G)$
\end{center}
using the Mayer-Vietoris Inequalities (Lemma \ref{MVGen}) and Lemma \ref{greekIneq2}.  Thus,
\begin{center}
$\displaystyle \sum_{\psi \in \Sigma_{\leq s}} b(\psi) \leq \sum_{j=0}^{n'-1}\binom{s}{j}6^{j}\Omega(F,G)$
\end{center}
\end{proof}

And our main result:
\begin{thm}\label{noStrict}
Let $G$ and $F$ be finite sets of definable functions from $H \to \mathbb{R}$, where $H \subset \mathbb{R}^{n}$ is a definable open set.
Let $|F|=s$, and $\text{dim}\{\bigcup_{g \in G} g=0\}=n'\leq n$.
The sum of the Betti numbers of a closed set  defined by non-strict inequalities on functions $f \in F$,  and equations on functions $g \in G$,
is bounded by
\begin{center}
$\displaystyle \sum_{i=0}^{n'}\sum_{j=0}^{n'-i}\binom{s}{j}6^{j}\Omega(F,G)$
\end{center}
\end{thm}

\begin{proof}
Follows from Lemmas \ref{greek} and \ref{greek with compelexity}.
\end{proof}

\begin{eg}[The Degree of a Polynomial]
In the polynomial, degree case we have $\Omega(F,G) = d(2d-1)^{n-1}$, and we obtain Theorem 4.1 of \cite{SignCon}, and the bound
$s^{n'}O(d)^{n}$.
\end{eg}

\begin{eg}[Pfaffian Functions]
We have
\begin{equation*}
\Omega(F,G) = 2^{r(r-1)/2}\beta(\alpha+2\beta-1)^{n-1}(\text{min}\{n,r\}\alpha + 2n\beta + (n-1)\alpha-2n+2)^{r},
\end{equation*}
and from this we get the bound
\begin{equation*}
s^{n'}2^{r(r-1)/2}O(n\beta + \text{min}\{n,r\}\alpha)^{n+r},
\end{equation*}
which is equivalent to Theorem 3.4 of \cite{PffNoe}.
\end{eg}

From this we can deduce the following:
\begin{cor} \label{onlyNStrict}
Let $X$ be a set defined by a monotone  Boolean  combination (i.e. exclusively connectives $\lor$, $\land$ are used, no negations) of definable functions containing only non-strict inequalities, then
\begin{center}
$\displaystyle b(X) \leq O(s^{n}\gamma(k, c(\sum_{i}f_{i}^{2} + |x|^{2})))$.
\end{center}
\end{cor}
\begin{proof}
Taking $G=\emptyset$ in Theorem \ref{noStrict} gives $n'=n$, and
\begin{equation}
\Omega(F,G) = \frac{\gamma(n, c(\sum_{i}f_{i}^{2} + |x|^{2})))}{2}.
\end{equation}
\end{proof}

\begin{cor}
The above bound also applies to open sets defined by a monotone Boolean combination of definable functions containing only strict
inequalities.
\end{cor}
\begin{proof}
Consider such a set, and in the Boolean combination replace each $<$ with $\geq$, $>$ with $\leq$, $\cup$ with $\cap$, and $\cap$
with $\cup$.  The original Boolean combination is equivalent to the negation of this new formula (which we can bound using the
previous Corollary).  But  Alexander's Duality tells us the set defined by this new Boolean combination must have the same Betti
sum as the original.
\end{proof}

%%%%%%%%%%%%%%%%%%%%%%%%%%%%%%%%%%%%%%%%%%%%%%%%%%%%%%%%%%%%%%%%%%%%%%%%%%%%%%%%%%%%%%%%
\section{Arbitrary Boolean Combinations}
%%%%%%%%%%%%%%%%%%%%%%%%%%%%%%%%%%%%%%%%%%%%%%%%%%%%%%%%%%%%%%%%%%%%%%%%%%%%%%%%%%%%%%%%

We now move to bound sets defined by an arbitrary Boolean combination of equations and inequalities, both strict and non-strict.  Our
method of proof is taken from the second edition of \cite{AlgInRAG}.  For a given set $X$ defined by such a combination, we
inductively define a series of new sets $X^{0}, \ldots, X^{t+1}$, and eventually show that $X \simeq X^{t+1}$. The latter is a closed
set, which can be dealt with using the method of the previous section.  This homotopy equivalence is show via considering various
sets defined in the inductive process, and showing that certain inclusion relations are in fact homotopy equivalences.  We can use
this result to directly calculate an explicit bound for a set defined by any Boolean combination, and we provide an example in the
polynomial, degree case which reduces exactly to the bound given in our source.  We also show how to find a bound parallel to
$O(s^{2}d)^{n}$.
\begin{defn}
Let $F$ be a finite set of definable functions with $t$ elements, and let $S$ be a bounded $F$-closed set.

Set $\text{SIGN}(S)$ to be the set of realisable sign conditions of $F$ whose realisation is contained in $S$.

For $\sigma \in \text{SIGN}(F)$, we call the number of elements in $\{f \in F:\sigma(f)=0\}$ the level of $\sigma$.

Let
$ \varepsilon_{1} , \varepsilon_{2} , \hdots , \varepsilon_{2t-1} , \varepsilon_{2t}$ be sufficiently small real numbers.
\end{defn}

We now start constructing a set that will approximate our original set.  We first define the elements that will make up this
construction.

Note in the following that the $c$ in $\sigma_{+}^{c}$ stands for closed, and the $o$ in $\sigma_{+}^{o}$ stands for open.

\begin{defn}
For each level $m$, $0 \leq m \leq t$, we set $\text{SIGN}_{m}(S)$ to be the subset of $\text{SIGN}(S)$ of elements of level $m$.

Given $\sigma \in \text{SIGN}_{m}(F,S)$, we set
\begin{align*}
\text{Reali}(\sigma_{+}^{c}) = S \cap \{\{-\varepsilon_{2m} \leq f \leq \varepsilon_{2m} \text{ for each } f \text{ such that }
\sigma{f}=0\} \cup \\
\cup \{f \geq 0 \text{ for each } f \text{ such that } \sigma{f}=1\} \cup\\
 \cup \{f \leq 0 \text{ for each } f \text{ such that } \sigma{f}=-1\}\}
\end{align*}
and
\begin{align*}
\text{Reali}(\sigma_{+}^{o}) = S \cap \{\{-\varepsilon_{2m-1} < f < \varepsilon_{2m-1} \text{ for each } f \text{ such that }
\sigma{f}=0\} \cup\\
 \cup \{f > 0 \text{ for each } f \text{ such that } \sigma{f}=1\} \cup \\
 \cup \{f < 0 \text{ for each } f \text{ such that } \sigma{f}=-1\}\}
\end{align*}
Note that
\begin{align*}
\text{Reali}(\sigma) \subset \text{Reali}(\sigma_{+}^{c}),\\
\text{Reali}(\sigma) \subset \text{Reali}(\sigma_{+}^{o}).
\end{align*}
\end{defn}

We now define the construction of our approximating set.

\begin{defn}
Let $X \subset S$ be a set defined on $F$ such that
\begin{equation*}
X = \bigcup_{\sigma \in \Sigma} \text{Reali}(\sigma)
\end{equation*}
with $\Sigma \subset \text{SIGN}(S)$.
Set $\Sigma_{m} = \Sigma \cap \text{SIGN}_{m}(S)$.

We define a sequence of sets $X^{m} \subset \R^{n}$ for $0 \leq m \leq t$ inductively by
\begin{itemize}
\item $X^{0} = X$

\item For $0 \leq m \leq t$,
\begin{equation*}
X^{m+1} = \left( X^{m} \cup \bigcup_{\sigma \in \Sigma_{m}} \text{Reali}(\sigma_{+}^{c})\right) \setminus \bigcup_{\sigma \in
\text{SIGN}_{m}(S) \setminus \Sigma_{m}} \text{Reali}(\sigma_{+}^{o})
\end{equation*}
\end{itemize}
We set $X' = X^{t+1}$.
\end{defn}

\begin{thm}\label{XEqXDash}
$X \simeq X'$
\end{thm}

Before proving this, we need to introduce the following new sets, which we define inductively.

\begin{defn}
For each $p$ with $ 0 \leq p \leq t+1$, we define sets $Y_{p}, Z_{p} \subset \R^{n}$ as follows
\begin{itemize}
\item We set
\begin{equation*}
Y_{p}^{p} = X \cup \bigcup_{\sigma \in \Sigma_{p}}\text{Reali}(\sigma_{+}^{c})
\end{equation*}
and
\begin{equation*}
Z_{p}^{p} = Y_{p}^{p} \setminus \bigcup_{\sigma \in \text{SIGN}_{p}(S) \setminus \Sigma_{p}}\text{Reali}(\sigma_{+}^{o})
\end{equation*}

\item For $p \leq m \leq t+1$ we define
\begin{equation*}
Y_{p}^{m+1} = \left(Y_{p}^{m} \cup \bigcup_{\sigma \in \Sigma_{m}}\text{Reali}(\sigma_{+}^{c})\right) \setminus \bigcup_{\sigma
\in \text{SIGN}_{m}(S) \setminus \Sigma_{m}}\text{Reali}(\sigma_{+}^{o})
\end{equation*}
and
\begin{equation*}
Z_{p}^{m+1} = \left(Z_{p}^{m} \cup \bigcup_{\sigma \in \Sigma_{m}}\text{Reali}(\sigma_{+}^{c})\right) \setminus \bigcup_{\sigma
\in \text{SIGN}_{m}(S) \setminus \Sigma_{m}}\text{Reali}(\sigma_{+}^{o})
\end{equation*}

\item Define $Y_{p} = Y_{p}^{t+1} \subset \R^{n}$, and $Z_{p} = Z_{p}^{t+1} \subset \R^{n}$.
\end{itemize}
\end{defn}
Note that
\begin{itemize}
\item $X = Y_{t+1} = Z_{t+1}$, and
\item $Z_{0} = X'$.
\end{itemize}

Also note that for each $p$ with $0 \leq p \leq t$,
\begin{itemize}
\item $Z_{p+1}^{p+1} \subset Y_{p}^{p}$, and
\item $Z_{p}^{p} \subset Y_{p}^{p}$.
\end{itemize}

The following Lemma follows directly from the above definitions of $Y_{p}$ and $Z_{p}$.

\begin{lem}
For each $p$ with $0 \leq p \leq t$,
\begin{itemize}
\item $Z_{p+1} \subset Y_{p}$, and
\item $Z_{p} \subset Y_{p}$.
\end{itemize}
\end{lem}
To prove Theorem \ref{XEqXDash}, it suffices to prove that the above inclusions are actually homotopy equivalences.

\begin{lem}\label{YpEqZp1}
For $1 \leq p \leq t$, $Y_{p} \simeq Z_{p+1}$.
\end{lem}

\begin{proof}
Let $Y_{p}(u) \subset \R^{n}$ be the set obtained by replacing $\varepsilon_{2p}$ in the definition of $Y_{p}$ by $u$, and for $u_{0}
> 0$ define
\begin{equation*}
Y_{p}((0,u_{0}]) =\ \{(x,u): x \in Y_{p}(u), u \in (0,u_{0}]\} \subset \R^{n+1}.
\end{equation*}
By Hardt's Triviality (Theorem \ref{Hardt}), there exists $u_{0} >0$ and a homeomorphism
\begin{equation*}
\psi:Y_{p}(u_{0}) \times (0,u_{0}] \to Y_{p}((0, u_{0}])
\end{equation*}
such that
\begin{enumerate}
  \item $\pi_{k+1}(\psi(x,u))=u$,
  \item $\psi(x, u_{0}) = (x, u_{0})$ for $x \in Y_{p}(u_{0})$,
  \item for all $u \in (0, u_{0}]$ and for every sign condition $\sigma$ on
  \begin{equation*}
  \bigcup_{f \in F}\{f, f \pm \varepsilon_{2t}, \ldots, f \pm \varepsilon_{2p+1}\}
  \end{equation*}
  $\psi(\cdot, u)$ defines a homeomorphism of $\text{Reali}(\sigma, Y_{p}(u_{0}))$ to $\text{Reali}(\sigma, Y_{p}(u))$.
\end{enumerate}
Now specify $u_{0}$ to be $\varepsilon_{2p}$, and set $\phi$ to be the map corresponding to $\psi$.
For $\sigma \in \Sigma_{p}$, define:
\begin{align*}
\text{Reali}(\sigma_{++}^{o}) = \{-2\varepsilon_{2p} < f < \varepsilon_{2p} \text{ for all } f \text{ such that }\sigma(f)=0\}
\cup\\
\cup \{f > -2\varepsilon_{2p}  \text{ for all } f \text{ such that }\sigma(f)=1\} \cup\\
\cup \{f < 2\varepsilon_{2p}  \text{ for all } f \text{ such that }\sigma(f)=-1\}.
\end{align*}
Let $\lambda:Y_{p} \to \R$ be a continuous definable function (for example piecewise linear would suffice) such that
\begin{itemize}
\item $\lambda(x) = 1$ on $Y_{p} \cap \bigcup_{\sigma \in \Sigma_{p}}\text{Reali}(\sigma_{+}^{c})$,
\item $\lambda(x) = 0$ on $Y_{p} \setminus \bigcup_{\sigma \in \Sigma_{p}}\text{Reali}(\sigma_{++}^{o})$,
\item $0 < \lambda(x) <1$ otherwise.
\end{itemize}
We now consider a homotopy $Y_{p} \times [0,\varepsilon_{2p}] \to Y_{p}$ by defining
\begin{itemize}
\item $h(x, t) = \pi_{1\ldots k}\circ \phi\left(x, \lambda(x)t + (1-\lambda(x))\varepsilon_{2p}\right)$ for $0 < t \leq
    \varepsilon_{2p}$
\item $h(x,0) = \lim_{t \to 0+}h(x,t)$ otherwise.
\end{itemize}
Note that this limit exists since $S$ is closed and bounded.
We now show that $h(x,0) \in Z_{p+1}$ for all $x \in Y_{p}$.  Let $x \in Y_{p}$ and $y = h(x,0)$.  We have two cases:
\begin{enumerate}
    \item $\lambda(x) < 1$.  Then we have $x \in Z_{p+1}$, and by property (3) of $\phi$ and the fact that $\lambda(x)<1$, $y \in
        Z_{p+1}$.
    \item $\lambda(x) \geq 1$.  Let $\sigma_{y}$ be the sign condition of $F$ at $y$, and suppose that $y \notin Z_{p+1}$. We have
        two cases:
        \begin{enumerate}
            \item $\sigma_{y} \in \Sigma$.  Therefore $y \in X$, and there exists a $\tau \in \text{SIGN}_{m}(S)\setminus\Sigma_{m}$
                with $m>p$ such that $y \in \text{Reali}(\tau_{+}^{o})$.
            \item $\sigma_{y} \notin \Sigma$.  In this case, taking $\tau = \sigma_{y}$, we have $\text{level}(\tau) >p$ and
                $y \in \text{Reali}(\tau_{+}^{o})$.  It follows from the definition of $y$ and property (3) of $\phi$ that
                for all $m>p$ and for all $\rho \in \text{SIGN}_{m}(S)$:
                \begin{itemize}
                    \item $y \in \text{Reali}(\rho_{+}^{o})$ implies that $x \in \text{Reali}(\rho_{+}^{o})$
                    \item $x \in \text{Reali}(\rho_{+}^{c})$ implies that $y \in \text{Reali}(\rho_{+}^{c})$
                \end{itemize}
                Thus $x \notin Y_{p}$, a contradiction.
        \end{enumerate}
\end{enumerate}
It follows that:
\begin{itemize}
\item $h(\cdot, \varepsilon_{2p}):Y_{p}\to Y_{p}$ is the identity,
\item $h(Y_{p}, 0) = Z_{p+1}$, and
\item $h(\cdot, t)$ restricted to $Z_{p+1}$ gives a homotopy between
\begin{equation*}
h(\cdot, \varepsilon_{2p})|_{Z_{p+1}} = \text{id}_{Z_{p+1}}
\end{equation*}
and
\begin{equation*}
h(\cdot, 0)|_{Z_{p+1}}.
\end{equation*}

\end{itemize}
Thus $Y_{p} \simeq Z_{p+1}$.
\end{proof}

\begin{lem}\label{ZpEqYp}
For all $p$ with $0 \leq p \leq t$, $Z_{p} \simeq Y_{p}$.
\end{lem}

\begin{proof}
For this proof define the following new sets for $u \in \R$:
\begin{itemize}
\item $Z_{p}'(u) \subset \R^{n}$ is the set obtained by replacing in the definition of $Z_{p}$, $\varepsilon_{2j}$ by
    $\varepsilon_{2j}-u$, and $\varepsilon_{2j-1}$ by $\varepsilon_{2j-1}+u$ for all $j>p$, and $\varepsilon_{2p}$ by
    $\varepsilon_{2p}-u$ and $\varepsilon_{2p-1}$ by $u$.

    For all $u_{0}>0$, define
    \begin{equation*}
    Z_{p}'((0, u_{0}]) = \{(x,u): x \in Z_{p}'(u), u \in (0, u_{0}]\}.
    \end{equation*}
\item Let $Y_{p}'(u) \subset \R^{n}$ be the set obtained by replacing in the definition of $Y_{p}$, $\varepsilon_{2j}$ by
    $\varepsilon_{2j}-u$ and $\varepsilon_{2j-1}$ by $\varepsilon_{2j-1}+u$ for all $j>p$, and $\varepsilon_{2p}$ by
    $\varepsilon_{2p}-u$.
\item For $\sigma\in \text{SIGN}_{m}(S)$ with $m \geq p$, let $\text{Reali}(\sigma_{+}^{c})(u) \subset \R^{n}$ denote the set
    obtained by replacing $\varepsilon_{2m}$ by $\varepsilon_{2m}-u$ in the definition of $\text{Reali}(\sigma_{+}^{c})$.
\item For $\sigma\in \text{SIGN}_{m}(S)$ with $m > p$, let $\text{Reali}(\sigma_{+}^{o})(u) \subset \R^{n}$ denote the set
    obtained by replacing $\varepsilon_{2m-1}$ by $\varepsilon_{2m-1}+u$ in the definition of $\text{Reali}(\sigma_{+}^{o})$.
\item Finally, for $\sigma \in \text{SIGN}_{p}(S)$, let $\text{Reali}(\sigma_{+}^{o})(u)\subset \R^{n}$ denote the set obtained by
    replacing $\varepsilon_{2p-1}$ by $u$ in the definition of $\text{Reali}(\sigma_{+}^{c})$.
\end{itemize}
Note that by the definitions, for all $u, v \in \R$ with $0 < u \leq v$, $Z_{p}'(u) \subset Y_{p}'(u), Z_{p}'(v) \subset Z_{p}'(u)$,
$Y_{p}'(v) \subset Y_{p}'(u)$, and
\begin{equation*}
\bigcup_{0<s\leq u}Y_{p}'(s) = \bigcup_{0<s \leq u} Z_{p}'(s).
\end{equation*}
Let $Z_{p}' = Z_{p}'(\varepsilon_{2p-1})$ and $Y_{p}' = Y_{p}'(\varepsilon_{2p-1})$.  It is easy to see that $Y'_{p} \simeq Y_{p}$,
and $Z_{p}' \simeq Z_{p}$.  We will now prove that $Y_{p}' \simeq Z_{p}'$, which suffices to prove the Lemma.

Let $\mu:Y_{p}' \to \R$ be a definable map with
\begin{equation*}
\mu(x)=\text{sup}_{u \in (0, \varepsilon_{2p-1}]}\{u:x\in Z_{p}'(u)\}.
\end{equation*}
We will prove below (Lemma \ref{muCts}) that $\mu$ is continuous.  Note that the definition of $Z_{p}'(u)$ (as well as $Y_{p}'(u)$)
is more complicated than the natural one consisting of replacing $\varepsilon_{2p-1}$ in the definition of $Z_{p}$ by $u$, due to the
fact that with the latter definition $\mu$ is not necessarily continuous.

We now construct a definable map
\begin{equation*}
h:Y_{p}'\times [0, \varepsilon_{2p-1}] \to Y_{p}'
\end{equation*}
as follows.  By Hardt's Triviality (Theorem \ref{Hardt}) there exists $u_{0}>0$ and a homeomorphism
\begin{equation*}
\psi:Z_{p}'(u_{0})\times (0,u_{0}] \to Z_{p}'((0,u_{0}])
\end{equation*}
such that:
\begin{enumerate}
\item $\pi_{k+1}(\psi(x,u)) = u$,
\item $\psi(x,u_{0}) = (x,u_{0})$ for $x \in Z_{p}'(u_{0})$,
\item for all $u \in (0,u_{0}]$ and for all sign conditions $\sigma$ of
\begin{equation*}
\bigcup_{f \in F}\{f, f\pm\varepsilon_{2t}, \ldots, f\pm\varepsilon_{2p+1}\},
\end{equation*}
the map $\psi(\cdot, u)$ restricts to a homeomorphism of $\text{Reali}(\sigma,Z_{p}'(u_{0}))$

to $\text{Reali}(\sigma,Z_{p}'(u))$.
\end{enumerate}
Now specify that $u_{0} = \varepsilon_{2p-1}$, and denote by $\phi$ the corresponding map
\begin{equation*}
\phi:Z_{p}'\times (0, \varepsilon_{2p-1}]\to Z_{p}'((0, \varepsilon_{2p-1}]).
\end{equation*}
Note that for all $u$ with $0 < u \leq \varepsilon_{2p-1}$, $\phi$ gives a homeomorphism
\begin{equation*}
\phi_{u}:Z_{p}'(u) \to Z_{p}'.
\end{equation*}
Hence for every pair $u, u'$ with $0 < u \leq u' \leq \varepsilon_{2p-1}$, there exists a homeomorphism
\begin{equation*}
\theta_{u,u'}:Z_{p}'(u) \to Z_{p}'(u')
\end{equation*}
obtained by composing $\phi_{u}$ with $\phi_{u'}^{-1}$.  For $0 \leq u' \leq u \leq \varepsilon_{2p-1}$, let $\theta_{u,u'}$ be the
identity map.  It is clear that $\theta_{u,u'}$ moves continuously with $u$ and $u'$.

For $x \in Y_{p}'$ and $t \in [0,\varepsilon_{2p-1}]$, now define
\begin{equation*}
h(x,t) = \theta_{\mu(x),t}(x).
\end{equation*}
It is easy to verify from the definition of $h$ and the properties of $\phi$ above that $h$ is continuous and:
\begin{itemize}
\item $h(\cdot, 0):Y_{p}' \to Y_{p}'$ is the identity map
\item $h(Y_{p}',\varepsilon_{2p-1})=Z_{p}'$,
\item $h(\cdot, t)$ restricts to a homeomorphism $Z_{p}' \times t \to Z_{p}'$ for every $t \in [0, \varepsilon_{2p-1}]$.
\end{itemize}
Thus, $Z_{p} \simeq Y_{p}$
\end{proof}

\begin{lem}\label{muCts}
For sufficiently small $\varepsilon_{2p-1}$, the definable map $\mu:Y_{p}' \to \mathbb{R}$ defined by
\begin{equation*}
\mu(x) = \text{sup}_{u \in (0,\varepsilon_{2p-1}]}\{u:x \in Z_{p}'(u)\}
\end{equation*}
is continuous.
\end{lem}

\begin{proof}
  In order to prove the continuity of $\mu$ we show that for
every point $ x' \in Y_{p}'$, for all $\varepsilon_{\mu}>0$ there exists $\delta_{\mu}>0$  such that for all $ x \in Y_{p}'$ with $|x-x'| < \delta_{\mu}$ we have $|\mu(x) - \mu(x')|<\varepsilon_{\mu}$.

Choose  $\varepsilon_{\mu}$.  Since each function $f \in F$ is continuous, setting $\varepsilon:=\varepsilon_{\mu}/2$ in the standard definition of continuity, for each $f$ there exists $\delta_{f}>0$ such that if $|x-x'|<\delta_{f}$ then $|f(x)-f(x')|<\varepsilon_{\mu}/2$.  Set $\delta_{\mu} = \text{min}_{f \in F}\{\delta_{f}\}$ , and fix $x$ and $x'$ with $|x-x'|<\delta_{\mu}$.   Let $u \in (0,
\varepsilon_{2p-1}]$ be such that $x \in Z_{p}'(u)$.  We show below that this implies that $x' \in Z_{p}'(u')$ for some $u'$
satisfying $|u-u'| < \varepsilon_{\mu}$.

Let $m$ be the largest integer such that there exists $\sigma \in \Sigma_{m}$ with $x \in \text{Reali}(\sigma_{+}^{c})(u)$: since $x
\in Z_{p}'(u)$ such an $m$ must exist.  We have two cases:
\begin{enumerate}
\item $m>p$: Let $\sigma \in \Sigma_{m}$ with $x \in \text{Reali}(\sigma_{+}^{c})(u)$.  Then by maximality of $m$, for all $f \in
    F$ with $\sigma(f) \neq 0$ we have $f(x) \neq 0$.  Therefore $x' \in \text{Reali}(\sigma_{+}^{c})(u')$ for
    all $u' < u - \text{max}_{f \in F, \sigma(f)=0}|f(x)-f(x')| \leq u - \varepsilon_{\mu}/2$ .  We can thus choose $u'$ such that $x' \in
    \text{Reali}(\sigma_{+}^{c})(u')$ and $|u-u'| < \varepsilon_{\mu}$.

\item $m \leq p$: If $x' \notin Z_{p}'(u)$ then since $x' \in Y_{p}' \subset Y_{p}'(u)$,
\begin{equation*}
x' \in \bigcup_{\sigma \in \text{SIGN}_{p}(F,S)\setminus\Sigma_{p}}\text{Reali}(\sigma_{+}^{o})(u).
\end{equation*}
Let $\sigma \in \text{SIGN}_{p}(S) \setminus \Sigma_{p}$ be such that $x' \in \text{Reali}(\sigma_{+}^{o})(u)$.  We prove by
contradiction that
\begin{equation*}
\displaystyle\text{max}_{f\in F, \sigma(f)=0}|f(x')| = u.
\end{equation*}

Assume the contrary.  Since $x \notin \text{Reali}(\sigma_{+}^{o})(u)$ by assumption, and $|x-x'| < \varepsilon_{\mu}$, there must
exist $f \in F$ with $\sigma(f) \neq 0$ and $f(x) = 0$.  Letting $\tau$ denote the sign condition defined by
$\tau(f) = 0$ if $f(x) = 0$ and $\tau(f) = \sigma(f)$ otherwise, we have that $\text{level}(\tau) >p$ and $x$
belongs to both $\text{Reali}(\tau_{+}^{o})(u)$ and $\text{Reali}(\tau_{+}^{c})(u)$.  Now there are two cases:
\begin{enumerate}
\item $\tau \in \Sigma$: then the fact that $x \in \text{Reali}(\tau_{+}^{c})(u)$ contradicts the choice of $m$, since $m
    \leq p$ and $\text{level}(\tau)>p$.

\item $\tau \notin \Sigma$: then $x$ gets removed at the level of $\tau$ in the construction of $Z_{p}'(u)$, and hence $x \in
    \text{Reali}(\rho_{+}^{c})(u)$ for some $\rho \in \Sigma$ with $\text{level}(\rho)>\text{level}(\tau) >p$.  This again
    contradicts the choice of $m$.
\end{enumerate}
Thus, $\text{max}_{f\in F, \sigma(f)=0}|f(x')|=u$, and since
\begin{equation*}
x' \notin \bigcup_{\sigma\in \text{SIGN}_{p}(C,S)\setminus\Sigma_{p}}\text{Reali}(\sigma_{+}^{o})(u')
\end{equation*}
for all $u' < \text{max}_{f \in F, \sigma(f)=0}|f(x')| = u$, we can choose $u'$ such that $|u-u'| < \varepsilon_{\mu}$ and $x' \notin \bigcup_{\sigma\in \text{SIGN}_{p}(F,S)\setminus\Sigma_{p}}\text{Reali}(\sigma_{+}^{o})(u')$.
\end{enumerate}
In both cases we have $x' \in Z_{p}'(u')$ for some $u'$ satisfying $|u-u'| < \varepsilon_{\mu}$.  This means for  $x \in Y_{p}'$ and $u$ with $x \in Z_{p}'(u)$, and for  $x'$ with $|x' - x| < \delta_{\mu}$, there exists $u'$ with $x' \in Z_{p}'(u')$, with $u' < u$ and $|u - u'| < \varepsilon_{\mu}$.  We therefore have $\mu(x) - \mu(x') < \varepsilon_{\mu}$.  Reversing the roles of $x$ and $x'$ in the above, we get $\mu(x') - \mu(x) < \varepsilon_{\mu}$.
Hence $|\mu(x) - \mu(x')|<\varepsilon_{\mu}$, and $\mu$ is continuous.
\end{proof}

\begin{proof}[Proof of Theorem \ref{XEqXDash}]
This follows from Lemmas \ref{YpEqZp1} and \ref{ZpEqYp}.
\end{proof}

\begin{cor}
Let $Y\subset \R^{n}$ be a set defined by a Boolean combination of definable functions $\{f_{1}, \ldots, f_{s}\}$.  Then there exists
a closed set $X\subset \R^{n}$ defined by $2s^{2}+1$ definable functions with complexities $c(f_{i})$, and $c(|x|)$  such that $X\simeq Y$.
\end{cor}

\begin{thm}
Let $Y\subset \R^{n}$ be a set defined by a Boolean combination of definable functions $\{f_{1}, \ldots, f_{s}\}$.  Then
\begin{equation*}
b(Y) \leq \sum_{i=0}^{n}\sum_{j=1}^{n-i} \binom{2s^{2}+1}{j}6^{j}\Omega(F',\emptyset)
\end{equation*}
where $F'$ contains all the elements $f_{i}$ and $|x|^{2}$.
\end{thm}

\begin{eg}[The Degree of a Polynomial]
We have $\Omega(F',\emptyset)=d(2d-1)^{n-1}$, and we obtain Theorem 7.50 of the second edition
of \cite{AlgInRAG}.
\end{eg}

\begin{eg}[Pfaffian Functions]
We have
\begin{align*}
\Omega(F',\emptyset)=2^{r(r-1)/2}\beta(\alpha+2\beta-1)^{n-1}(\text{min}\{n,r\}\alpha + 2n\beta + (n-1)\alpha-2n+2)^{r},
\end{align*}
and we get
\begin{align*}
b(Y) \leq &\sum_{i=0}^{n}\sum_{j=1}^{n-i} \binom{2s^{2}+1}{j}6^{j}2^{r(r-1)/2}\beta(\alpha+2\beta-1)^{n-1}\cdot\\
&\cdot(\text{min}\{n,r\}\alpha + 2n\beta + (n-1)\alpha-2n+2)^{r}.
\end{align*}
\end{eg}

\begin{cor}\label{arbBoolGen}
\begin{align*}
\displaystyle b(Y) &\leq O\left(s^{2n} \Omega(F',\emptyset)\right)\\
&\leq O\left(s^{2n}\gamma\left(n, c(\sum_{i}f_{i}^{2}+ |x|^{2})\right)\right)
\end{align*}
\end{cor}
\begin{proof}
As in \ref{onlyNStrict}.
\end{proof}

%\begin{cor} \label{arbBoolGen}
%If $\text{max}(c((f_{i}) - 1)^{2})) = m$, then
%\begin{equation*}
%\displaystyle b(X) \leq
%O\left(s^{2n}\gamma\left(n,t_{+}\left(t_{+,s}^{*}\left(m\right),c\left( |x|^{2}+1\right)\right)\right)\right)
%\end{equation*}
%\end{cor}

\begin{eg}[The Degree of a Polynomial]
We have
\begin{equation*}
\gamma\left(n, c(\sum_{i}f_{i}^{2}+ |x|^{2})\right) = 2d(2d-1)^{n-1},
\end{equation*}
and we get the bound $O((s^{2}d)^{n})$.
\end{eg}

\begin{eg}[Pfaffian Functions]
We get a bound
\begin{align*}
\displaystyle &b(X) \leq
O\left(s^{2n}2^{r(r-1)/2}\beta(\alpha+2\beta-1)^{n-1}(\text{min}\{n,r\}\alpha + 2n\beta + (n-1)\alpha-2n+2)^{r}\right)\\
&\leq O\left(s^{2n}2^{r(r-1)/2}(\alpha+2\beta)^{n}(\text{min}\{n,r\}\alpha + 2n\beta + (n-1)\alpha-2n+2)^{r}\right)\\
&\leq O(s^{2n}2^{r(r-1)/2}(2n\beta + (\text{min}\{n,r\}+n-1)\alpha)^{n+r}),
\end{align*}
which is equivalent to the bound stated in \cite{Zell}.
\end{eg}

%%%%%%%%%%%%%%%%%%%%%%%%%%%%%%%%%%%%%%%%%%%%%%%%%%%%%%%%%%%%%%%%%%%%%%%%%%%%%%%%%%%%%%%%
%%%%%%%%%%%%%%%%%%%%%%%%%%%%%%%%%%%%%%%%%%%%%%%%%%%%%%%%%%%%%%%%%%%%%%%%%%%%%%%%%%%%%%%%
\chapter{Sets Defined with Quantifiers}\label{chQuant}
%%%%%%%%%%%%%%%%%%%%%%%%%%%%%%%%%%%%%%%%%%%%%%%%%%%%%%%%%%%%%%%%%%%%%%%%%%%%%%%%%%%%%%%%
%%%%%%%%%%%%%%%%%%%%%%%%%%%%%%%%%%%%%%%%%%%%%%%%%%%%%%%%%%%%%%%%%%%%%%%%%%%%%%%%%%%%%%%%

%%%%%%%%%%%%%%%%%%%%%%%%%%%%%%%%%%%%%%%%%%%%%%%%%%%%%%%%%%%%%%%%%%%%%%%%%%%%%%%%%%%%%%%%
\section{Introduction}
%%%%%%%%%%%%%%%%%%%%%%%%%%%%%%%%%%%%%%%%%%%%%%%%%%%%%%%%%%%%%%%%%%%%%%%%%%%%%%%%%%%%%%%%
We now move on to tackle the most general situation, where $X$ is a subset in $[-1,1]^{n_{0}} \subset \mathbb{R}^{n_{0}}$ defined by
\begin{equation}\label{arbSet}
X = \{x_{0}:Q_{1}x_{1}Q_{2}x_{2}\ldots Q_{\nu}x_{\nu}((x_{0},x_{1},\ldots,x_{\nu})\in X_{\nu})\},
\end{equation}
where $Q_{i} \in \{\exists, \forall\}$, $Q_{i} \neq Q_{i+1}$, $x_{i} \in \mathbb{R}^{n_{i}}$, and $X_{\nu}$ is some definable set in
$[-1,1]^{n_{0}+\hdots+n_{\nu}}$.
\begin{eg}
If $\nu = 1$ and $Q_{1} = \exists$, then $X$ is the projection of $X_{1}$.
\end{eg}
We make use of existing methods and ideas from two main sources in our new setting, as outlined below.

In \cite{BNSASPSets}, the authors present a method for bounding the Betti numbers of a surjective map using a particular spectral
sequence.  They then proceed to demonstrate how to bound the Betti numbers of a set of the form \eqref{arbSet}, in the case where
$X_{\nu}$ is open or closed and is the difference between a finite CW-complex and its subcomplex.  To achieve this, they firstly
state forms of well-known results such as Alexander's Duality (which considers complements), the Mayer-Vietoris Inequality (which
deals with unions and intersections), and De Morgan's Law.  They also introduce a form of notation to manipulate various specific
unions/intersections. These ideas are then used, along with the fact that the existential quantifier is equivalent to projection, to
transform the problem into a form that can be solved using the spectral sequence method applied to the projection map.  Finally,
upper bounds for sub-Pfaffian sets are found by simply counting the number of terms in the expression resulting from the above, and
directly calculating the complexity of the summand.

In \cite{ApproxDefSetCompFam}, a method is shown to approximate a given definable set by a compact family $T$, based on sets that
``represent'' the given set.  This ``representation'' is axiomatically defined, and a particular case of ``representation'' is
constructed in the setting of a bounded definable set of points satisfying a Boolean combination of equations and inequalities.  This
result is then used to demonstrate how to improve existing bounds on Betti numbers of sub-Pfaffian sets, both in the quantifier-free
case, and, making use of the spectral sequence result above, in the case of a single projection.

My work seeks to combine the above ideas to produce a bound for any definable set of the form \eqref{arbSet}, in terms of the general
complexity definition I give at the beginning of this work.  I proceed by presenting various parts of the two above papers, adapted
to my new context, and generalised to give the desired result.

Note that early results for semi-Pfaffian sets defined by quantifier-free formulae (see \cite{Zell}) required the set to be restricted, i.e. $closure(X) \subset (0,1)^{n}$; this requirement was removed in \cite{ApproxDefSetCompFam}.  Similarly, the method shown in  \cite{BNSASPSets} requires the quantifier-free part of the set to be open or closed.
Because our result uses the construction of a compact representation as above, which guarantees that the closed set $T \subset
(0,1)^{n}$, we can remove this restriction and produce a result for any set definable in an o-minimal structure over the reals, in particular, including those defined using quantifiers.

%%%%%%%%%%%%%%%%%%%%%%%%%%%%%%%%%%%%%%%%%%%%%%%%%%%%%%%%%%%%%%%%%%%%%%%%%%%%%%%%%%%%%%%%
\section{Sidenote -- the case where $X_{\nu}$ is open or closed}
%%%%%%%%%%%%%%%%%%%%%%%%%%%%%%%%%%%%%%%%%%%%%%%%%%%%%%%%%%%%%%%%%%%%%%%%%%%%%%%%%%%%%%%%
In \cite{BNSASPSets} a bound is given in terms of a summation of sets (an analogue to Theorem \ref{arbQuant}) in the case where $X_{\nu}$ is open or closed, and is the difference between a finite CW-complex and one of its subcomplexes.  Our new axiomatic complexity measure can be applied to this bound  to produce an explicit bound for sets defined in o-minimal structures. It is a simple calculation to show that this bound reduces to the results given in \cite{BNSASPSets} in the polynomial and Pfaffian cases.

In the following we produce bounds that are weaker the above.  This is due to the fact that $X_{\nu}$ need not be open or closed, and not due to the new complexity metric.

%%%%%%%%%%%%%%%%%%%%%%%%%%%%%%%%%%%%%%%%%%%%%%%%%%%%%%%%%%%%%%%%%%%%%%%%%%%%%%%%%%%%%%%%
\section{A Spectral Sequence Associated with a Surjective Map}
%%%%%%%%%%%%%%%%%%%%%%%%%%%%%%%%%%%%%%%%%%%%%%%%%%%%%%%%%%%%%%%%%%%%%%%%%%%%%%%%%%%%%%%%
We now describe how to associate a particular mathematical object with a given surjective map; this in particular allows us to bound
the Betti numbers of the codomain of this map.  We will use this result later in the instance of projection maps.  The mathematical
object used is a spectral sequence; a good introduction to spectral sequence can be found in \cite{SpecSeq}.

This section is taken directly from \cite{BNSASPSets}.
\begin{defn}
A continuous map $f: X \to Y$ is \em locally split \em if for any $y \in Y$ there is an open neighbourhood $U$ of $y$ and a section
$s: U \to X$ of $f$ (i.e., $s$ is continuous and $fs = \text{Id}$).  In particular, a projection of an open set in $\mathbb{R}^{n}$
on a subspace of $\mathbb{R}^{n}$ is always locally split.
\end{defn}

\begin{defn}
For two maps $f_{1}: X_{1} \to Y$ and $f_{2}: X_{2} \to Y$, the \em fibered product \em of $X_{1}$ and $X_{2}$ is defined as
\begin{center}
$X_{1} \times_{Y} X_{2} = \{(x_{1}, x_{2}) \in X_{1} \times X_{2}: f_{1}(x_{1}) = f_{2}(x_{2})\}$.
\end{center}
\end{defn}

\begin{thm} \label{spectralSeqSurjMap}
Let $f:X \to Y$ be a surjective cellular map.  Assume that $f$ is either closed or locally split.  Then for any Abelian group $G$,
there exists a spectral sequence $E_{p,q}^{r}$ converging to $H_{\ast}(Y,G)$ with
\begin{equation*}
E_{p,q}^{1} = H_{q}(W_{p},G)
\end{equation*}
where
\begin{equation*}
W_{p} = \underbrace{X\times_{Y} \ldots \times_{Y} X}_{p+1\text{ times}}.
\end{equation*}
In particular,
\begin{equation*}
\text{dim }H_{k}(Y,G) \leq \sum_{p+q = k} \text{dim }H_{q}(W_{p},G),
\end{equation*}
for all $k$.
\end{thm}

\begin{proof}
See \cite{BNSASPSets}, Theorem 1.
\end{proof}

%The requirement for the map to be closed in the above comes from the fact that a non-closed map can result in trivial fibers.  For
%example, take the projection map, and a shape in three dimensions that has the two dimensional projection of a figure 8, but in the
%original shape is a line with excluded endpoints each directly above the cross of the 8.  To overcome this difficulty in later work,
%we introduce a way to approximate an arbitrary set with a closed set, which would make the example above ``fatter'', and allow the
%theorem to give an accurate picture of the original shape.

%%%%%%%%%%%%%%%%%%%%%%%%%%%%%%%%%%%%%%%%%%%%%%%%%%%%%%%%%%%%%%%%%%%%%%%%%%%%%%%%%%%%%%%%
\section{Alexander's duality}
%%%%%%%%%%%%%%%%%%%%%%%%%%%%%%%%%%%%%%%%%%%%%%%%%%%%%%%%%%%%%%%%%%%%%%%%%%%%%%%%%%%%%%%%
Again the following is taken from \cite{BNSASPSets}.

Let
\begin{equation*}
I_{i}^{n} = \bigcap_{1 \leq j \leq n} \{-i \leq x_{j} \leq i\} \subset \mathbb{R}^{n}.
\end{equation*}

Define the ``thick boundary'' $\partial I_{i}^{n} = I_{i+1}^{n} \backslash I_{i}^{n}$.  The following is a version of Alexander's
duality theorem.

\begin{lem}[Alexander's duality]\label{Alexander}
If $X \subset I_{i}^{n}$ is an open set in $I_{i}^{n}$, then for any $q \in \mathbb{Z}$, $q \leq n-1$,
\begin{equation}
H_{q}(I_{i}^{n}\backslash X, \mathbb{R}) \cong \tilde{H}_{n-q-1}(X \cup \partial I_{i}^{n}, \mathbb{R}).
\end{equation}
If $X \subset I_{i}^{n}$ is a closed set in $I_{i}^{n}$, then for any $q \in \mathbb{Z}$, $q \leq n-1$,
\begin{equation}
H_{q}(I_{i}^{n}\backslash X, \mathbb{R}) \cong \tilde{H}_{n-q-1}(X \cup closure(\partial I_{i}^{n}), \mathbb{R}).
\end{equation}
\end{lem}

\begin{proof}
See \cite{BNSASPSets}, Lemma 4.
\end{proof}

%%%%%%%%%%%%%%%%%%%%%%%%%%%%%%%%%%%%%%%%%%%%%%%%%%%%%%%%%%%%%%%%%%%%%%%%%%%%%%%%%%%%%%%%
\section{Approximation of Definable Sets by Compact Families}\label{compSec}
%%%%%%%%%%%%%%%%%%%%%%%%%%%%%%%%%%%%%%%%%%%%%%%%%%%%%%%%%%%%%%%%%%%%%%%%%%%%%%%%%%%%%%%%
We now outline the methods shown in \cite{ApproxDefSetCompFam} to find a compact set to approximate an arbitrary set, which can
belong to a large class of definable sets, including sets defined by any Boolean combination of equations and inequalities.

We take the following Definition and two Lemmas directly from \cite{ApproxDefSetCompFam}:
\begin{defn}\label{representation}
Let $G$ be a definable compact set.  Consider a definable family $\{S_{\delta}\}_{\delta>0}$ of compact subsets of $G$ such that, for
all $\delta', \delta \in (0,1)$, if $\delta'>\delta$, then $S_{\delta'} \subset S_{\delta}$.  Denote
$S:=\bigcup_{\delta>0}S_{\delta}$.

For each $\delta>0$, let $\{S_{\delta,\varepsilon}\}_{\delta,\varepsilon>0}$ be a definable family of compact subsets of $G$ such
that:
\begin{enumerate}[(i)]
\item for all $\varepsilon, \varepsilon' \in (0,1)$, if $\varepsilon' > \varepsilon$, then $S_{\delta,\varepsilon} \subset
    S_{\delta, \varepsilon'}$;
\item $S_{\delta} = \bigcap_{\varepsilon>0}S_{\delta, \varepsilon}$;
\item for all $\delta'>0$ sufficiently smaller than $\delta$, and for all $\varepsilon'>0$, there exists an open in $G$ set $U
    \subset G$ such that $S_{\delta} \subset U \subset S_{\delta', \varepsilon'}$.
\end{enumerate}
We say that $S$ is \em represented \em by the families $\{S_{\delta}\}_{\delta>0}$ and
$\{S_{\delta,\varepsilon}\}_{\delta,\varepsilon>0}$ in $G$.
\end{defn}

Let $S'$ be represented by $\{S'_{\delta}\}_{\delta>0}$ and $\{S'_{\delta,\varepsilon}\}_{\delta,\varepsilon>0}$ in $G$, and let
$S''$ be represented by $\{S''_{\delta}\}_{\delta>0}$ and $\{S''_{\delta,\varepsilon}\}_{\delta,\varepsilon>0}$ in $G$.

\begin{lem}\label{repJoinSets}
The set $S' \cap S''$ is represented by the families $\{S'_{\delta}\cap S''_{\delta}\}_{\delta>0}$ and $\{S'_{\delta,\varepsilon}\cap
S''_{\delta,\varepsilon}\}_{\delta,\varepsilon>0}$ in $G$, while $S' \cup S''$ is represented by $\{S'_{\delta}\cup
S''_{\delta}\}_{\delta>0}$ and $\{S'_{\delta,\varepsilon}\cup S''_{\delta,\varepsilon}\}_{\delta,\varepsilon>0}$ in $G$.
\end{lem}

\begin{proof}
By checking of Definition \ref{representation}.
\end{proof}

Let $S$ be represented by $\{S_{\delta}\}_{\delta>0}$ and $\{S_{\delta, \varepsilon}\}_{\delta, \varepsilon>0}$ in $G$, and let
$F:D\to H$ be a continuous definable map, where $D$ and $H$ are definable, $S \subset D \subset G$, and $H$ is compact.

\begin{lem} \label{openRep}
Let $D$ be open in $G$, and $F$ be an open map.  Then $F(S)$ is represented by the families $\{F(S_{\delta})\}_{\delta>0}$ and
$\{F(S_{\delta, \varepsilon})\}_{\delta, \varepsilon>0}$ in $H$.
\end{lem}

\begin{proof}
Follows from checking Definition \ref{representation}; openness is required for (iii) to hold.
\end{proof}

In \cite{ApproxDefSetCompFam}, this general case is called the \em definable \em case (i.e. it has been axiomatically defined).  A
more specific case is actually constructed (in \cite{ApproxDefSetCompFam}) in the situation of a set defined by a Boolean combination
of equations and inequalities, and is referred to as the \em constructible \em case.  We outline this now, again taking the following
directly from \cite{ApproxDefSetCompFam}.

Let $S = \{x: \mathcal{F}(x)\}\subset \mathbb{R}^{n}$ be a bounded definable set of points satisfying a Boolean combination
$\mathcal{F}$ of equations of the kind $h(x)=0$ and inequalities of the kind $h(x)>0$, where $h:\mathbb{R}^{n} \to \mathbb{R}$ are
continuous definable functions (e.g. polynomials).  As $G$ take a closed ball of sufficiently large radius centred at $0$.  We now
define the representing families $\{S_{\delta}\}$ and $\{S_{\delta, \varepsilon}\}$.

\begin{defn}
For a given finite set $\{h_{1},\ldots,h_{k}\}$ of functions $h_{i}:\mathbb{R}^{n} \to \mathbb{R}$, define its \em sign set \em as a
non-empty subset in $\mathbb{R}^{n}$ of the kind
\begin{equation*}
h_{i_{1}}=\hdots=h_{i_{k_{1}}}=0, h_{i_{k_{1}+1}}>0, \ldots, h_{i_{k_{2}}}>0, h_{i_{k_{2}+1}}<0, \ldots h_{i_{k}}<0,
\end{equation*}
where $i_{1}, \ldots, i_{k_{1}}, \ldots, i_{k_{2}}, \ldots, i_{k}$ is a permutation of $1, \ldots k$.
\end{defn}

Let now $\{h_{1}, \ldots h_{k}\}$ be the set of all functions in the Boolean formula defining $S$.  Then $S$ is a disjoint union of
some sign sets of $\{h_{1}, \ldots h_{k}\}$.  The set $S_{\delta}$ is the result of the replacement, independently in each sign set
in this union, of all inequalities $h>0$ and $h<0$ by $h \geq \delta$ and $h \leq -\delta$ respectively.  The set $S_{\delta,
\varepsilon}$ is obtained by replacing %independently in each sign set,
all expressions $h>0$, $h<0$, and $h=0$ by $h \geq \delta$,
$h \leq -\delta$ and $-\varepsilon \leq h \leq \varepsilon$ respectively.  According to Lemma \ref{repJoinSets}, the set $S$ being
the union of sign sets, is represented by the families $\{S_{\delta}\}$ and $\{S_{\delta, \varepsilon}\}$ in $G$.

Now suppose that the set $S \subset \mathbb{R}^{n}$, defined as above by a Boolean formula $\mathcal{F}$, is not necessarily bounded.
In this case as $G$ take the definable one-point (Alexandrov) compactification of $\mathbb{R}^{n}$.  Note that each function $h$ is
continuous in $G \setminus \{\infty\}$.  Define sets $S_{\delta}$ and $S_{\delta, \varepsilon}$ as in the bounded case, replacing
equations and inequalities, independently in each sign set of $\{h_{1},\ldots,h_{k}\}$, and then taking the conjunction of the
resulting formula with $|x|^{2} \leq 1/\delta$.  Again, $S$ is represented by $\{S_{\delta}\}$ and $\{S_{\delta, \varepsilon}\}$ in
$G$.

We now return to the more general, \em definable, \em case.  The following Definitions and Lemma are taken  from
\cite{ApproxDefSetCompFam}, with a slight modification of notation to indicate that $T$ depends on the family chosen to represent
$S$, and on the value $m$ (as defined below).

\begin{defn}
For a sequence $\varepsilon_{0}, \delta_{0}, \varepsilon_{1}, \delta_{1}, \ldots, \varepsilon_{m}, \delta_{m}$, where $m\geq0$, and a
particular representation $\Upsilon = \{S_{\delta,\varepsilon}\}_{\delta, \varepsilon >0}$ of $S$, introduce the compact set
\begin{equation*}
T_{\Upsilon,m}(S):=S_{\delta_{0},\varepsilon_{0}} \cup S_{\delta_{1},\varepsilon_{1}} \cup \ldots \cup
S_{\delta_{m},\varepsilon_{m}}.
\end{equation*}
\end{defn}

In the following, see Definition \ref{ll} for the $\ll$ symbol.

\begin{lem}\label{existC}
For any $m\geq 0$,  for
\begin{equation}\label{infintesimals}
0 < \varepsilon_{0} \ll \delta_{0} \ll \varepsilon_{1} \ll \delta_{1} \ll \ldots \ll \varepsilon_{m} \ll \delta_{m} \ll 1
\end{equation}
and for a representation $\Upsilon$ of $S$, there is a surjective map $C: \mathbf{T} \to \mathbf{S}$ from the finite set $\mathbf{T}$
of all connected components of $T_{\Upsilon,m}(S)$ onto the set $\mathbf{S}$ of all the connected components of $S$ such that, for
any $S' \in \mathbf{S}$, we have
\begin{equation*}
\bigcup_{T'\in C^{-1}(S')}T'=T_{\Upsilon,m}(S').
\end{equation*}
If $m>0$, then $C$ is bijective.
\end{lem}

\begin{proof}
The surjectivity follows directly from Definition \ref{representation}.  Bijectivity is proven in Lemma 1.4 of
\cite{ApproxDefSetCompFam}.
\end{proof}

The previous Lemma essentially states that the passage to $T_{\Upsilon,m}$ preserves connectivity.  The next Theorem is essentially the key result
in \cite{ApproxDefSetCompFam}, and is again taken directly.

In the remained of this document, we denote $T:=T_{\Upsilon,m}(S)$ whenever this does not lead to confusion.

We assume $m>0$, and that $S$ is connected in order to make the homotopy groups $\pi_{k}(S)$ and $\pi_{k}(T)$ independent of a base
point.

\begin{thm} \label{STiso}
For a given representation $\Upsilon$ of $S$:
\begin{enumerate}[(i)]
\item For \eqref{infintesimals} and for every $1 \leq k \leq m$, there are epimorphisms
\begin{align*}
\psi_{k}&:\pi_{k}(T) \to \pi_{k}(S),\\
\phi_{k}&:H_{k}(T) \to H_{k}(S),
\end{align*}
and, in particular, $\text{rank }H_{k}(S) \leq \text{ rank }H_{k}(T)$.

\item In the constructible case, for \eqref{infintesimals} and for every $1 \leq k \leq m-1$, $\psi_{k}$ and $\phi_{k}$ are
    isomorphisms, in particular $\text{rank }H_{k}(S) = \text{ rank }H_{k}(T)$. Moreover, if $m\geq\text{dim}(S)$, then $T\simeq
    S$.
\end{enumerate}
\end{thm}
\begin{proof}
See \cite{ApproxDefSetCompFam}, Theorem 1.5.
\end{proof}
Note that in the above Theorem $\pi_{k}$ refers to homotopy groups, not projections.  The following Corollary is the direct
implication of this existing work in our new setting, and is the only original item in this section.

\begin{cor}\label{TVals}
Let $S \subset \mathbb{R}^{n}$ be defined by an arbitrary Boolean combination of $s$ distinct definable functions $h_{i}$, with
maximum complexity $\text{max}(c(h_{i}))$.  Then there exists a compact set $T \subset \mathbb{R}^{n}$ defined by a Boolean
combination (with no negations) of $4(n+1)s$ non-strict inequalities, with maximum complexity
$\text{max}(c(h_{i}),c(|x|^{2}))$, with $T\simeq S$.
\end{cor}

%%%%%%%%%%%%%%%%%%%%%%%%%%%%%%%%%%%%%%%%%%%%%%%%%%%%%%%%%%%%%%%%%%%%%%%%%%%%%%%%%%%%%%%%
\subsection{Projections} \label{compProj}
%%%%%%%%%%%%%%%%%%%%%%%%%%%%%%%%%%%%%%%%%%%%%%%%%%%%%%%%%%%%%%%%%%%%%%%%%%%%%%%%%%%%%%%%
To move towards dealing with quantifiers, we take a first look at how to deal with projections.  This section is expanded on in more
detail later, we only touch on it now to show how the previous work can be applied in our new setting.  In the following definition,
we show that $T(\pi(S))$ can indeed be defined in the natural way.

\begin{defn}\label{TProj}
 Let $\pi:\mathbb{R}^{n+r}\to \mathbb{R}^{n}$ be the projection function, and $S\subset\mathbb{R}^{n+r}$.  The set $S$ is represented
 by the families $\{S_{\delta}\}_{\delta}$ and $\{S_{\delta, \varepsilon}\}_{\delta, \varepsilon}$ in the compactification of
 $\mathbb{R}^{n+r}$ as described above.  According to Lemma \ref{openRep}, the projection $\pi(S)$ is represented by the families
 $\{\pi(S_{\delta})\}_{\delta}$ and $\{\pi(S_{\delta, \varepsilon})\}_{\delta, \varepsilon}$ in the compactification of
 $\mathbb{R}^{n}$, we call this  particular representation $\Pi$.  Fix $m=n+r$.  Then for a given representation $\Upsilon$ of $S$, we can define $\pi(T_{\Upsilon,m}(S))$ to be equal to $T_{\Pi,m}(\pi(S))$.
\end{defn}

We now demonstrate how this can be applied in the constructible case to give a bound using our general complexity definition,
analogous to bounds given in \cite{ApproxDefSetCompFam} for the polynomial (complexity = degree) case, and for the Pfaffian case.
Let  $S = \{(x,y):\mathcal{F}(x,y)\}\subset\mathbb{R}^{n+r}$, where $\mathcal{F}$ is a Boolean combination of definable equations and
inequalities of the kind $h_{i}(x,y)=0$ or $h_{i}(x,y)>0$, where $h_{i}$  are definable functions $(1 \leq i \leq s)$.  Suppose that
the maximum complexity of these functions is $\text{max}(c(h_{i}))$.

\begin{lem} \label{bettiProj}
\begin{align*}
 b_{k}(\pi(S)) \leq
\left(\frac{k^{3} +4k^{2} +5k +2}{2}s\right)^{n+(k+1)r} O\left(\sum_{0\leq p \leq k}\gamma\left(n+(p+1)r,
c\left(\sum_{i}h_{i}^{2}+ |x|^{2}\right)\right)\right)
\end{align*}
\end{lem}

\begin{proof}
For $k=0$, we have $b_{0}(\pi(S)) \leq b_{0}(S)$. Lemma \ref{existC} tells us that there exists a surjective map from the set of
connected components of $T_{\Upsilon,m}(S)$ to the set of connected components of $S$, and therefore $b_{0}(S)\leq
b_{0}(T_{\Upsilon,m}(S))$.  The set $T_{\Upsilon,m}$ is defined by $4(k+1)s$ definable functions and is in $\mathbb{R}^{n}$, hence,
by Corollary \ref{onlyNStrict}, $b_{0}(T_{\Upsilon,m}(S)) \leq O\left(s^{n}\gamma(n,
c(\sum_{i}h_{i}^{2}+ |x|^{2}))\right)$.

 Now assume $k>0$, and fix $m=k$.  From Definition \ref{TProj}, we have $T_{\Pi,m}(\pi(S)) = \pi(T_{\Upsilon,m}(S))$, and recall from Section \ref{basicAT} that $b_{i}(X) = \text{rank }H_{i}(X)$.  According to
 Theorem \ref{spectralSeqSurjMap}, we have
\begin{equation*}
b_{k}(\pi(T_{\Upsilon,m}(S))) \leq \sum_{p+q=k} b_{q}(W_{p}),
\end{equation*}
where
\begin{equation*}
W_{p} = \underbrace{T_{\Upsilon,m}(S)\times_{\pi(T_{\Upsilon,m}(S))}\hdots \times_{\pi(T_{\Upsilon,m}(S))}T_{\Upsilon,m}(S)}_{p+1
\text{ times}}.
\end{equation*}
The fibre product $W_{p} \subset \mathbb{R}^{n+(p+1)r}$ is definable by a Boolean formula with
\begin{equation*}
4(p+1)(k+1)s
\end{equation*}
o-minimal functions.  Hence, by Corollary \ref{onlyNStrict}, we have
\begin{equation*}
b_{q}(W_{p}) \leq O\left(((p+1)(k+1)s)^{n+(p+1)r}\gamma(n+(p+1)r, c(\sum_{i}h_{i}^{2}+ |x|^{2}))\right).
\end{equation*}
It follows that
\begin{align*}
b_{k}(T_{\Pi,m}(\pi(S)))
&\leq \sum_{0\leq p \leq k}O\left(((p+1)(k+1)s)^{n+(p+1)r}\gamma(n+(p+1)r, c(\sum_{i}h_{i}^{2}+ |x|^{2}))\right)\\
&\leq \left(\frac{k^{3} +4k^{2} +5k +2}{2}s\right)^{n+(k+1)r}
O\left(\sum_{0\leq p \leq k}\gamma(n+(p+1)r, c(\sum_{i}h_{i}^{2}+ |x|^{2}))\right).
\end{align*}
Finally, by Theorem \ref{STiso}, $b_{k}(\pi(S))\leq b_{k}(T_{\Pi,m}(\pi(S)))$, and the result follows.
\end{proof}

\begin{eg}[The Degree of a Polynomial]
In the case of polynomials and degree, we have
\begin{align*}
\gamma\left(n+(p+1)r, c\left(\sum_{i}h_{i}^{2}+|x|^{2}\right)\right)
\leq O(d^{n+(p+1)r}),
\end{align*}
 and therefore $b_{k}(\pi(S)) \leq ((k+1)sd)^{O(n+kr)}$, which is the result given in \cite{ApproxDefSetCompFam} Theorem 6.4.
\end{eg}

\begin{eg}[Pfaffian Functions]
In this case we have
\begin{align*}
\gamma\left(n+(p+1)r, c\left(\sum_{i}h_{i}^{2}+|x|^{2}\right)\right)
\leq O(2^{r^{2}}((n+(p+1)r)(\alpha + \beta))^{n+(p+1)r})
\end{align*}
and therefore
\begin{align*}
b_{k}(\pi(S)) \leq (ks)^{O(n+(k+1)r)}2^{(kl)^{2}}((n+(k+1)r)(\alpha+\beta))^{n+(k+1)r+kl},
\end{align*}
which is the result given in \cite{ApproxDefSetCompFam} Theorem 6.8.
\end{eg}

%%%%%%%%%%%%%%%%%%%%%%%%%%%%%%%%%%%%%%%%%%%%%%%%%%%%%%%%%%%%%%%%%%%%%%%%%%%%%%%%%%%%%%%%
\section{Notation}
%%%%%%%%%%%%%%%%%%%%%%%%%%%%%%%%%%%%%%%%%%%%%%%%%%%%%%%%%%%%%%%%%%%%%%%%%%%%%%%%%%%%%%%%

In the following, remember $\overline{A}$ means the complement of $A$.  Existing work (\cite{BNSASPSets}) defines an initial set $X$, and proceeds
to define sets $X_{i}$ that fit the relation $\pi(X_{i}) = \overline{X_{i-1}}$, and uses this recurrence later in the inductive proof
of the main result.  However, in this setting we must define things slightly differently, as we are using the concept of
``representation'', and we must be able to find a set to ``represent'' the set in the recurrence relation.  This requires the map to
be open, and the complement of the projection does not necessarily take open sets to open sets.  We can circumvent this difficulty by using
projection and co-projection (complement of the projection of the complement) in our definition, as they are both open maps (i.e.
take open sets to open sets). We therefore use the slightly more cumbersome notation that follows, where later work will come to
depend on the parity of $i$.

We assume that $X \subset [-1,1]^{n_{0}} \subset \mathbb{R}^{n_{0}}$, and $X_{\nu} \subset [-1,1]^{n_{0} + \hdots + n_{\nu}}$.  To
restate our original formula, we have
\begin{equation}
X = \{x_{0}:Q_{1}x_{1}Q_{2}x_{2}\ldots Q_{\nu}x_{\nu}((x_{0},x_{1},\ldots,x_{\nu})\in X_{\nu})\},
\end{equation}
where $Q_{i}\neq Q_{i+1}$.
We will also assume that $Q_{1} = \exists$, the other case can be dealt with using Alexander's Duality (and is explicitly dealt with
later using another method in the case of exactly one universal quantifier).

Let $\overline{\pi(\overline{S})}$ be denoted by $\text{cp}(S)$.
We now define sets
\begin{equation*}
X_{i} = \{(x_{0}, \ldots, x_{i}):Q_{i+1}x_{i+1}Q_{i+2}x_{i+2}\ldots
Q_{\nu}x_{\nu}((x_{0}, x_{1}, \ldots, x_{\nu}) \in X_{\nu})\}.
\end{equation*}
This definition satisfies
\begin{align*}
X &= X_{0}\\
&= \pi_{1}(X_{1})\\
&= \pi_{1}(cp_{2}(X_{2}))\\
&= \pi_{1}(cp_{2}(\pi_{3}(X_{3})))\\
&\vdots\\
&= \pi_{1}(cp_{2}(\pi_{3}(\ldots(\tau_{i}(X_{i}))))),\\
\end{align*}
where $\tau_{i} = \pi_{i}$ if $i$ is odd, and $\tau_{i} = cp_{i}$ if $i$ is even.  Essentially this notation gives identical sets to
the existing notation for odd $i$, and the complement of previous sets for even $i$.

For a set $I_{i}^{m_{i}} \times I_{i-1}^{m_{i-1}} \times \ldots \times I_{1}^{m_{1}}$, define $\partial(I_{i}^{m_{i}} \times
I_{i-1}^{m_{i-1}} \times \ldots \times I_{1}^{m_{1}})$ as
\begin{equation*}
(I_{i+1}^{m_{i}} \times I_{i}^{m_{i-1}} \times \ldots \times I_{2}^{m_{1}}) \backslash (I_{i}^{m_{i}} \times I_{i-1}^{m_{i-1}} \times
\ldots \times I_{1}^{m_{1}})
\end{equation*}
for even $i$, and as the closure of this difference for odd $i$.

Let $p_{1}, \ldots, p_{i}$ be some positive integers to be defined later.  Define
\begin{equation*}
B_{i}^{i} = \partial(I_{i-1}^{n_{0}+(p_{1}+1)n_{1}} \times I_{i-2}^{(p_{2}+1)n_{2}} \times \ldots \times I_{1}^{(p_{i-1}+1)n_{i-1}})
\times I_{1}^{n_{i}}.
\end{equation*}
For any $j$, $i < j \leq \nu$, define $B_{j}^{i} = \overline{B_{j-1}^{i}} \times I_{1}^{n_{j}}$, where complements are in appropriate
cubes.

%%%%%%%%%%%%%%%%%%%%%%%%%%%%%%%%%%%%%%%%%%%%%%%%%%%%%%%%%%%%%%%%%%%%%%%%%%%%%%%%%%%%%%%%
\subsection{General Intersections and Unions}
%%%%%%%%%%%%%%%%%%%%%%%%%%%%%%%%%%%%%%%%%%%%%%%%%%%%%%%%%%%%%%%%%%%%%%%%%%%%%%%%%%%%%%%%
We now define an ``intersection'' of particular types of sets, again taking all Definitions and Lemmas in this section directly from
\cite{BNSASPSets}.

\begin{defn} \label{prodsOfSets}
\begin{itemize}

\item[(i)] Let $Y \subset I_{v}^{n_{0}} \times I_{v}^{(p_{1}+1)n_{1}} \times I_{v-1}^{(p_{2}+1)n_{2}} \times \ldots \times I_{\nu
    - l +2}^{(p_{l-1}+1)n_{l-1}} \times I_{1}^{n_{l} + \ldots + n_{i}}$, where $1 \leq l \leq i$, $v \geq i$, and let $J \subset
    \{(j_{l}, \ldots, j_{i}): 1 \leq j_{k} \leq p_{k} + 1, l \leq k \leq i\}$.  Then define $\bigsqcap_{i,J}^{l}Y$ as an intersection
    of sets
    \begin{multline*}
    \{(x_{0}, x_{1}^{(1)}, \ldots, x_{1}^{(p_{1}+1)}, \ldots, x_{i}^{(1)},\ldots, x_{i}^{(p_{i}+1)}): \\
    x_{0} \in I_{v}^{n_{0}}, x_{k}^{(m)} \in I_{v-k+1}^{n_{k}} (1 \leq k \leq l-1),    \\
    x_{k}^{(m)} \in I_{1}^{n_{k}} (l \leq k \leq i), (x_{0}, x_{1}^{(1)}, \ldots, x_{l-1}^{(p_{l-1}+1)}, x_{l}^{(j_{l})}, \ldots,
    x_{i}^{(j_{i})}) \in Y \}
    \end{multline*}
    over all $(j_{l},\ldots, j_{i}) \in J$.

\item[(ii)]  Let $Y \subset I_{v}^{n_{0}} \times I_{v}^{(p_{1}+1)n_{1}} \times I_{v-1}^{(p_{2}+1)n_{2}} \times \ldots \times
    I_{\nu - l +2}^{(p_{l-1}+1)n_{l-1}} \times I_{1}^{n_{l} + \ldots + n_{i}+ n_{i+1}}$.  Define $\bigsqcap_{i,J}^{l,i+1}Y$ as an
    intersection of sets
    \begin{multline*}
    \{(x_{0}, x_{1}^{(1)}, \ldots, x_{1}^{(p_{1}+1)}, \ldots, x_{i}^{(1)},\ldots, x_{i}^{(p_{i}+1)}, x_{i+1}): \\
    x_{0} \in I_{v}^{n_{0}}, x_{k}^{(m)} \in I_{v-k+1}^{n_{k}} (1 \leq k \leq l-1),  x_{k}^{(m)} \in I_{1}^{n_{k}} (l \leq k \leq
    i),  \\
     x_{i+1} \in I_{1}^{n_{i+1}}, (x_{0}, x_{1}^{(1)}, \ldots, x_{l-1}^{(p_{l-1}+1)}, x_{l}^{(j_{l})}, \ldots, x_{i}^{(j_{i})})
     \in Y \}
    \end{multline*}
    over all $(j_{l},\ldots, j_{i}) \in J$.

\item[(iii)] If $l=i$ and $J = \{j: 1 \leq j \leq p_{i}+1\}$ we use the notation $\bigsqcap_{i}^{i}Y$ for $\bigsqcap_{i,J}^{i}Y$

\end{itemize}
\end{defn}

The following shows how these intersections combine:

\begin{lem} \label{combProd}
Let
\begin{equation*}
Y \subset I_{v}^{n_{0}} \times I_{v}^{(p_{1}+1)n_{1}} \times I_{v-1}^{(p_{2}+1)n_{2}} \times \hdots \times
I_{v-l+2}^{(p_{l-1}+1)n_{l-1}} \times I_{1}^{n_{l}+\hdots + n_{i} + n_{i+1}}.
\end{equation*}
Then for any $J \subset \{j:1\leq j \leq  p_{i+1} + 1\}$, $J' \subset \{(j_{l}, \ldots, j_{i}): 1 \leq j_{k} \leq p_{k} +1, l \leq k
\leq i\}$ we have
\begin{equation*}
\bigsqcap_{i+1,J}^{i+1} \bigsqcap_{i,J'}^{l, i+1}Y = \bigsqcap_{i+1, J' \times J}^{l}Y.
\end{equation*}
\end{lem}

\begin{proof}
Straightforward.
\end{proof}

We similarly define unions:

\begin{defn}
Let $Y$, $l$, $i$, $J$ be as in Definition \ref{prodsOfSets}.  Define $\bigsqcup_{i,J}^{l}Y$ and $\bigsqcup_{i,J}^{l,i+1}Y$ similar
to $\bigsqcap_{i,J}^{l}Y$ and $\bigsqcap_{i,J}^{l,i+1}Y$ respectively, replacing in Definition \ref{prodsOfSets} ``intersection'' by
``union''.
\end{defn}

We now state a version of De Morgan's law, which show how unions, intersections, and complements interact.

\begin{lem}[De Morgan's Law]\label{DeMorgan}
\begin{equation*}
\bigsqcup_{i,J}^{l}Y = \overline{\left(\bigsqcap_{i,J}^{l}\overline{Y}\right)}
\end{equation*}

\begin{equation*}
\bigsqcup_{i,J}^{l,i+1}Y = \overline{\left(\bigsqcap_{i,J}^{l,i+1}\overline{Y}\right)}
\end{equation*}

where complements are in appropriate cubes.
\end{lem}

\begin{proof}
Straightforward.
\end{proof}

We now define projections:

\begin{defn}\label{tiAndProj}
Let $t_{i} = n_{0} + n_{1}(p_{1}+1) + \hdots + n_{i}(p_{i}+1)$.  Define projection maps
\begin{equation*}
\pi_{i}:\mathbb{R}^{n_{0} + \hdots + n_{i}} \to \mathbb{R}^{n_{0} + \hdots + n_{i-1}}
\end{equation*}
\begin{equation*}
(x_{0}, \ldots, x_{i}) \mapsto (x_{0}, \ldots, x_{i-1})
\end{equation*}
and for $j<i$,
\begin{equation*}
\pi_{i,j}:\mathbb{R}^{t_{j}+n_{j+1}+\ldots+n_{i}} \to \mathbb{R}^{t_{j}+n_{j+1}+\ldots+n_{i-1}}
\end{equation*}
\begin{equation*}
(x_{0}, x_{1}^{(1)},\ldots,x_{j}^{(p_{j}+1)},x_{j+1},\ldots,x_{i}) \mapsto (x_{0},
x_{1}^{(1)},\ldots,x_{j}^{(p_{j}+1)},x_{j+1},\ldots,x_{i-1}).
\end{equation*}
\end{defn}

We finally show how unions and projections commute

\begin{lem} \label{unionProj}
Let
\begin{equation*}
Y \subset I_{v}^{n_{0}} \times I_{v}^{(p_{1}+1)n_{1}} \times I_{v-1}^{(p_{2}+1)n_{2}} \times \hdots \times
I_{v-l+2}^{(p_{l-1}+1)n_{l-1}} \times I_{1}^{n_{l}+\hdots + n_{i} + n_{i+1}}.
\end{equation*}
Then
\begin{equation*}
\bigsqcup_{i,J}^{l} \pi_{i+1,l-1}(Y) = \pi_{i+1,i}\left(\bigsqcup_{i,J}^{l,i+1}Y\right).
\end{equation*}
\end{lem}

\begin{proof}
Straightforward.
\end{proof}

%%%%%%%%%%%%%%%%%%%%%%%%%%%%%%%%%%%%%%%%%%%%%%%%%%%%%%%%%%%%%%%%%%%%%%%%%%%%%%%%%%%%%%%%
\section{Application of Spectral Sequence Result}
%%%%%%%%%%%%%%%%%%%%%%%%%%%%%%%%%%%%%%%%%%%%%%%%%%%%%%%%%%%%%%%%%%%%%%%%%%%%%%%%%%%%%%%%

We present a Lemma that will allow us to use Theorem \ref{spectralSeqSurjMap} in our proof.

\begin{lem} \label{specSeqAppli}
For any closed $A \subset I_{v}^{n_{0}} \times I_{v}^{(p_{1}+1)n_{1}} \times I_{v-1}^{(p_{2}+1)n_{2}} \times \hdots \times
I_{v-i+2}^{(p_{i-1}+1)n_{i-1}} \times I_{1}^{n_{i}}$,

\begin{equation*}
b_{k}(\pi_{i,i-1}(A)) \leq \sum_{p_{i}+q_{i} = k} b_{q_{i}}(\bigsqcap_{i}^{i}(A)).
\end{equation*}
\end{lem}

\begin{proof}
%Follows trivially from Theorem \ref{spectralSeqSurjMap}.
Theorem \ref{spectralSeqSurjMap} tells us that
\begin{equation*}
b_{k}(Y) \leq \sum_{p+q=k}b_{q}(W_{p}).
\end{equation*}
We will take $Y = \pi_{i,i-1}(A)$, and $f = \pi_{i,i-1}$.  For convenience, let us say $I = I_{v}^{n_{0}} \times I_{v}^{(p_{1}+1)n_{1}} \times I_{v-1}^{(p_{2}+1)n_{2}} \times \hdots \times
I_{v-i+2}^{(p_{i-1}+1)n_{i-1}}$ so $A \subset I \times I_{1}^{n_{i}}$, so we have
$\pi_{i,i-1}(A) \subset I$.

From the definition,
\begin{equation*}
W_{p_{i}} = \{((a,b_{1}),(a,b_{2}),\ldots,(a,b_{p_{i}+1}): a \in I, b_{j} \in I_{1}^{n_{i}}, (a,b_{j}) \in A, 1 \leq j \leq
p_{i}+1\}.
\end{equation*}
Similarly, we have
\begin{equation*}
\bigsqcap_{i}^{i}(A) = \{(a, b_{1}, b_{2},\ldots,b_{p_{i}+1}):a \in I, b_{j} \in I_{1}^{n_{i}}, (a,b_{j}) \in A, 1 \leq j \leq
p_{i}+1\},
\end{equation*}
and the result follows.
\end{proof}

%%%%%%%%%%%%%%%%%%%%%%%%%%%%%%%%%%%%%%%%%%%%%%%%%%%%%%%%%%%%%%%%%%%%%%%%%%%%%%%%%%%%%%%%
\section{Outline of Proof}
%%%%%%%%%%%%%%%%%%%%%%%%%%%%%%%%%%%%%%%%%%%%%%%%%%%%%%%%%%%%%%%%%%%%%%%%%%%%%%%%%%%%%%%%

We now have all the tools we need to resolve the main problems of this section.    Our proof of the main result is through induction
on the number of quantifiers.  The case of one existential quantifier (equivalent to projection) is straightforward, and indeed is
essentially done in Section \ref{compProj}.  The one universal quantifier case can be deduced using Alexander's duality, but we
present another proof, which is a good introduction to more complex proofs to follow.  The two and three quantifiers cases are also
given, as these illustrate the method of proof more clearly than if one immediately moved to the general case.  These two extra cases
are only presented  with $Q_{1} = \exists$; the corresponding results for $Q_{1} = \forall$ can either be deduced through similar
methods, or directly from these results using Alexander's duality.

Our main proof in this section, that of the case of an arbitrary number of quantifiers, makes use of most of the preceding results in
this chapter.  To summarise these results we have:
\begin{itemize}
\item Lemma \ref{specSeqAppli} shows us how to use Theorem \ref{spectralSeqSurjMap} to transform the problem of finding Betti
    numbers of projections into finding the sum of Betti numbers of ``intersections'' as defined in Definition \ref{prodsOfSets}

\item Lemma \ref{Alexander}, Alexander's Duality, shows how to take the Betti numbers of complements.

\item Lemma \ref{MVGen}, the Mayer-Vietoris Inequality, shows how to turn intersections into unions, and unions into
    intersections.

\item Lemma \ref{combProd} shows us how to combine two of the defined ``intersections'' into one.

\item Lemma \ref{DeMorgan}, De Morgan's Law, shows how the complement of the union is the intersection of the complements.

\item Lemma \ref{unionProj} shows how unions and projections commute.

\item Section \ref{compSec} shows us how to approximate a given set by a compact set, provided we can find sets to ``represent''
    it.
\end{itemize}

We use these tools, along with the fact that the existential quantifier is equivalent to projection, to reduce the problem of $\nu$
quantifiers into the problem of $\nu - 1$ quantifiers, and derive the result using induction.

%%%%%%%%%%%%%%%%%%%%%%%%%%%%%%%%%%%%%%%%%%%%%%%%%%%%%%%%%%%%%%%%%%%%%%%%%%%%%%%%%%%%%%%%
\section{Quantifiers}
%%%%%%%%%%%%%%%%%%%%%%%%%%%%%%%%%%%%%%%%%%%%%%%%%%%%%%%%%%%%%%%%%%%%%%%%%%%%%%%%%%%%%%%%

In the following, we assume that when calculating $T_{\Upsilon,m}(S)$ for any given set $S$, $m$ is chosen with $m \geq dim(S)$, which ensures $b_{k}(S) \leq
b_{k}(T_{\Upsilon,m}(S))$ for every $k$.

%%%%%%%%%%%%%%%%%%%%%%%%%%%%%%%%%%%%%%%%%%%%%%%%%%%%%%%%%%%%%%%%%%%%%%%%%%%%%%%%%%%%%%%%
\subsection{One Existential Quantifier}
%%%%%%%%%%%%%%%%%%%%%%%%%%%%%%%%%%%%%%%%%%%%%%%%%%%%%%%%%%%%%%%%%%%%%%%%%%%%%%%%%%%%%%%%
Let us restate the original formula being studied, Equation \eqref{arbSet}:
\begin{equation*}
X = \{x_{0}:Q_{1}x_{1}Q_{2}x_{2}\ldots Q_{\nu}x_{\nu}((x_{0},x_{1},\ldots,x_{\nu})\in X_{\nu})\}.
\end{equation*}
We will now consider the case where $\nu = 1$ and $Q_{1} = \exists$, therefore
\begin{equation*}
X = \{x_{0}:\exists x_{1}((x_{0},x_{1})\in X_{1})\},
\end{equation*}
where $X_{1} = \{F(x_{0},x_{1})\}$.  Let $F$ be a Boolean combination of definable equations and inequalities of the kind $f_{i}=0$
and $f_{i}>0$, where  there are $s$ different functions $f_{i}$.

\begin{lem}\label{oneQuantT}
For any representation $\Upsilon$ of $X_{1}$,
\begin{equation*}
b_{q_{0}}(X) \leq \sum_{p_{1}+q_{1}=q_{0}} b_{q_{1}}(\bigsqcap_{1,J_{1}^{1}}^{1} T_{\Upsilon,m}(X_{1})),
\end{equation*}
where $J_{1}^{1}=\{1,\ldots,p_{1}+1\}$.
\end{lem}

\begin{proof}
We have
\begin{align*}
b_{q_{0}}(X) &= b_{q_{0}}(\pi_{1}(X_{1})),
\end{align*}
remembering that $\Pi$ is the representation defined via the projection in Definition \ref{TProj}, from Theorem \ref{STiso} the previous expression is bounded by
\begin{equation*}
 b_{q_{0}}(T_{\Pi,m}(\pi_{1}(X_{1}))),
\end{equation*}
which by Definition \ref{TProj} is equal to
\begin{equation*}
 b_{q_{0}}(\pi_{1}(T_{\Upsilon,m}(X_{1}))),
\end{equation*}
and by Lemma \ref{specSeqAppli} this is bounded by
\begin{equation*}
\sum_{p_{1}+q_{1} = q_{0}} b_{q_{1}}(\bigsqcap_{1}^{1}(T_{\Upsilon,m}(X_{1}))).
\end{equation*}
\end{proof}

%%%%%%%%%%%%%%%%%%%%%%%%%%%%%%%%%%%%%%%%%%%%%%%%%%%%%%%%%%%%%%%%%%%%%%%%%%%%%%%%%%%%%%%%
\subsection{One Universal Quantifier}
%%%%%%%%%%%%%%%%%%%%%%%%%%%%%%%%%%%%%%%%%%%%%%%%%%%%%%%%%%%%%%%%%%%%%%%%%%%%%%%%%%%%%%%%
We will now consider the case where $\nu = 1$ and $Q_{1} = \forall$, therefore
\begin{equation*}
X = \{x_{0}:\forall x_{1}((x_{0},x_{1})\in X_{1})\},
\end{equation*}
where $X_{1} = \{F(x_{0},x_{1})\}$.  Note that a single universal quantifier is equivalent to $\neg \exists \neg$, and therefore
equivalent to the complement of the projection of the complement, or the ``co-projection''.

This can be done relatively directly using Alexander's Duality, but we present the following method as it is how we approach the same
situation later.

In the following section, it is important to remember that we use $\overline{A}$ to mean the complement of $A$, and that we let
$\overline{\pi(\overline{S})}$ be denoted by $\text{cp}(S)$.

\begin{lem}\label{subsets}
If $A \subset B$, then $\text{cp}(A) \subset \text{cp}(B)$.
\end{lem}

\begin{proof}
We have $A \subset B$, therefore $\overline{A} \supset \overline{B}$, $\pi(\overline{A}) \supset \pi(\overline{B})$, and
$\overline{\pi(\overline{A})} \subset \overline{\pi(\overline{B})}$.
\end{proof}

When taking complements throughout the rest of this chapter, remember that
%\begin{equation*}
$X_{\nu} \subset [-1,1]^{n_{0} + \hdots + n_{\nu}}$.
%\end{equation*}

The next Lemma follows from Lemma \ref{openRep}, but is explicitly proven here as it is vital in later work.

\begin{lem}\label{cpRep}
Let $G$ be a definable compact set, and let $\{S_{\delta, \varepsilon}\}_{\delta, \varepsilon >0}$ be a definable family of compact subsets of $G$.
If $S$ is represented by $\{S_{\delta, \varepsilon}\}_{\delta, \varepsilon >0}$ then $\text{cp}(S)$ is represented by
$\{\text{cp}(S_{\delta, \varepsilon})\}_{\delta, \varepsilon >0}$.
\end{lem}

\begin{proof}
We must show that the three requirements in the definition of ``representation'' (Definition \ref{representation}) hold.
For any $\delta>0$ we have:
\begin{enumerate}[(i)]
\item For all $\varepsilon, \varepsilon' \in (0,1)$, if $\varepsilon' > \varepsilon$, then $S_{\delta,\varepsilon}
    \subset S_{\delta, \varepsilon'}$, and thus from Lemma \ref{subsets} we have $\text{cp}(S_{\delta,\varepsilon}) \subset
    \text{cp}(S_{\delta, \varepsilon'})$.

\item $S_{\delta} = \bigcap_{\varepsilon>0}S_{\delta, \varepsilon}$, so $\text{cp}(S_{\delta}) =
    \text{cp}(\bigcap_{\varepsilon>0}S_{\delta, \varepsilon}) =\bigcap_{\varepsilon>0}\text{cp}(S_{\delta, \varepsilon})$.

\item Since $S$ represented by $\{S_{\delta, \varepsilon}\}_{\delta, \varepsilon >0}$, for all $\delta'>0$ sufficiently smaller than $\delta$, and for all $\varepsilon'>0$, there exists an open in $G$
    set $U \subset G$ such that $S_{\delta} \subset U \subset S_{\delta', \varepsilon'}$. We  therefore have $\text{cp}(S_{\delta})
    \subset \text{cp}(U) \subset \text{cp}(S_{\delta', \varepsilon'})$. Since $\text{cp}(U)$ is open if $U$ is open, and
    indeed $U$ is open, this condition is satisfied.
\end{enumerate}
\end{proof}

We now reformulate Lemma \ref{specSeqAppli} for open sets:
\begin{lem}\label{specSeqAppli2}
For any open
\begin{equation*}
A \subset I_{v}^{n_{0}} \times I_{v}^{(p_{1}+1)n_{1}} \times I_{v-1}^{(p_{2}+1)n_{2}} \times \hdots \times
I_{v-i+2}^{(p_{i-1}+1)n_{i-1}} \times I_{1}^{n_{i}},
\end{equation*}
\begin{equation*}
b_{k}(\pi_{i,i-1}(A)) \leq \sum_{p_{i}+q_{i} = k} b_{q_{i}}(\bigsqcap_{i}^{i}(A)).
\end{equation*}
\end{lem}
This is used in the following, where we consider the case where $\nu = 1$ and $Q_{1} = \forall$:

\begin{lem}
\begin{align*}
b_{q_{0}}(X) =& b_{q_{0}}(\overline{\pi_{1}(\overline{(X_{1})})}) \\
 \leq& \sum_{J\subset \{1, \ldots, m+1\}} \sum_{p_{1}+q_{1}=n_{0}-q_{0}-2+|J|}b_{q_{1}}\left(
 \bigsqcap_{1,J_{1}^{1}}^{1}\left(\bigcup_{j \in J}\overline{(X_{1})_{\delta_{j}, \varepsilon_{j}}} \cup \partial I_{1}^{n_{0}} \times
 I^{n_{1}}_{1}\right)\right).
\end{align*}
\end{lem}
\begin{proof}
Let $\Upsilon$ be the representation $\{\text{cp}((X_{1} )_{\delta, \varepsilon})\}_{\delta, \varepsilon>0}$ of $\text{cp}(X_{1})$, which is a valid representation from Lemma \ref{cpRep}. We have from Theorem \ref{STiso}, $b_{q_{0}}(\text{cp}(X_{1})) \leq b_{q_{0}}(T_{\Upsilon,m}(\text{cp}(X_{1})))$, and explicitly using this representation and the definition of $T$ we have
\begin{align*}
\displaystyle
b_{q_{0}}\left(T_{\Upsilon,m}\left(\text{cp}\left(X_{1}\right)\right)\right)&= b_{q_{0}}\left(\text{cp}\left((X_{1})_{\delta_{0},
\varepsilon_{0}}\right) \right) \cup \hdots \cup \text{cp}\left((X_{1})_{\delta_{m}, \varepsilon_{m}}\right)\\
&= b_{q_{0}}\left(\overline{\pi\left(\overline{(X_{1})_{\delta_{0}, \varepsilon_{0}}}\right)} \cup \hdots \cup
\overline{\pi\left(\overline{(X_{1})_{\delta_{m}, \varepsilon_{m}}}\right)}\right)\\
&= b_{q_{0}}\left(\overline{\pi\left(\overline{(X_{1})_{\delta_{0}, \varepsilon_{0}}}\right) \cap \hdots \cap
\pi\left(\overline{(X_{1})_{\delta_{m}, \varepsilon_{m}}}\right)}\right)
\end{align*}
Now we have $(X_{1})_{\delta, \varepsilon}$ is closed, so $\overline{(X_{1})_{\delta, \varepsilon}}$ is open,
$\pi\left(\overline{(X_{1})_{\delta, \varepsilon}}\right)$ is open, and the intersection of any finite collection of open sets is
open, therefore the set we are considering is closed. By Alexander's Duality we have,
\begin{align*}
\displaystyle
&b_{q_{0}}\left(\overline{\pi\left(\overline{(X_{1})_{\delta_{0}, \varepsilon_{0}}}\right) \cap \hdots \cap
\pi\left(\overline{(X_{1})_{\delta_{m}, \varepsilon_{m}}}\right)}\right) \\
&=b_{n_{0}-q_{0}-1}\left(\left(\pi\left(\overline{(X_{1})_{\delta_{0}, \varepsilon_{0}}}\right) \cap \hdots \cap
\pi\left(\overline{(X_{1})_{\delta_{m}, \varepsilon_{m}}}\right) \right)\cup \partial I_{1}^{n_{0}}\right)
\end{align*}
and using the Mayer-Vietoris Inequalities this does not exceed
\begin{align*}
\displaystyle
&\sum_{J\subset \{1, \ldots, m+1\}} b_{n_{0}-q_{0}-2+|J|}\left(\bigcup_{j \in J}\pi\left(\overline{(X_{1})_{\delta_{j},
\varepsilon_{j}}}\right) \cup \partial I_{1}^{n_{0}}\right)\\
&= \sum_{J\subset \{1, \ldots, m+1\}} b_{n_{0}-q_{0}-2+|J|}\left(\bigcup_{j \in J}\pi\left(\overline{(X_{1})_{\delta_{j},
\varepsilon_{j}}}\right) \cup \pi\left(\partial I_{1}^{n_{0}} \times I_{1}^{n_{1}}\right)\right)\\
&= \sum_{J\subset \{1, \ldots, m+1\}} b_{n_{0}-q_{0}-2+|J|}\left(\pi\left(\bigcup_{j \in J}\overline{(X_{1})_{\delta_{j}
\varepsilon_{j}}} \cup \partial I_{1}^{n_{0}} \times I_{1}^{n_{1}}\right)\right).
\end{align*}
Now using Lemma \ref{specSeqAppli2} this is bounded by
\begin{align*}
 \sum_{J\subset \{1, \ldots, m+1\}} \sum_{p_{1}+q_{1}=n_{0}-q_{0}-2+|J|}b_{q_{1}}\left(\bigsqcap_{1,J_{1}^{1}}^{1}\left(\bigcup_{j
\in J}\overline{(X_{1})_{\delta_{j}, \varepsilon_{j}}} \cup \partial I_{1}^{n_{0}} \times I^{n_{1}}_{1}\right)\right).
\end{align*}
\end{proof}

%%%%%%%%%%%%%%%%%%%%%%%%%%%%%%%%%%%%%%%%%%%%%%%%%%%%%%%%%%%%%%%%%%%%%%%%%%%%%%%%%%%%%%%%
\subsection{Two Quantifiers - $\exists\> \forall$}
%%%%%%%%%%%%%%%%%%%%%%%%%%%%%%%%%%%%%%%%%%%%%%%%%%%%%%%%%%%%%%%%%%%%%%%%%%%%%%%%%%%%%%%%

\begin{lem}
\begin{align*}
&b_{q_{0}}(X) = \displaystyle b_{q_{0}}\left(\pi_{1}\left({\rm cp}_{2}(X_{2})\right)\right)\\
 &\leq \sum_{p_{1}+q_{1}=q_{0}} \sum_{\hat{J_{1}^{1}} \subset J_{1}^{1}}
 \sum_{J_{2}^{2}\subset \hat{J_{1}^{1}}} \sum_{K_{1} \subset \{1, \ldots, m\}}
 \sum_{p_{2} + q_{2} = t_{1} - q_{1} - |\hat{J_{1}^{1}}| - |K_{1}| - |J_{2}^{2}| - 1 }\\
 &b_{q_{2}}\left(\bigsqcap_{2}^{2}\left(\bigsqcup_{1, J_{2}^{2}}^{1,2} \bigcup_{i \in K_{1}} \overline{(X_{2})_{\delta_{i},
 \varepsilon_{i}}}  \cup B_{2}^{2}\right)\right)
\end{align*}
\end{lem}

\begin{proof}

Firstly, from Lemma \ref{oneQuantT}:
\begin{align*}
b_{q_{0}}(X) =& \displaystyle b_{q_{0}}\left(\pi_{1}\left(\cp_{2}(X_{2})\right)\right)\\
 \leq& \sum_{p_{1}+q_{1}=q_{0}} b_{q_{1}}\left(\bigsqcap_{1, J_{1}^{1}}^{1}\left(T(\cp_{2}(X_{2}))\right)\right)\\
\end{align*}
and because of Lemma \ref{cpRep} this is equal  to
\begin{align*}
\displaystyle \sum_{p_{1}+q_{1}=q_{0}} b_{q_{1}}\left(\bigsqcap_{1, J_{1}^{1}}^{1}\left(\cp_{2}((X_{2})_{\delta_{0}, \varepsilon_{0}})
\cup \hdots \cup \cp_{2}((X_{2})_{\delta_{m}, \varepsilon_{m}})) \right)\right) \\
= \sum_{p_{1}+q_{1}=q_{0}} b_{q_{1}}\left(\bigsqcap_{1, J_{1}^{1}}^{1}\left(\overline{\pi_{2}(\overline{(X_{2})_{\delta_{0},
\varepsilon_{0}}}) \cap \hdots \cap \pi_{2}({\overline{(X_{2})_{\delta_{m}, \varepsilon_{m}}}})} \right)\right)\\
= \sum_{p_{1}+q_{1}=q_{0}} b_{q_{1}}\left(\overline{\bigsqcup_{1, J_{1}^{1}}^{1}\left(\pi_{2}(\overline{(X_{2})_{\delta_{0},
\varepsilon_{0}}}) \cap \hdots \cap \pi_{2}({\overline{(X_{2})_{\delta_{m}, \varepsilon_{m}}}}) \right)}\right)\\
\end{align*}
which by Alexander's Duality is bounded by
\begin{align*}
 &\sum_{p_{1}+q_{1}=q_{0}} b_{t_{1} - q_{1} - 1}\left(\bigsqcup_{1,
J_{1}^{1}}^{1}\left(\pi_{2}(\overline{(X_{2})_{\delta_{0}, \varepsilon_{0}}}) \cap \hdots \cap
\pi_{2}({\overline{(X_{2})_{\delta_{m}, \varepsilon_{m}}}}) \right) \cup \partial I_{1}^{t_{1}}\right)\\
 &= \sum_{p_{1}+q_{1}=q_{0}} b_{t_{1} - q_{1} - 1}\left(\bigsqcup_{1, J_{1}^{1}}^{1}\left(\bigcap_{0\leq i \leq
m}\pi_{2}(\overline{(X_{2})_{\delta_{i}, \varepsilon_{i}}}) \right) \cup \partial I_{1}^{t_{1}}\right).\\
\end{align*}
The Mayer-Vietoris Inequalities in turn bound this by
\begin{align*}
\displaystyle  \sum_{p_{1}+q_{1}=q_{0}} \sum_{\hat{J_{1}^{1}} \subset
J_{1}^{1}} b_{t_{1} - q_{1} - |\hat{J_{1}^{1}}|}\left(\bigsqcap_{1, \hat{J_{1}^{1}}}^{1}\left(
\bigcap_{0\leq i \leq m}\pi_{2}(\overline{(X_{2})_{\delta_{i}, \varepsilon_{i}}}) \right) \cup \partial I_{1}^{t_{1}}\right),\\
%\displaystyle = \sum_{p_{1}+q_{1}=q_{1}} \sum_{\hat{J_{1}^{1}} \subset J_{1}^{1}} b_{t_{1} - q -
%|\hat{J_{1}^{1}}|}\left(\bigsqcap_{???}\pi_{2}(\overline{(X_{2})_{\delta_{i}, \varepsilon_{i}}})  \cup \partial I_{1}^{t_{1}}\right).\\
\end{align*}
again using the Mayer-Vietoris Inequalities we have
\begin{align*}
  &b_{t_{1} - q_{1} - |\hat{J_{1}^{1}}|}\left(\bigsqcap_{1, \hat{J_{1}^{1}}}^{1}\left(
\bigcap_{0\leq i \leq m}\pi_{2}(\overline{(X_{2})_{\delta_{i}, \varepsilon_{i}}}) \right) \cup \partial I_{1}^{t_{1}}\right)\\
 &\leq \sum_{J_{2}^{2}\subset \hat{J_{1}^{1}}} \sum_{K_{1} \subset \{1, \ldots, m\}} b_{t_{1} - q_{1} - |\hat{J_{1}^{1}}| - |K_{1}| - |J_{2}^{2}| - 1 }\left(\bigsqcup_{1, J_{2}^{2}}^{1} \bigcup_{i \in K_{1}}
\pi_{2}(\overline{(X_{2})_{\delta_{i}, \varepsilon_{i}}})  \cup \partial I_{1}^{t_{1}}\right)\\
 &= \sum_{J_{2}^{2}\subset \hat{J_{1}^{1}}} \sum_{K_{1} \subset \{1, \ldots, m\}} b_{t_{1} - q_{1} - |\hat{J_{1}^{1}}| - |K_{1}| - |J_{2}^{2}| - 1 }\left(\bigsqcup_{1, J_{2}^{2}}^{1} \bigcup_{i \in K_{1}}
\pi_{2}(\overline{(X_{2})_{\delta_{i}, \varepsilon_{i}}})  \cup \pi_{2,1}(\partial I_{1}^{t_{1}} \times I_{1}^{n_{2}})\right)\\
\displaystyle &= \sum_{J_{2}^{2}\subset \hat{J_{1}^{1}}} \sum_{K_{1} \subset \{1, \ldots, m\}}b_{t_{1} - q_{1} - |\hat{J_{1}^{1}}| - |K_{1}| - |J_{2}^{2}| - 1 }\left(\pi_{2,1}\left(\bigsqcup_{1, J_{2}^{2}}^{1,2} \bigcup_{i \in
 K_{1}} \overline{(X_{2})_{\delta_{i}, \varepsilon_{i}}}  \cup \partial I_{1}^{t_{1}} \times I_{1}^{n_{2}}\right)\right)
\end{align*}
and by Lemma \ref{specSeqAppli} we have
\begin{align*}
&b_{t_{1} - q_{1} - |\hat{J_{1}^{1}}| - |K_{1}| - |J_{2}^{2}| - 1 }\left(\pi_{2,1}\left(\bigsqcup_{1, J_{2}^{2}}^{1,2} \bigcup_{i \in
K_{1}} \overline{(X_{2})_{\delta_{i}, \varepsilon_{i}}}  \cup \partial I_{1}^{t_{1}} \times I_{1}^{n_{2}}\right)\right)\\
&\leq \sum_{p_{2} + q_{2} = t_{1} - q_{1} - |\hat{J_{1}^{1}}| - |K_{1}| - |J_{2}^{2}| - 1
}b_{q_{2}}\left(\bigsqcap_{2}^{2}\left(\bigsqcup_{1, J_{2}^{2}}^{1,2} \bigcup_{i \in K_{1}} \overline{(X_{2})_{\delta_{i},
\varepsilon_{i}}}  \cup \partial I_{1}^{t_{1}} \times I_{1}^{n_{2}}\right)\right)\\
&= \sum_{p_{2} + q_{2} = t_{1} - q_{1} - |\hat{J_{1}^{1}}| - |K_{1}| - |J_{2}^{2}| - 1
}b_{q_{2}}\left(\bigsqcap_{2}^{2}\left(\bigsqcup_{1, J_{2}^{2}}^{1,2} \bigcup_{i \in K_{1}} \overline{(X_{2})_{\delta_{i},
\varepsilon_{i}}}  \cup B_{2}^{2}\right)\right)
\end{align*}
\end{proof}

%%%%%%%%%%%%%%%%%%%%%%%%%%%%%%%%%%%%%%%%%%%%%%%%%%%%%%%%%%%%%%%%%%%%%%%%%%%%%%%%%%%%%%%%
\subsection{Three Quantifiers - $\exists\> \forall\> \exists$}
%%%%%%%%%%%%%%%%%%%%%%%%%%%%%%%%%%%%%%%%%%%%%%%%%%%%%%%%%%%%%%%%%%%%%%%%%%%%%%%%%%%%%%%%

\begin{lem}
\begin{align*}
b_{q_{0}}(X) = &\displaystyle b_{q_{0}}\left(\pi_{1}\left({\rm cp}_{2}(\pi_{3}(X_{3}))\right)\right)\\
 \leq &\sum_{p_{1}+q_{1}=q_{0}} \sum_{\hat{J_{1}^{1}} \subset J_{1}^{1}}
 \sum_{J_{2}^{2}\subset \hat{J_{1}^{1}}} \sum_{K_{1} \subset \{1, \ldots, m\}}
 \sum_{p_{2} + q_{2} = t_{1} - q_{1} - |\hat{J_{1}^{1}}| - |K_{1}| - |J_{2}^{2}| - 1 }\\
 &\sum_{1 \leq k_{2} \leq p_{2} +1} \sum_{\widehat{J}_{2}^{2}\subset \{1,\ldots, p_{2}+1\}, |\widehat{J}_{2}^{2}|=k_{2}}
  \sum_{1 \leq s_{2} \leq q_{2}+k_{2}+1+|K_{1}|} \\
&\sum_{J_{2}^{1} \subset J_{1}^{1} \times \widehat{J}_{2}^{2}, J_{2}^{2} \subset J_{1}^{2} \times \widehat{J}_{2}^{2}, K_{2} \subset
K_{1},  |J_{2}^{1}| + |J_{2}^{2}|+ |K_{1}| = s_{2}}\\
&\sum_{p_{3}+q_{3}=t_{2}-q_{2}-k_{2}+s_{2}-1}
b_{q_{3}}\left(\bigsqcap_{3}^{3}\left(\bigsqcup_{2, J_{2}^{1}}^{1,3}\left( \bigcup_{i \in K_{2}} (X_{3})_{\delta_{i}, \varepsilon_{i}}
\right) \cup \bigsqcup_{2, J_{2}^{2}}^{2,3}B_{3}^{2}\cup B_{3}^{3}\right)\right)
\end{align*}
where $J_{1}^{1} = \{ 1, \ldots, p_{1}+1\}$.
\end{lem}

\begin{proof}
Working from the results in the previous section, we have a summand involving
\begin{align*}
b_{q_{2}}\left(\bigsqcap_{2}^{2}\left(\bigsqcup_{1, J_{2}^{2}}^{1,2} \bigcup_{i \in K_{1}} \overline{(X_{2})_{\delta_{i},
\varepsilon_{i}}}  \cup B_{2}^{2}\right)\right).
\end{align*}
From De Morgan's Law and Alexander's duality, this is equal to
\begin{align*}
&b_{q_{2}}\left(\overline{\bigsqcup_{2}^{2}\left(\bigsqcap_{1, J_{2}^{2}}^{1,2} \bigcap_{i \in K_{1}} (X_{2})_{\delta_{i},
\varepsilon_{i}}  \cap \overline{B_{2}^{2}}\right)}\right)\\
&= b_{t_{2}-q_{2}-1}\left(\bigsqcup_{2}^{2}\left(\bigsqcap_{1, J_{2}^{2}}^{1,2} \bigcap_{i \in K_{1}} (X_{2})_{\delta_{i},
\varepsilon_{i}}  \cap \overline{B_{2}^{2}}\right)\cup \partial(I_{2}^{t_{2}} \times I_{1}^{n_{2}(p_{2}+1)})\right).
\end{align*}
By the Mayer-Vietoris inequality the preceding is bounded by
\begin{align*}
\begin{gathered}
\sum_{1 \leq k_{2} \leq p_{2} +1} \sum_{\widehat{J}_{2}^{2}\subset \{1,\ldots, p_{2}+1\}, |\widehat{J}_{2}^{2}|=k_{2}}\\
b_{t_{2}-q_{2}-k_{2}}\left(\bigsqcap_{2, \widehat{J}_{2}^{2}}^{2}\left(\bigsqcap_{1, J_{2}^{2}}^{1,2} \bigcap_{i \in K_{1}}
(X_{2})_{\delta_{i}, \varepsilon_{i}}  \cap \overline{B_{2}^{2}}\right)\cup \partial(I_{2}^{t_{2}} \times
I_{1}^{n_{2}(p_{2}+1)})\right),
\end{gathered}
\end{align*}
with the summand of this expression being equal to
\begin{align*}
b_{t_{2}-q_{2}-k_{2}}\left(\bigsqcap_{2, J_{1}^{1} \times \widehat{J}_{2}^{2}}^{1}\left( \bigcap_{i \in K_{1}} (X_{2})_{\delta_{i},
\varepsilon_{i}}  \right) \cap \bigsqcap_{2, J_{1}^{2} \times \widehat{J}_{2}^{2}} ^{2}\overline{B_{2}^{2}}\cup \partial(I_{2}^{t_{2}}
\times I_{1}^{n_{2}(p_{2}+1)})\right),
\end{align*}
with $J_{1}^{2}=\{1\}$.  By the Mayer-Vietoris inequality, this is bounded by
\begin{align*}
\sum_{1 \leq s_{2} \leq q_{2}+k_{2}+1+|K_{1}|}
\sum_{J_{2}^{1} \subset J_{1}^{1} \times \widehat{J}_{2}^{2}, J_{2}^{2} \subset J_{1}^{2} \times \widehat{J}_{2}^{2}, K_{2} \subset
K_{1},  |J_{2}^{1}| + |J_{2}^{2}|+ |K_{1}| = s_{2}}\\
b_{t_{2}-q_{2}-k_{2}+s_{2}-1}\left(\bigsqcup_{2, J_{2}^{1}}^{1}\left( \bigcup_{i \in K_{2}} (X_{2})_{\delta_{i}, \varepsilon_{i}}
\right) \cup \bigsqcup_{2, J_{2}^{2}} ^{2}\overline{B_{2}^{2}}\cup \partial(I_{2}^{t_{2}} \times I_{1}^{n_{2}(p_{2}+1)})\right),
\end{align*}

We have $X_{2} = \pi_{3}(X_{3})$, and $\pi_{3}(X_{3})$ is represented by $\pi_{3}((X_{3})_{\delta, \varepsilon})$, so the summand of
the previous expression is equal to
\begin{align*}
b_{t_{2}-q_{2}-k_{2}+s_{2}-1}\left(\bigsqcup_{2, J_{2}^{1}}^{1}\left( \bigcup_{i \in K_{2}} \pi_{3}((X_{3})_{\delta_{i},
\varepsilon_{i}})  \right) \cup \right.
\left.\bigsqcup_{2, J_{2}^{2}} ^{2}\pi_{3,1}(B_{3}^{2})\cup \pi_{3,2}(\partial(I_{2}^{t_{2}} \times I_{1}^{n_{2}(p_{2}+1)}) \times
I_{1}^{n_{3}})\right)
\end{align*}
which is equal to
\begin{align*}
b_{t_{2}-q_{2}-k_{2}+s_{2}-1}\left(\pi_{3,2}\left(\bigsqcup_{2, J_{2}^{1}}^{1,3}\left( \bigcup_{i \in K_{2}} (X_{3})_{\delta_{i},
\varepsilon_{i}}  \right) \cup \right.\right.
\left.\left.\bigsqcup_{2, J_{2}^{2}} ^{2,3}B_{3}^{2}\cup B_{3}^{3}\right)\right),
\end{align*}
which in turn is bounded by
\begin{align*}
\sum_{p_{3}+q_{3}=t_{2}-q_{2}-k_{2}+s_{2}-1}
b_{q_{3}}\left(\bigsqcap_{3}^{3}\left(\bigsqcup_{2, J_{2}^{1}}^{1,3}\left( \bigcup_{i \in K_{2}} (X_{3})_{\delta_{i}, \varepsilon_{i}}
\right) \cup \bigsqcup_{2, J_{2}^{2}}^{2,3}B_{3}^{2}\cup B_{3}^{3}\right)\right),
\end{align*}

\end{proof}

%%%%%%%%%%%%%%%%%%%%%%%%%%%%%%%%%%%%%%%%%%%%%%%%%%%%%%%%%%%%%%%%%%%%%%%%%%%%%%%%%%%%%%%%
\subsection{Arbitrary number of Quantifiers}
%%%%%%%%%%%%%%%%%%%%%%%%%%%%%%%%%%%%%%%%%%%%%%%%%%%%%%%%%%%%%%%%%%%%%%%%%%%%%%%%%%%%%%%%
In the following, we take $Q_{1} =\exists$. The other case can be computed similarly, or calculated directly using Alexander's
duality.
We follow the general idea used in the two and three quantifier cases, and proceed using induction.
\begin{thm}[Main Result A]\label{arbQuant}
For any $i$, let $\tau_{i} = \pi_{i}$ if $i$ is odd and $\tau_{i} = cp_{i}$ if $i$ is even.  Let $\overbrace{\cdot}$ denote
complement if $i$ is even, and have no effect if $i$ is odd.  Then for any $i$:
\begin{align}\label{arbQuantGen}
%\begin{split}
&b_{q_{0}}(X) \leq b_{q_{0}}(\pi_{1}(\cp_{2}(\pi_{3}(\ldots(\tau_{i}(X_{i})))))) \nonumber\\
&\leq \sum_{p_{1}+q_{1}=q_{0}} \sum_{\hat{J_{1}^{1}} \subset J_{1}^{1}}
 \sum_{J_{2}^{2}\subset \hat{J_{1}^{1}}} \sum_{K_{1} \subset \{1, \ldots, m\}}\nonumber\\
 &\sum_{p_{2} + q_{2} = t_{1} - q_{1} - |\hat{J_{1}^{1}}| - |K_{1}| - |J_{2}^{2}| - 1 }
\sum_{1\leq k_{2} \leq p_{2}+1}
\sum_{\hat{J}_{2}^{2} \subset \{1,\ldots,p_{2}+1\}, |\hat{J}_{2}^{2}| = k_{2}}\nonumber\\
&\sum_{1 \leq s_{2} \leq q_{2} + k_{2} + 1 + |K_{1}|}
\sum_{J_{2}^{1} \subset J_{1}^{1} \times \hat{J}_{2}^{2}, J_{2}^{2} \subset J_{1}^{2} \times \hat{J}_{2}^{2},|J_{2}^{1}| +
|J_{2}^{2}| = s_{2}}
\sum_{p_{3}+q_{3} = t_{2} - q_{2} - k_{2} + s_{2} - 1} \hdots \nonumber\\
&\hdots \sum_{1 \leq k_{i-1} \leq p_{i-1} + 1}
\sum_{\hat{J}_{i-1}^{i-1} \subset \{1,\ldots,p_{i-1}+1\},|\hat{J}_{i-1}^{i-1}|=k_{i-1}}
\sum_{1 \leq s_{i-1} \leq q_{i-1} + k_{i-1}+1+|K_{i-2}|}\\
&\sum_{J_{i-1}^{1} \subset J_{i-2}^{1} \times \hat{J}_{i-1}^{i-1},\ldots, J_{i-1}^{i-1} \subset J_{i-2}^{i-1} \times
\hat{J}_{i-1}^{i-1}, K_{i-1} \subset K_{i-2}, |J_{i-1}^{1}|+\hdots+|J_{i-1}^{i-1}| + |K_{i-1}| = s_{i-1}}\nonumber\\
&\sum_{p_{i}+q_{i} = t_{i-1} - q_{i-1} - k_{i-1}+s_{i-1}-1}\nonumber\\
&b_{q_{i}}\left(\bigsqcap_{i}^{i}\left(\bigsqcup_{i-1,J_{i-1}^1}^{1,i}\left(\bigcup_{j \in K_{i-1}}
\overbrace{(X_{i})_{\delta_{j}, \varepsilon_{j}}}\right) \cup \bigcup_{2 \leq r \leq i-1} \bigsqcup_{i-1,J_{i-1}^{r}}^{r,i}B_{i}^{r}
\cup B_{i}^{i}\right)\right),\nonumber
%\end{split}
\end{align}
where $J_{1}^{1}=\{1,\ldots,p_{1}+1\}$.
\end{thm}

\begin{proof}
By induction on $i$.  We have dealt with the cases $i = 1,2,3$ above.  Now, suppose Equation (\ref{arbQuantGen}) is true, and $i$ is odd.
Then, from De Morgan's Law (Lemma \ref{DeMorgan}) and Alexander's Duality (Lemma \ref{Alexander}) we have:
\begin{align*}
&b_{q_{i}}\left(\bigsqcap_{i}^{i}\left(\bigsqcup_{i-1,J_{i-1}^1}^{1,i}\left(\bigcup_{j \in K_{i-1}}(X_{i})_{\delta_{j},
\varepsilon_{j}}\right) \cup \bigcup_{2 \leq r \leq i-1} \bigsqcup_{i-1,J_{i-1}^{r}}^{r,i}B_{i}^{r} \cup B_{i}^{i}\right)\right) \\
&=b_{q_{i}}\left(\overline{\left(\bigsqcup_{i}^{i}\left(\bigsqcap_{i-1,J_{i-1}^1}^{1,i}\left(\bigcap_{j \in
K_{i-1}}\overline{\left((X_{i})_{\delta_{j}, \varepsilon_{j}}\right)}\right) \cap \bigcap_{2 \leq r \leq i-1}
\bigsqcap_{i-1,J_{i-1}^{r}}^{r,i}\overline{B_{i}^{r}} \cup \overline{B_{i}^{i}}\right)\right)}\right)  \\
&\leq b_{t_{i}-q_{i}-1}\Bigg(\bigsqcup_{i}^{i}\left(\bigsqcap_{i-1,J_{i-1}^{1}}^{1,i}\left(\bigcap_{j \in
K_{i-1}}\overline{\left((X_{i})_{\delta_{j}, \varepsilon_{j}}\right)}\right) \cap \bigcap_{2 \leq r \leq i-1}
\bigsqcap_{i-1,J_{i-1}^{r}}^{r,i} \overline{B_{i}^{r}} \cap \overline{B_{i}^{i}}\right)
\cup \\
 &\cup \partial\left(I_{i}^{n_{0}+(p_{1}+1)n_{1}} \times \hdots \times I_{1}^{(p_{i}+1)n_{i}}\right)\Bigg).
\end{align*}

Using the Mayer-Vietoris Inequality (Lemma \ref{MVGen}), this is bounded by
\begin{align*}
&\sum_{1 \leq k_{i} \leq p_{i}+1} \sum_{\hat{J}_{i}^{i} \subset \{1, \ldots, p_{i}+1\}, |\hat{J}_{i}^{i}|=k_{i}}
b_{t_{i}-q_{i}-k_{i}}\left(\bigsqcap_{i,\hat{J}_{i}^{i}}^{i}\left(\bigsqcap_{i-1,J_{i-1}^{1}}^{1,i}\left(\bigcap_{j \in
K_{i-1}}\overline{\left((X_{i})_{\delta_{j}, \varepsilon_{j}}\right)}\right) \cap\right.\right.\\
 &\left.\left.\cap
 \bigcap_{2 \leq r \leq i-1}
\bigsqcap_{i-1,J_{i-1}^{r}}^{r,i} \overline{B_{i}^{r}} \cap \overline{B_{i}^{i}}\right)
\cup  \partial\left(I_{i}^{n_{0}+(p_{1}+1)n_{1}} \times \hdots \times I_{1}^{(p_{i}+1)n_{i}}\right)\right),
\end{align*}
and, by Lemma \ref{combProd}, the summand is bounded by
\begin{align*}
&b_{t_{i}-q_{i}-k_{i}}\left(\bigsqcap_{i,J_{i-1}^{1}\times \hat{J}_{i}^{i}}^{1}\left(\bigcap_{j \in
K_{i-1}}\overline{\left((X_{i})_{\delta_{j}, \varepsilon_{j}}\right)}\right) \cap \bigcap_{2 \leq r \leq i}
\bigsqcap_{i,J_{i-1}^{r}\times \hat{J}_{i}^{i}}^{r}\overline{B_{i}^{r}}\cup \right.\\
&\left. \cup \partial \left(I_{i}^{n_{0}+(p_{1}+1)n_{1}} \times \hdots \times I_{1}^{(p_{i}+1)n_{i}}\right)\right),
\end{align*}
where $J_{i-1}^{i} = \{1\}$.  Using the Mayer-Vietoris Inequality (Lemma \ref{MVGen}), the last expression does not exceed
\begin{align*}
%\begin{split}
&\sum_{1 \leq s_{i} \leq q_{i}+k_{i}+1 + |K_{i-1}|}\sum_{J_{i}^{1}\subset J_{i-1}^{1} \times \hat{J}_{i}^{i},\ldots,J_{i}^{i} \subset J_{i-1}^{i} \times \hat{J}_{i}^{i}, K_{i} \subset
K_{i-1},|J_{i}^{1}| + \hdots + |J_{i}^{i}|+ |K_{i}| = s_{i}}\\
&b_{t_{i}-q_{i}-k_{i}+s_{i}-1}\left(\bigsqcup_{i,J_{i}^{1}}^{1}\left(\bigcup_{j \in K_{i}}\overline{\left((X_{i})_{\delta_{j},
\varepsilon_{j}}\right)}\right)\cup \bigcup_{2 \leq r \leq i} \bigsqcup_{i,J_{i}^{r}}^{r}\overline{B_{i}^{r}} \right.\\
&\cup \left.\partial \left(I_{i}^{n_{0}+(p_{1}+1)n_{1}} \times \hdots \times I_{1}^{(p_{i}+1)n_{i}}\right)\right).
%\end{split}
\end{align*}

We have $X_{i} = cp_{i+1}(X_{i+1})$, and by Lemma \ref{cpRep} $cp_{i+1}(X_{i+1})$ is represented by
$cp((X_{i+1})_{\delta, \varepsilon})$, so the summand of the previous expression is equal to
\begin{align*}
&b_{t_{i}-q_{i}-k_{i}+s_{i}-1}\left(\bigsqcup_{i,J_{i}^{1}}^{1}\left(
\bigcup_{j \in K_{i}}\overline{\left(cp_{i+1}((X_{i+1})_{\delta_{j}, \varepsilon_{j}})\right)}\right)\cup\right.\\
&\left. \bigcup_{2 \leq r \leq i} \bigsqcup_{i,J_{i}^{r}}^{r}\overline{B_{i}^{r}} \cup \partial \left(I_{i}^{n_{0}+(p_{1}+1)n_{1}}
\times \hdots \times I_{1}^{(p_{i}+1)n_{i}}\right)\right)\\
&= b_{t_{i}-q_{i}-k_{i}+s_{i}-1}\left(\bigsqcup_{i,J_{i}^{1}}^{1}\left(
\bigcup_{j \in K_{i}}\left(\pi_{i+1}(\overline{(X_{i+1})_{\delta_{j}, \varepsilon_{j}}})\right)\right)\cup\right.\\
&\left. \bigcup_{2 \leq r \leq i} \bigsqcup_{i,J_{i}^{r}}^{r}\overline{B_{i}^{r}} \cup \partial \left(I_{i}^{n_{0}+(p_{1}+1)n_{1}}
\times \hdots \times I_{1}^{(p_{i}+1)n_{i}}\right)\right)\\
&= b_{t_{i}-q_{i}-k_{i}+s_{i}-1}
\left(\bigsqcup_{i,J_{i}^{1}}^{1}\bigcup_{j \in K_{i}}\left(\pi_{i+1}(\overline{(X_{i+1})_{\delta_{j},
\varepsilon_{j}}})\right)\cup\right. \\
&\left.\bigcup_{2 \leq r \leq i} \bigsqcup_{i,J_{i}^{r}}^{r}\pi_{i+1,r-1}(B_{i+1}^{r}) \cup  \pi_{i+1,i}\left(\partial \left(I_{i}^{n_{0}+(p_{1}+1)n_{1}} \times \hdots \times I_{1}^{(p_{i}+1)n_{i}}\right)\times
 I_{1}^{n_{i}+1}\right)\right),
\end{align*}
and bringing the projection to the front, this is equal to
\begin{align*}
b_{t_{i}-q_{i}-k_{i}+s_{i}-1}\left(\pi_{i+1,i}\left(\bigsqcup_{i,J_{i}^{1}}^{1,i+1}\bigcup_{j \in
K_{i}}\left(\overline{(X_{i+1})_{\delta_{j}, \varepsilon_{j}}}\right) \cup \bigcup_{2 \leq r \leq i}
\bigsqcup_{i,J_{i}^{r}}^{r,i+1}B_{i+1}^{r} \cup B_{i+1}^{i+1}\right)\right).
\end{align*}
Finally, from Lemma \ref{specSeqAppli}, this is bounded by
\begin{align*}
&\sum_{p_{i+1}+q_{i+1} = t_{i}-q_{i}-k_{i}+s_{i}-1}
&b_{q_{i+1}}\left(\bigsqcap_{i+1}^{i+1}\left(\bigsqcup_{i,J_{i}^{1}}^{1,i+1}\bigcup_{j \in K_{i}}\left(\overline{(X_{i+1})_{\delta_{j},
\varepsilon_{j}}}\right) \cup \bigcup_{2 \leq r \leq i}\bigsqcup_{i,J_{i}^{r}}^{r,i+1}B_{i+1}^{r}\cup B_{i+1}^{i+1}\right)\right).
\end{align*}
The case when $i$ is even follows directly by replacing $\bigcup_{j \in K_{i-1}}(X_{i})_{\delta_{j}, \varepsilon_{j}}$ with
$\bigcup_{j \in K_{i-1}}\overline{(X_{i})_{\delta_{j}, \varepsilon_{j}}}$ at the start of the proof, immediately changing to
\begin{equation*}
\displaystyle \bigcap_{j \in K_{i-1}}(X_{i})_{\delta_{j}, \varepsilon_{j}}
\end{equation*}
 in the next line and carrying this through to the stage where projections are introduced, replacing $\bigcup_{j \in
 K_{i}}(X_{i})_{\delta_{j}, \varepsilon_{j}}$ with $\bigcup_{j \in K_{i}}\pi_{i+1}((X_{i+1})_{\delta_{j}, \varepsilon_{j}})$, and
 finishing with an expression involving $\bigcup_{j \in K_{i}}(X_{i+1})_{\delta_{j}, \varepsilon_{j}}$.
\end{proof}

%%%%%%%%%%%%%%%%%%%%%%%%%%%%%%%%%%%%%%%%%%%%%%%%%%%%%%%%%%%%%%%%%%%%%%%%%%%%%%%%%%%%%%%%
\section{Upper Bounds}
%%%%%%%%%%%%%%%%%%%%%%%%%%%%%%%%%%%%%%%%%%%%%%%%%%%%%%%%%%%%%%%%%%%%%%%%%%%%%%%%%%%%%%%%
We have succeeded in finding an expression to bound each Betti number of $X$, in terms of some quantifier-free formula involving a
set we can bound using the tools in the above chapters.  It now remains to explicitly deduce a formula using the functions supplied
by the complexity definition.

Theorem \ref{arbQuant} tells us that $b_{q_{0}}(X)$ is bounded by:
\begin{align*}
\begin{gathered}
\sum_{p_{1}+q_{1}=q_{0}} \sum_{\hat{J_{1}^{1}} \subset J_{1}^{1}}
 \sum_{J_{2}^{2}\subset \hat{J_{1}^{1}}} \sum_{K_{1} \subset \{1, \ldots, m\}}\\
 \sum_{p_{2} + q_{2} = t_{1} - q_{1} - |\hat{J_{1}^{1}}| - |K_{1}| - |J_{2}^{2}| - 1 }
\sum_{1\leq k_{2} \leq p_{2}+1}
\sum_{\hat{J}_{2}^{2} \subset \{1,\ldots,p_{2}+1\}, |\hat{J}_{2}^{2}| = k_{2}}\\
\sum_{1 \leq s_{2} \leq q_{2} + k_{2} + 1 + |K_{1}|}
\sum_{J_{2}^{1} \subset J_{1}^{1} \times \hat{J}_{2}^{2}, J_{2}^{2} \subset J_{1}^{2} \times \hat{J}_{2}^{2},|J_{2}^{1}| +
|J_{2}^{2}| = s_{2}}
\sum_{p_{3}+q_{3} = t_{2} - q_{2} - k_{2} + s_{2} - 1} \hdots \\
\hdots \sum_{1 \leq k_{i-1} \leq p_{i-1} + 1}
\sum_{\hat{J}_{i-1}^{i-1} \subset \{1,\ldots,p_{i-1}+1\},|\hat{J}_{i-1}^{i-1}|=k_{i-1}}
\sum_{1 \leq s_{i-1} \leq q_{i-1} + k_{i-1}+1+|K_{i-2}|}\\
\sum_{J_{i-1}^{1} \subset J_{i-2}^{1} \times \hat{J}_{i-1}^{i-1},\ldots, J_{i-1}^{i-1} \subset J_{i-2}^{i-1} \times
\hat{J}_{i-1}^{i-1}, K_{i-1} \subset K_{i-2}, |J_{i-1}^{1}|+\hdots+|J_{i-1}^{i-1}| + |K_{i-1}| = s_{i-1}}\\
\sum_{p_{i}+q_{i} = t_{i-1} - q_{i-1} - k_{i-1}+s_{i-1}-1}\\
b_{q_{i}}\left(\bigsqcap_{i}^{i}\left(\bigsqcup_{i-1,J_{i-1}^1}^{1,i}\left(\bigcup_{j \in K_{i-1}}
\overbrace{(X_{i})_{\delta_{j}, \varepsilon_{j}}}\right) \cup \bigcup_{2 \leq r \leq i-1} \bigsqcup_{i-1,J_{i-1}^{r}}^{r,i}B_{i}^{r}
\cup B_{i}^{i}\right)\right).
\end{gathered}
\end{align*}

We start by bounding the Betti numbers of the summand.  To do this we prove the following lemma, which allows us to use the results
in the previous sections.

\begin{lem}
 Let $X_{\nu}$ be defined by a quantifier-free Boolean formula, consisting of $s$ atoms of the type $f_{i}>0$ or $f_{i}=0$.  Let $t_{\nu}$ be defined as in Definition \ref{tiAndProj}. Then the set
\begin{equation*}
\bigsqcap_{\nu}^{\nu}\left(\bigsqcup_{\nu-1,J_{\nu-1}^1}^{1,\nu}\left(\bigcup_{j \in K_{\nu-1}}
\overbrace{(X_{\nu})_{\delta_{j}, \varepsilon_{j}}}\right) \cup \bigcup_{2 \leq r \leq \nu-1}
\bigsqcup_{\nu-1,J_{\nu-1}^{r}}^{r,\nu}B_{\nu}^{r} \cup B_{\nu}^{\nu}\right) \subset \mathbb{R}^{t_\nu}
\end{equation*}
is defined by a quantifier-free Boolean formula with no negations, having at most
\begin{align*}
\begin{gathered}
    \Bigg(\left(2t_{\nu-2}+1\right)4s\left(n_{0} + \hdots + n_{\nu}\right) + 2t_{\nu-1} + 2n_{\nu}  +\Big((4t_{r-1}+2\left(n_{r}+\hdots+n_{\nu}\right))\cdot\\
    \cdot(2t_{\nu-2}+1)
     + 2t_{\nu-1} + 2n_{\nu}\Big)(\nu-2)\Bigg)\left(t_{\nu-1}+1\right) \leq st_{\nu-1}^{O(1)}
     \end{gathered}
    \end{align*}
atoms.
\end{lem}

\begin{proof}
Our method consists of simple counting and calculation.
\begin{itemize}
    \item The set $(X_{\nu})_{\delta_{j}, \varepsilon_{j}}$ is defined by a Boolean formula with no negations consisting of $4s$
        non-strict inequalities of the type $f+\delta_{j}$, $f-\delta_{j}$, $f+\varepsilon_{j}$, and $f-\varepsilon_{j}$.

    \item The set $\overbrace{(X_{\nu})_{\delta_{j}, \varepsilon_{j}}}$ is either the same as the above, or its negation, which
        is a Boolean formula with no negations consisting of $4s$ strict inequalities.  Either
        way, the set is defined by only non-strict or only strict inequalities, and no negations.

    \item We have $\{1, \ldots, m\} \supset K_{1} \supset \hdots \supset K_{\nu}$, so $|K_{i}| \leq m$ for all $i$.  We define
        $m$ when constructing $T$, and from Theorem \ref{STiso} we know if $m$ is at least the dimension of the space, then $T$
        is a `good' approximation.  The maximum dimension we work in is $n_{0} + \hdots + n_{\nu}$. Therefore
        \begin{equation*}
        \bigcup_{j \in K_{\nu-1}} \overbrace{(X_{\nu})_{\delta_{j}, \varepsilon_{j}}}
        \end{equation*}
        is a set defined by a Boolean formula with no negations consisting of at most $4s(n_{0} + \hdots + n_{\nu})$ only strict
        or only non-strict inequalities.

  \item The set
  \begin{equation*}
  \bigsqcup_{\nu-1,J_{\nu-1}^1}^{1,\nu}\left(\bigcup_{j \in K_{\nu-1}} \overbrace{(X_{\nu})_{\delta_{j}, \varepsilon_{j}}}\right)
  \subset \mathbb{R}^{t_{\nu-1}+n_{\nu}}
   \end{equation*}
   is defined by a Boolean formula with no negations having $|J_{\nu - 1}^{1}|4s(n_{0} + \hdots + n_{\nu}) \leq s_{\nu-1}4s(n_{0}
   + \hdots + n_{\nu}) \leq (2t_{\nu-2}+1)4s(n_{0} + \hdots + n_{\nu})$ atoms, and at most
   $2t_{\nu-1}+2n_{\nu}$ linear atoms (defining $I_{1}^{t_{\nu-1}+n_{\nu}}$).
  \item For any $2 \leq r \leq \nu$, the set $B_{\nu}^{\nu} \subset \mathbb{R}^{t_{r-1}+n_{r}}$ is defined by a Boolean formula
      with no negations having $4t_{r-1} + 2n_{r}$ linear atomic inequalities.
  \item Therefore, all sets of the kind $B_{j}^{r}$ for $j \geq r$ are defined by Boolean formulae  with no negations having
      $4t_{r-1}+2(n_{r} + \hdots + n_{j})$ linear inequalities.  In particular the set $B_{\nu}^{r} \subset
      \mathbb{R}^{t_{r-1}+n_{r}+\hdots+n_{\nu}}$ is defined by $4t_{r-1}+2(n_{r}+\hdots+n_{\nu})$ linear atomic inequalities.
  \item For any $2 \leq r \leq \nu -1$, the set $\bigsqcup_{\nu-1,J_{\nu-1}^{r}}^{r,\nu}B_{\nu}^{r} \subset
      \mathbb{R}^{t_{\nu-1}+n_{\nu}}$ is defined by a formula with no negations having at most
      \begin{align*}
      &(4t_{r-1}+2(n_{r}+\hdots+n_{\nu}))|J_{\nu-1}^{r}| + 2t_{\nu-1} + 2n_{\nu} \\
      \leq& (4t_{r-1}+2(n_{r}+\hdots+n_{\nu}))s_{\nu-1} + 2t_{\nu-1} + 2n_{\nu} \\
      \leq& (4t_{r-1}+2(n_{r}+\hdots+n_{\nu}))(2t_{\nu-2}+1) + 2t_{\nu-1} + 2n_{\nu}
      \end{align*}
      linear atoms.
  \item The set $\bigcup_{2 \leq r \leq \nu-1} \bigsqcup_{\nu-1,J_{\nu-1}^{r}}^{r,\nu}B_{\nu}^{r} \subset
      \mathbb{R}^{t_{\nu-1}+n_{\nu}}$ is therefore defined by a Boolean formula with no negations having at most
      \begin{equation*}
      ((4t_{r-1}+2(n_{r}+\hdots+n_{\nu}))(2t_{\nu-2}+1) + 2t_{\nu-1} + 2n_{\nu})(\nu-2)
      \end{equation*}
      atoms.
  \item Therefore, the set
    \begin{equation*}
    \bigsqcap_{\nu}^{\nu}\left(\bigsqcup_{\nu-1,J_{\nu-1}^1}^{1,\nu}\left(\bigcup_{j \in K_{\nu-1}}
\overbrace{(X_{\nu})_{\delta_{j}, \varepsilon_{j}}}\right) \cup \bigcup_{2 \leq r \leq \nu-1}
\bigsqcup_{\nu-1,J_{\nu-1}^{r}}^{r,\nu}B_{\nu}^{r} \cup B_{\nu}^{\nu}\right)
    \end{equation*}
    is defined by a Boolean formula with no negations having at most
    \begin{align*}
\begin{gathered}
    \Bigg(\left(2t_{\nu-2}+1\right)4s\left(n_{0} + \hdots + n_{\nu}\right) + 2t_{\nu-1} + 2n_{\nu}  +\Big((4t_{r-1}+2\left(n_{r}+\hdots+n_{\nu}\right))\cdot\\
    \cdot(2t_{\nu-2}+1)
     + 2t_{\nu-1} + 2n_{\nu}\Big)(\nu-2)\Bigg)\left(t_{\nu-1}+1\right) \leq st_{\nu-1}^{O(1)}
     \end{gathered}
    \end{align*}
    atoms.

\end{itemize}
\end{proof}

We can now use Corollary \ref{arbBoolGen}, along with our particular defined complexity, to find an upper bound for
\begin{equation*}
\bigsqcap_{\nu}^{\nu}\left(\bigsqcup_{\nu-1,J_{\nu-1}^1}^{1,\nu}\left(\bigcup_{j \in K_{\nu-1}}
\overbrace{(X_{\nu})_{\delta_{j}, \varepsilon_{j}}}\right) \cup \bigcup_{2 \leq r \leq \nu-1}
\bigsqcup_{\nu-1,J_{\nu-1}^{r}}^{r,\nu}B_{\nu}^{r} \cup B_{\nu}^{\nu}\right).
\end{equation*}

We now count the number of times this expression is added.  We use the following notation, taken from \cite{ApproxDefSetCompFam}:
\begin{defn}
Let $f, g, h:\mathbb{N}^{l} \to \mathbb{N}$ be three functions.  The expression $f \leq g^{O(h)}$ means that there exists $c \in \mathbb{N}$ such that $f \leq g^{ch}$ everywhere on $\mathbb{N}^{l}$.
\end{defn}

\begin{lem}\label{numAdTerms}
The number of additive terms in \eqref{arbQuant} does not exceed
\begin{equation*}
2^{O(i^{2}(2(n_{0}+\hdots+n_{\nu}))^{i}n_{0}n_{1}\ldots n_{i-2})}.
\end{equation*}
\end{lem}

\begin{proof}
Firstly, we have to count the number of terms introduced by
\begin{equation*}
\sum_{\hat{J_{1}^{1}} \subset J_{1}^{1}}
 \sum_{J_{2}^{2}\subset \hat{J_{1}^{1}}} \sum_{K_{1} \subset \{1, \ldots, m\}}.
 \end{equation*}
 We have $J_{1}^{1}=\{1,\ldots,p_{1}+1\}$, so the number of terms in the first summation is bounded by $2^{p_{1}+1}$, as is that in the
 second summation, and the third summation has no more than $2^{m}$ terms.

 Next we can partition the remaining terms into $i-1$ groups of the kind
\begin{equation*}
\sum_{1 \leq k_{j} \leq p_{j}+1} \sum_{\hat{J}_{j}^{j} \subset \{1, \ldots, p_{j}+1\}, |\hat{J}_{j}^{j}|=k_{j}}
\end{equation*}
\begin{equation*}
\sum_{1 \leq s_{j} \leq q_{j}+k_{j}+1+|K_{i-1}|}\\
\end{equation*}
\begin{equation*}
 \sum_{J_{j}^{1}\subset J_{j-1}^{1}\times \hat{J}_{j}^{j},\ldots, J_{j}^{j}\subset
J_{j-1}^{j}\times \hat{J}_{j}^{j}, K_{j} \subset K_{j-1}, |J_{j}^{1}| +\hdots+|J_{j}^{j}| + |K_{j}|=s_{j}}
\end{equation*}
\begin{equation*}
\sum_{p_{j+1}+q_{j+1} = t_{j} - q_{j} -k_{j} +s_{j} -1},
\end{equation*}
where $1 \leq j \leq i-1$.

The number of terms in
\begin{equation*}
\sum_{1 \leq k_{j} \leq p_{j}+1} \sum_{\hat{J}_{j}^{j} \subset \{1, \ldots, p_{j}+1\}, |\hat{J}_{j}^{j}|=k_{j}}
\end{equation*}
is $2^{p_{j}+1}$.

The number of terms in
\begin{equation*}
\sum_{1 \leq s_{j} \leq q_{j}+k_{j}+1+|K_{i-1}|}
 \end{equation*}
\begin{equation*}
 \sum_{J_{j}^{1}\subset J_{j-1}^{1}\times \hat{J}_{j}^{j},\ldots, J_{j}^{j}\subset
J_{j-1}^{j}\times \hat{J}_{j}^{j}, K_{j} \subset K_{j-1}, |J_{j}^{1}| +\hdots+|J_{j}^{j}| + |K_{j}|=s_{j}}
\end{equation*}
does not exceed $2^{(j+1)(q_{j}+k_{j}+1+|K_{i-1}|)}\leq 2^{(j+1)(q_{j}+k_{j}+1+n_{0}+\hdots+n_{\nu})}$.

The number of terms in
\begin{equation*}
\sum_{p_{j+1}+q_{j+1} = t_{j} - q_{j} -k_{j} +s_{j} -1}
\end{equation*}
does not exceed $t_{j}+1$.

Therefore, it follows that the total number of terms in the $j$th group does not exceed
\begin{equation*}
2^{p_{j}+1+(j+1)(q_{j}+k_{j}+1+n_{0}+\hdots+n_{\nu})}(t_{j}+1).
\end{equation*}

We have $s_{j} \leq q_{j}+k_{j}+1 + |K_{j-1}|$, so $t_{j}-q_{j}-k_{j}+s_{j}-1 \leq t_{j}+|K_{j-1}|$, and therefore $p_{j+1}+q_{j+1}
\leq t_{j}+|K_{j-1}|$.  Hence we have
\begin{align*}
&2^{p_{j}+1+(j+1)(q_{j}+k_{j}+1+n_{0}+\hdots+n_{\nu})}(t_{j}+1) \\
&\leq 2^{p_{j}+1+(j+1)(q_{j}+p_{j}+2+n_{0}+\hdots+n_{\nu})}(t_{j}+1)\\
&\leq 2^{(j+1)(p_{j}+1+q_{j}+p_{j}+2+n_{0}+\hdots+n_{\nu})}(t_{j}+1)\\
&\leq 2^{(j+1)(2p_{j}+q_{j}+3+n_{0}+\hdots+n_{\nu})}(t_{j}+1)\\
&\leq 2^{O(j(t_{j-1}+n_{0}+\hdots+n_{\nu}))}.
\end{align*}

Since $t_{j} = n_{0} + n_{1}(p_{1}+1) + \hdots + n_{j}(p_{j}+1)$, $p_{l}\leq t_{l-1}+|K|$, and therefore $t_{j} \leq
(2|K|)^{j}n_{0}n_{1}\ldots n_{j}+1$ (proof of this follows easily using induction, see Lemma \ref{indproo} below), the number of terms
in the $j$th group does not exceed
\begin{align*}
2^{O(j(2|K|)^{j}n_{0}n_{1}\ldots n_{j-1})} \leq
2^{O(j(2(n_{0}+\hdots+n_{\nu}))^{j}n_{0}n_{1}\ldots n_{j-1})}.
\end{align*}

It follows that the total number of terms in \eqref{arbQuant} does not exceed
\begin{equation*}
2^{O(i^{2}(2(n_{0}+\hdots+n_{\nu}))^{i}n_{0}n_{1}\ldots n_{i-2})}.
\end{equation*}
\end{proof}

\begin{rem}
Thanks to James Davenport for suggesting this improvement during the viva. In the preceding we counted each of the $\Sigma$ separately, for example for
\begin{equation*}
\sum_{\hat{J_{1}^{1}} \subset J_{1}^{1}}
 \sum_{J_{2}^{2}\subset \hat{J_{1}^{1}}}
 \end{equation*}
 we bounded each of the summations by $2^{p_{1}+1}$ and get a total of $4^{p_{1}+1}$.  We could, however rewrite this as
 \begin{equation*}
\sum_{J_{2}^{2}\subset \hat{J_{1}^{1}} \subset J_{1}^{1}}
 \end{equation*}
 and bound this by $3^{p_{1}+1}$, resulting in a slightly better bound. Similar techniques could be used throughout.  This does not produce any difference at the level of $2^{O(\ldots)}$, but may be useful in future work.
\end{rem}

\begin{lem}\label{indproo}
\begin{equation*}
t_{j} \leq (2|K|)^{j}n_{0}n_{1}\ldots n_{j}+1
\end{equation*}
\end{lem}
\begin{proof}
By induction.  Firstly, we have $t_{j} = n_{0} + n_{1}(p_{1}+1) + \hdots + n_{j}(p_{j}+1)$ and $p_{l} \leq t_{l-1}+|K|$.
\begin{itemize}

\item For $j=0$, $t_{0} = n_{0}$

\item For $j=1$,
\begin{itemize}
\item If $n_{0} > 1$ or $n_{1} > 1$
\begin{align*}
t_{1} &= n_{0} + n_{1}(p_{1}+1)\\
&\leq n_{0} + n_{1}(n_{0}+|K|+1)\\
&= n_{0} + n_{1} + n_{0}n_{1} + n_{1}|K|\\
&\leq 2n_{0}n_{1} +n_{1}|K|, \text{ since } n_{0} \text{ and } n_{1} \text{ are positive integers and not both 1}\\
&\leq 2|K|n_{0}n_{1}, \text{ as } |K| \geq 2
\end{align*}

\item If $n_{0}=n_{1}=1$ we have
\begin{align*}
t_{1} &= n_{0} + n_{1}(p_{1}+1)\\
&\leq 1+ 1(1+|K| +1)\\
&= |K| + 3 \text{ and as } |K| \geq 2\\
&\leq 2|K| + 1
\end{align*}
\end{itemize}

\item Assume $t_{j} \leq (2|K|)^{j}n_{0}n_{1}\ldots n_{j}$.  Then
\begin{align*}
t_{j+1} &= t_{j} + n_{j+1}(p_{j+1}+1)\\
&\leq t_{j} + n_{j+1}(t_{j} + |K| +1)\\
&\leq (2|K|)^{j}n_{0}n_{1}\ldots n_{j} + (2|K|)^{j}n_{0}n_{1}\ldots n_{j}n_{j+1} + n_{j+1}|K| + n_{j+1}\\
&\leq 2^{j+1}|K|^{j}n_{0}n_{1}\ldots n_{j}n_{j+1} + n_{j+1}|K|\\
&\leq (2|K|)^{j+1}n_{0}n_{1}\ldots n_{j}n_{j+1}
\end{align*}

\end{itemize}

\end{proof}

We now know how to calculate the Betti numbers of the summand, and we know how many times we are summing, so we only need to use a
result from the previous sections to produce a concrete bound.

We will use Corollary \ref{arbBoolGen}, (which is equivalent to the bound $O(s^{2}d)^{n}$ in the basic polynomial, degree case).

%\begin{defn}\label{omega}
%Let
%\begin{align*}
%\displaystyle \omega(F,G) &= \max_{i}\left(b(\textit{Reali}(f_{i}=\delta, Z_{r}))\right)\\
% &= \max_{i}\left(\frac{\gamma(n, c(f_{i}^2 + \sum_{g \in G} g^{2} + |x|^{2}))}{2}\right);
%\end{align*}
%we then define $\Omega(F,G)$ to be
%\begin{center}
%$\displaystyle\max(\omega(F,G), b(Z_{r})) = \max\left(\omega(F,G), \frac{\gamma(n, c(\sum_{g \in G} g^{2} + |x|^{2} )}{2}\right)$.
%\end{center}
%\end{defn}

\begin{defn}
 For a set of functions $F = \{f_{1}, \ldots, f_{s}\}$ from $H \to \R$, where $H \subset \R^{n}$ define the function
 \begin{align*}
 \displaystyle \Omega(F) &= \max_{1 \leq i \leq s}\left(\frac{\gamma(n, c(f_{i}^2  + |x|^{2}))}{2}\right).
 \end{align*}
\end{defn}

 Note the following is really a Corollary of Theorem \ref{arbQuant}, but is called a Theorem because it is one of the most significant results of this thesis.

\begin{thm}[Main Result B]\label{clsdUpBd}

Let $C$ denote the maximum complexity of the atoms in
\begin{equation*}
Z = \bigsqcap_{\nu}^{\nu}\left(\bigsqcup_{\nu-1,J_{\nu-1}^1}^{1,\nu}\left(\bigcup_{j \in K_{\nu-1}}
\overbrace{(X_{\nu})_{\delta_{j}, \varepsilon_{j}}}\right) \cup \bigcup_{2 \leq r \leq \nu-1}
\bigsqcup_{\nu-1,J_{\nu-1}^{r}}^{r,\nu}B_{\nu}^{r} \cup B_{\nu}^{\nu}\right).
\end{equation*}
Let $F$ be a set containing elements $g$, $|x|^{2}$, where $g$ is any function of complexity $C$ with domain $\mathbb{R}^{t_{\nu}}$.  Let $u_{j} = n_{0}+n_{1}+\hdots+n_{j}$, and $w_{j}=n_{0}n_{1}\hdots n_{j}$.  Then:
\begin{align*}
b_{q_{0}}(X) \leq (2^{\nu^{2}}u_{\nu}^{\nu}sw_{\nu-1})^{O((2u_{\nu})^{\nu}w_{\nu})}\Omega(F).
\end{align*}
\end{thm}

\begin{proof}
We take Theorem \ref{arbQuant} in the case $i=\nu$.  Firstly, we directly apply the bound given in \ref{arbBoolGen} to the set $Z \subset \mathbb{R}^{t_{\nu}}$, which is defined by $st_{\nu-1}^{O(1)}$ atoms, and get
\begin{equation*}
b_{q_{0}}(Z) \leq O((st_{\nu-1}^{O(1)})^{2t_{\nu}}\Omega(F))
\end{equation*}

From Lemma \ref{numAdTerms} we know that we are adding this together
\begin{equation*}
2^{O(\nu^{2}(2(n_{0}+\hdots+n_{\nu}))^{\nu}n_{0}n_{1}\ldots n_{\nu-2})}
\end{equation*}
times, so we have
\begin{align*}
b_{q_{0}}(X) &\leq 2^{O(\nu^{2}(2(n_{0}+\hdots+n_{\nu}))^{\nu}n_{0}n_{1}\ldots n_{\nu-2})}O((st_{\nu-1}^{O(1)})^{2t_{\nu}} \Omega(F))\\
&\leq 2^{O(\nu^{2}(2(n_{0}+\hdots+n_{\nu}))^{\nu}n_{0}n_{1}\ldots n_{\nu-2})}O((st_{\nu-1})^{O(t_{\nu})}\Omega(F))
\end{align*}
We have $t_{j} \leq (2|K|)^{j}n_{0}n_{1}\ldots n_{j} \leq (2(n_{0}+\hdots+n_{\nu}))^{j}n_{0}n_{1}\ldots n_{j}$, so
\begin{align*}
&b_{q_{0}}(X) \leq 2^{O(\nu^{2}(2(n_{0}+\hdots+n_{\nu}))^{\nu}n_{0}n_{1}\ldots n_{\nu-2})}\cdot\\
&\cdot O((s(2(n_{0}+\hdots+n_{\nu}))^{\nu-1}n_{0}n_{1}\ldots n_{\nu-1})^{O((2(n_{0}+\hdots+n_{\nu}))^{\nu}n_{0}n_{1}\ldots n_{\nu})}\Omega(F))
\end{align*}
And substituting $u_{j} = n_{0}+n_{1}+\hdots+n_{j}$, $w_{j}=n_{0}n_{1}\hdots n_{j}$ gives
\begin{align*}
b_{q_{0}}(X) &\leq 2^{O(\nu^{2}(2u_{\nu})^{\nu}w_{\nu-2})}O((s(2u_{\nu})^{\nu-1}w_{\nu-1})^{O((2u_{\nu})^{\nu}w_{\nu})}\Omega(F))\\
&\leq (2^{\nu^{2}}u_{\nu}^{\nu}sw_{\nu-1})^{O((2u_{\nu})^{\nu}w_{\nu})}\Omega(F).
\end{align*}

\end{proof}

In the preceding result, and the following examples, we see a  $u_{\nu}^{\nu}$ factor appearing in a number of places.  This is the consequence of not restricting to the case of $X_{\nu}$ being open or closed, and thus needing to use the $T$ construction to approximate by a compact set.  To construct $T$ we take the union of $m$ sets, and we use $n_{0} + \ldots + n_{\nu} = u_{\nu}$ to bound $m$, which results in this extra factor.  A possible improvement would be to find a tighter bound for $m$.

\begin{eg}[The Degree of a Polynomial]
If the original polynomials have maximum degree $d$, then the polynomials defining $Z$ will also have maximum degree $d$, and $\Omega(F)= O(d^{t_{\nu}}) \leq O(d^{(2(n_{0}+\hdots+n_{\nu}))^{\nu}n_{0}n_{1}\ldots n_{\nu}})$, so
\begin{align*}
b_{q_{0}}(X) \leq (2^{\nu^{2}}u_{\nu}^{\nu}dsw_{\nu-1})^{O((2u_{\nu})^{\nu}w_{\nu})}.
\end{align*}

Note the  difference to the bound  of
\begin{align*}
(2^{\nu^{2}}dsw_{\nu-1})^{O(2^{\nu}w_{\nu})}.
\end{align*}
 in section 8.1 of \cite{BNSASPSets} for the case where $X_{\nu}$ is open or closed.
\end{eg}

\begin{eg}[Pfaffian Functions]
If the functions defining $X_{\nu}$ are Pfaffian functions of order $r$, degree $(\alpha, \beta)$, then the functions defining $Z$ are Pfaffian functions of degree $(\alpha, \beta)$ and order at most $r(2t_{\nu-2}+1)(t_{\nu-1}+1) = O(rt_{\nu-2}t_{\nu-1})$.  We then have
\begin{align*}
\Omega(F) = &2^{O(rt_{\nu-2}t_{\nu-1})(O(rt_{\nu-2}t_{\nu-1})-1)/2}\beta(\alpha+2\beta-1)^{t_{\nu}-1}\cdot\\
&\cdot(\text{min}\{t_{\nu},O(rt_{\nu-2}t_{\nu-1})\}\alpha + 2t_{\nu}\beta + (t_{\nu}-1)\alpha-2t_{\nu}+2)^{O(rt_{\nu-2}t_{\nu-1})}\\
\leq &2^{O(rt_{\nu-2}t_{\nu-1})^{2}}\beta(\alpha+2\beta-1)^{t_{\nu}-1}\cdot\\
&\cdot(\text{min}\{t_{\nu},O(rt_{\nu-2}t_{\nu-1})\}\alpha + 2t_{\nu}\beta + (t_{\nu}-1)\alpha-2t_{\nu}+2)^{O(rt_{\nu-2}t_{\nu-1})}\\
\leq &2^{O(rt_{\nu-2}t_{\nu-1})^{2}}(\alpha+\beta)^{O(t_{\nu})}\cdot\\
&\cdot(\text{min}\{t_{\nu},O(rt_{\nu-2}t_{\nu-1})\}\alpha + 2t_{\nu}\beta + (t_{\nu}-1)\alpha-2t_{\nu}+2)^{O(rt_{\nu-2}t_{\nu-1})}\\
\leq &2^{O(rt_{\nu-2}t_{\nu-1})^{2}}(\alpha+\beta)^{O(t_{\nu})}(t_{\nu}(\alpha+\beta))^{O(rt_{\nu-2}t_{\nu-1})}.
\end{align*}
We have
$t_{j} \leq (2(n_{0}+\hdots+n_{\nu}))^{j}n_{0}n_{1}\ldots n_{j}= (2u_{\nu})^{j}w_{j}$, so
\begin{align*}
\Omega(F) \leq &2^{O(r(2u_{\nu})^{\nu-2}w_{\nu-2} (2u_{\nu})^{\nu-1}w_{\nu-1})^{2}}(\alpha+\beta)^{O((2u_{\nu})^{\nu}w_{\nu})}\\
&((2u_{\nu})^{\nu}w_{\nu}(\alpha+\beta))^{O(r(2u_{\nu})^{\nu-2}w_{\nu-2}(2u_{\nu})^{\nu-1}w_{\nu-1})}\\
\leq &2^{O(r^{2}(2u_{\nu})^{4\nu}w_{\nu-2}^{2}w_{\nu-1}^{2})}(\alpha+\beta)^{O((2u_{\nu})^{\nu}w_{\nu})}
((2u_{\nu})^{\nu}w_{\nu}(\alpha+\beta))^{O(r(2u_{\nu})^{2\nu}w_{\nu-2}w_{\nu-1})}.
\end{align*}
Now we have
\begin{align*}
b_{q_{0}}(X)\leq &(2^{\nu^{2}}u_{\nu}^{\nu}sw_{\nu-1})^{O((2u_{\nu})^{\nu}w_{\nu}) } 2^{O(r^{2}(2u_{\nu})^{4\nu}w_{\nu-2}^{2}w_{\nu-1}^{2})}(\alpha + \beta)^{O((2u_{\nu})^{\nu}w_{\nu})}\\
&((2u_{\nu})^{\nu}w_{\nu}(\alpha+\beta))^{O(r(2u_{\nu})^{2\nu}w_{\nu-2}w_{\nu-1})}\\
\leq &s^{O((2u_{\nu})^{\nu}w_{\nu})}(w_{\nu-1}^{O((2u_{\nu})^{\nu}w_{\nu})})(2^{\nu^{2}}u_{\nu}^{\nu})^{O((2u_{\nu})^{\nu}w_{\nu}) } 2^{O(r^{2}(2u_{\nu})^{4\nu}w_{\nu-2}^{2}w_{\nu-1}^{2})}(\alpha + \beta)^{O((2u_{\nu})^{\nu}w_{\nu})}\\
&((2u_{\nu})^{\nu}w_{\nu}(\alpha+\beta))^{O(r(2u_{\nu})^{2\nu}w_{\nu-2}w_{\nu-1})}\\
\leq &s^{O((2u_{\nu})^{\nu}w_{\nu})}(2^{\nu^{2}}u_{\nu}^{\nu})^{O((2u_{\nu})^{\nu}w_{\nu}) } 2^{O(r^{2}(2u_{\nu})^{4\nu}w_{\nu-2}^{2}w_{\nu-1}^{2})}(w_{\nu-1}(\alpha + \beta))^{O((2u_{\nu})^{\nu}w_{\nu})}\\
&((2u_{\nu})^{\nu}w_{\nu}(\alpha+\beta))^{O(r(2u_{\nu})^{2\nu}w_{\nu-2}w_{\nu-1})}\\
\leq &s^{O((2u_{\nu})^{\nu}w_{\nu})}(2^{\nu^{2}}u_{\nu}^{\nu})^{O((2u_{\nu})^{\nu}w_{\nu}) } 2^{O(r^{2}(2u_{\nu})^{4\nu}w_{\nu-2}^{2}w_{\nu-1}^{2})}
((2u_{\nu})^{\nu})^{O(r(2u_{\nu})^{2\nu}w_{\nu-2}w_{\nu-1})}\\
&(w_{\nu}(\alpha+\beta))^{O((2u_{\nu})^{\nu}w_{\nu}+r(2u_{\nu})^{2\nu}w_{\nu-2}w_{\nu-1})}\\
\leq &s^{O((2u_{\nu})^{\nu}w_{\nu})}(2^{\nu^{2}}u_{\nu}^{\nu})^{O((2u_{\nu})^{\nu}w_{\nu}) } 2^{O(r^{2}(2u_{\nu})^{4\nu}w_{\nu-2}^{2}w_{\nu-1}^{2})}
((2u_{\nu})^{\nu})^{O(r(2u_{\nu})^{2\nu}w_{\nu-2}w_{\nu-1})}\\
&(w_{\nu}(\alpha+\beta))^{O((2u_{\nu})^{\nu}w_{\nu}+r(2u_{\nu})^{2\nu}w_{\nu-2}w_{\nu-1})}\\
\leq &s^{O((2u_{\nu})^{\nu}w_{\nu})}(2^{\nu^{2}}u_{\nu}^{\nu})^{O((2u_{\nu})^{\nu}w_{\nu}) } (2u_{\nu})^{O(\nu(r(2u_{\nu})^{2\nu}w_{\nu-2}w_{\nu-1})+r^{2}(2u_{\nu})^{4\nu}w_{\nu-2}^{2}w_{\nu-1}^{2})}\\
&(w_{\nu}(\alpha+\beta))^{O((2u_{\nu})^{\nu}w_{\nu}+r(2u_{\nu})^{2\nu}w_{\nu-2}w_{\nu-1})}\\
\leq &s^{O((2u_{\nu})^{\nu}w_{\nu})}(2^{\nu^{2}}u_{\nu}^{\nu})^{O((2u_{\nu})^{\nu}w_{\nu}) } (2u_{\nu})^{O(r^{2}(2u_{\nu})^{4\nu}w_{\nu-2}^{2}w_{\nu-1}^{2})}\\
&(w_{\nu}(\alpha+\beta))^{O((2u_{\nu})^{\nu}w_{\nu}+r(2u_{\nu})^{2\nu}w_{\nu-2}w_{\nu-1})}\\
\leq &s^{O((2u_{\nu})^{\nu}w_{\nu})}(2^{\nu}u_{\nu})^{O(\nu(2u_{\nu})^{\nu}w_{\nu}+r^{2}(2u_{\nu})^{4\nu}w_{\nu-2}^{2}w_{\nu-1}^{2}) } \\
&(w_{\nu}(\alpha+\beta))^{O((2u_{\nu})^{\nu}w_{\nu}+r(2u_{\nu})^{2\nu}w_{\nu-2}w_{\nu-1})}.
\end{align*}
Compare this with the bound of
\begin{align*}
s^{O(2^{\nu}w_{\nu})}2^{O(\nu2^{\nu}w_{\nu}+r^{2}2^{4\nu}w_{\nu-2}^{2}w_{\nu-1}^{2}) }
(w_{\nu}(\alpha+\beta))^{O(2^{\nu}w_{\nu}+r2^{2\nu}w_{\nu-2}w_{\nu-1})}
\end{align*}
given in section 8.2 of \cite{BNSASPSets} for the case where $X_{\nu}$ is open or closed.
\end{eg}

%%%%%%%%%%%%%%%%%%%%%%%%%%%%%%%%%%%%%%%%%%%%%%%%%%%%%%%%%%%%%%%%%%%%%%%%%%%%%%%%%%%%%%%%%%%%%%%%%%%%%%%%%%%%%%%%%%

\addcontentsline{toc}{chapter}{Bibliography}

\clearpage
\hspace{1cm}
\clearpage

\bibliography{sources}{}

\begin{thebibliography}{10}

\bibitem{BettiSA}
Saugata Basu.
\newblock On bounding the {B}etti numbers and computing the {E}uler
  characteristic of semi-algebraic sets.
\newblock {\em Discrete and Computational Geometry}, 22:1--18, 1999.

\bibitem{AlgInRAG}
Saugata Basu, Richard Pollack, and Marie-Fran\c{c}oise Roy.
\newblock {\em Algorithms in Real Algebraic Geometry}.
\newblock Algorithms and Computations in Mathematics. Springer, Berlin, 2003.

\bibitem{SignCon}
Saugata Basu, Richard Pollack, and Marie-Fran\c{c}oise Roy.
\newblock On the {B}etti numbers of sign conditions.
\newblock {\em Proceedings of the American Mathematical Society},
  133(4):965--974, 2005.

\bibitem{RASASets}
Riccardo Benedetti and Jean-Jacques Risler.
\newblock {\em Real Algebraic and Semi-Algebraic Sets}.
\newblock Hermann, Editurs Des Sciences Et Des Arts, Paris, 1990.

\bibitem{SpecSeq}
Timothy Chow.
\newblock You could have invented spectral sequences.
\newblock {\em Notices of the American Mathematical Society}, 53:15--19, 2006.

\bibitem{OMin}
Michel Coste.
\newblock {\em An Introduction to O-Minimal Geometry}.
\newblock RAAG publications, November 1999.

\bibitem{SAG}
Michel Coste.
\newblock {\em An Introduction to Semi-Algebraic Geometry}.
\newblock RAAG publications, October 2002.

\bibitem{quantElim}
Andrei Gabrielov.
\newblock Counterexamples to quantifier elimination for fewnomial and
  exponential expressions.
\newblock {\em Moscow Mathematical Journal}, 7:453--460, 2007.

\bibitem{PffNoe}
Andrei Gabrielov and Nicolai Vorobjov.
\newblock Complexity of computations with {P}faffian and {N}oetherian
  functions.
\newblock {\em Normal Forms, Bifurcations and Finiteness Problems in
  Differential Equations}, pages 211--250, 2004.

\bibitem{BNSASetsQFForm}
Andrei Gabrielov and Nicolai Vorobjov.
\newblock {B}etti numbers of semialgebraic sets defined by quantifier-free
  formulae.
\newblock {\em Discrete and Computational Geometry}, 33:395--401, 2005.

\bibitem{ApproxDefSetCompFam}
Andrei Gabrielov and Nicolai Vorobjov.
\newblock {Approximation of definable sets by compact families, and upper
  bounds on homotopy and homology}.
\newblock {\em J. London Math. Soc.}, 80(1):35--54, 2009.

\bibitem{BNSASPSets}
Andrei Gabrielov, Nicolai Vorobjov, and Thierry Zell.
\newblock {B}etti numbers of semialgebraic and sub-{P}faffian sets.
\newblock {\em Journal of the London Mathematical Society}, 69:27--43, 2004.

\bibitem{AT}
Allen Hatcher.
\newblock {\em Algebraic Topology}.
\newblock Cambridge University Press, 2002.

\bibitem{DiffTop}
Morris Hirsch.
\newblock {\em Differential Topology}.
\newblock Graduate Texts in Mathematics. Springer-Verlag, New York, 1976.

\bibitem{Khov1}
Askold Khovanskii.
\newblock On a class system of transcendental equations.
\newblock {\em Soviet Math. Dokl.}, (22):762--765, 1980.

\bibitem{Khov2}
Askold Khovanskii.
\newblock Fewnomials.
\newblock {\em Transl. Math. Monogr.}, (88), 1991.
\newblock Amer. Math. Soc., Providence, RI.

\bibitem{Milnor}
John Milnor.
\newblock On the {B}etti numbers of real varieties.
\newblock {\em Proc. AMS}, 15:275--280, 1964.

\bibitem{OlePetr}
Olga Oleinik and Ivan Petrovsky.
\newblock On the topology of real algebraic surfaces.
\newblock {\em Izvestiia Akademii Nauk Sssr Seriia Matematicheskaia},
  13:389--402, 1949.

\bibitem{Thom}
Rene Thom.
\newblock Sur l'homologie des variétés algebriques réelles.
\newblock In {\em Differential and Combinatorial Topology}, pages 255--265.
  Princeton University Press, Princeton, 1965.

\bibitem{tameTopOMin}
Lou van~den Dries.
\newblock {\em Tame topoplogy and O-Minimal Structures}.
\newblock Cambridge University Press, 1998.

\bibitem{pfaffOMin}
Alex Wilkie.
\newblock A theorem of the complement and some new o-minimal structures.
\newblock {\em Selecta Mathematica, New Series}, 5:397--421, 1999.
\newblock 10.1007/s000290050052.

\bibitem{Zell}
Thierry Zell.
\newblock {B}etti numbers of semi-{P}faffian sets.
\newblock {\em Journal of Pure and Applied Algebra}, pages 323--338, 1999.

\end{thebibliography}
\bibliographystyle{plain}

\end{document}